\documentclass[12pt,leqno]{amsart}
\selectfont

\usepackage{qtree}
\usepackage{wrapfig}
\usepackage{epic,eepic,ecltree}

\usepackage{amssymb,color, euscript, enumerate}
\usepackage{amsthm}
\usepackage{amsmath}
\usepackage{amscd}
\usepackage{txfonts}
\usepackage{comment}
\usepackage{bm}
\usepackage{amscd}

\numberwithin{equation}{section}

\newcommand{\F}{F_{0,1,\infty}}
\newcommand{\A}{\Aut (0,1,\infty)}
\newcommand{\CP}{\C P^1}
\newcommand{\CPm}{\C P^1 \setminus \{0,1,\infty\}}

\newcommand{\Z}{\mathbb{Z}}
\newcommand{\R}{\mathbb{R}}
\newcommand{\C}{\mathbb{C}}
\newcommand{\1}{\bm{1}}
\newcommand{\z}{{\bar{z}}}
\newcommand{\zij}{{z_{ij}}}
\newcommand{\zijd}{{\bar{z}_{ij}}}
\newcommand{\h}{{\bar{h}}}
\newcommand{\p}{{\bar{p}}}
\newcommand{\w}{{\bar{w}}}
\newcommand{\y}{{\bar{y}}}
\newcommand{\x}{{\bar{x}}}
\newcommand{\uz}{\underline{z}}
\newcommand{\ux}{\underline{x}}
\newcommand{\s}{{\bar{s}}}
\newcommand{\n}{{\bar{n}}}

\newcommand{\al}{\alpha}
\newcommand{\be}{\beta}

\newcommand{\om}{\omega}
\newcommand{\omb}{{\bar{\omega}}}
\newcommand{\si}{\sigma}

\newcommand{\Ld}{{\overline{L}}}
\newcommand{\D}{{\bar{D}}}

\newcommand{\Cor}{{\mathrm{Cor}}}
\newcommand{\GCor}{{\mathrm{GCor}}}

\newcommand{\Aut}{\mathrm{Aut}\,}
\newcommand{\Ch}{\mathrm{Chart}(0,1,\infty)}

\newcommand{\tw}{{{I\hspace{-.1em}I}_{1,1}}}
\newcommand{\twN}{{{I\hspace{-.1em}I}_{N,N}}}

\newcommand{\End}{\mathrm{End}}
\newcommand{\fora}{\text{ for any }}
\newcommand{\tand}{\text{ and }}

\newcommand{\HVect}{\underline{\text{H-Vect}}}

\newtheorem{thm}{Theorem}[section]
\newtheorem{lem}{Lemma}[section]
\newtheorem{prop}{Proposition}[section]
\newtheorem{cor}{Corollary}[section]
\newtheorem{rem}{Remark}[section]

\begin{document}
\title[Full vertex algebras]{Full vertex algebra and bootstrap
- consistency of four point functions in 2d CFT}
% Consistency, analysis, two dimensional conformal field theory on projective space
% 
%
%\author{YUTO MORIWAKI}
%\date{}
%\address[Y. Moriwaki]{Graduate School of Mathematical Science, The University of Tokyo, 3-8-1
%Komaba, Meguro-ku, Tokyo 153-8914, Japan}
%\email {moriwaki@ms.u-tokyo.ac.jp}%
%\maketitle

\begin{center}
{\LARGE \bf Full vertex algebra and bootstrap

- consistency of four point functions in 2d CFT
} \par \bigskip

\renewcommand*{\thefootnote}{\fnsymbol{footnote}}
{\normalsize
Yuto Moriwaki \footnote{email: \texttt{moriwaki@ms.u-tokyo.ac.jp}}
}
\par \bigskip
{\footnotesize Kavli Institute for the Physics and Mathematics of the Universe, Chiba, Japan}

\par \bigskip
\end{center}

\vspace*{8mm}

\noindent
\textbf{Abstract.}
% algebraic axiom of 2D CFT ? is needed.
% chiral part is vertex operator algebra.
% Huang-Kong give an axiom called full field algebra
%
% in physics, known as bootstrap hypothesis
% we give a consistency only by assuming the bootstrap.
%
In physics, it is believed that the consistency of two dimensional conformal field theory follows
from the bootstrap equation. %, which is a part of the consistencies.
In this paper, we introduce the notion of a full vertex algebra
by analyzing the bootstrap equation,
which is a ``real analytic'' generalization of a $\Z$-graded vertex algebra.
%We also introduce a full vertex operator algebra, a dual module and QP-generated is generalized 
%to a full vertex algebra. 
We also give a mathematical formulation of the consistency of four point correlation functions in two dimensional conformal field theory
and prove it for a full vertex algebra with additional assumptions on the conformal symmetry.
In particular, we show that the bootstrap equation
together with the conformal symmetry implies the consistency of
 four point correlation functions.
As an application, a deformable family of full vertex algebras parametrized by the Grassmanian is constructed,
which appears in the toroidal compactification of string theory.
This give us examples satisfying the above assumptions.
\vspace*{8mm}

\section*{Introduction}
%
% What is QFT?
% Defined by the consistency of correlation funcitons
% CFT by huang, motivation for deformation and
% non-rational CFT.
% Huang start from (regular) vertex algebra
% we start from CFT.

Quantum field theory has a long history in physics (see e.g. \cite{S}) and has attracted mathematicians in the past several decades due to its unexpected mathematical consequences or conjectures, for example \cite{W, Ma}.
One aim of quantum field theory is to calculate {\it $n$ point correlation functions}, that is, the vacuum expectation value of an interaction of $n$ particles.
An interaction of $n$ particles decomposes into 
subsequent interactions of three particles.
Thus, an $n$ point correlation function can be expressed in terms of three point correlation functions, depending on a choice of decompositions.
A quantum field theory requires that the resulting $n$ point correlation functions are \emph{independent} 
of the choice of decompositions.
This principle is known as the {\it consistency of quantum field theory}.
Although it is known to be difficult to construct mathematically rigorous quantum field theories, surprisingly many such examples, especially \emph{conformal field theories}, (i.e., quantum field theories with ``conformal symmetry''),  have been found in two dimension, see \cite{FMS}.

In (not necessarily two dimensional) conformal field theories, it is believed in physics, that the whole consistency of $n$ point correlation functions follows from 
the \emph{bootstrap equations (or hypothesis)},
which are distinguished consistencies of four point correlation functions \cite{FGG,P2}.
This hypothesis was used successively by Belavin, Polyakov and Zamolodchikov in \cite{BPZ} where the modern study of two dimensional conformal field theories was initiated.

In two dimensional conformal field theory, fields are operator valued real analytic functions.
It is noteworthy that the subalgebra consisting of holomorphic fields satisfies a purely algebraic axiom, which was introduced by Borcherds \cite{B}, see also \cite{G}.
It is called a vertex algebra or a vertex operator algebra \cite{FLM}
and has been studied intensively by many authors, see e.g., \cite{LL,FHL,FB,K}.
We note that in the language of a vertex algebra,
three point correlation functions correspond to
the vertex operator and the bootstrap equations 
essentially correspond to the Borcherds identity, which is an axiom of vertex algebras. 
In contrast, the axiom of the (non-holomorphic) whole algebra of a conformal field theory needs analytic properties and seems impossible to describe in an algebraic way.

Moore and Seiberg constructed a conformal field theory as an extension of holomorphic and anti-holomorphic vertex operator algebras by their modules \cite{MS1,MS2}.
The bootstrap equations in this case are translated as a monodromy invariant property of the four point correlation functions. In the physics literature, this property was reformulated later by Fuchs, Runkel and Schweigert in \cite{FRS}, which says that the algebra describing the conformal field theory is a Frobenius algebra object in a braided tensor category constructed from holomorphic and anti-holomorphic vertex operator algebras.

A mathematical approach in this direction is due to Huang and Kong \cite{HK} based on the  abstract representation theory of vertex algebras developed by Huang and Lepowsky in a series of papers \cite{HL1,HL2,HL3,H1,H2}. One of the prominent results is obtained by Huang \cite{H3,H4}, which states that the representation category of a regular vertex operator algebra (of strong CFT type) inherits a modular tensor category structure. Note that a vertex operator algebra is called regular if every (weak) module is completely reducible.
%There is a nice class of vertex operator algebra, called rational $C_2$ cofinite vertex operator algebra.

Then, Huang and Kong \cite{HK} introduced a notion of full field algebras, which is a mathematical axiomatization of the algebras describing two dimensional conformal field theories. Their definition is based upon a part of the consistencies of $n$ point correlation functions. In \cite{HK}, they constructed conformal field theories, called \emph{diagonal theories} in physics, as finite module extensions of the tensor product of regular vertex operator algebras in two variables ``$(z,\z)$''.

%There are three purposes of this paper

In this paper, we study an axiom of conformal field theory on $\C P^1$ and introduce a notion of \emph{full vertex algebras} starting with the bootstrap equations, which are expected to be sufficient to derive the whole consistency of the theory. 

The definition of full vertex algebras is independent of the theory of vertex algebras, which is an essential building block of the approach by Huang and Kong. 
Rather, the existence of vertex subalgebras as holomorphic and anti-holomorphic part of the algebra is naturally derived from our definition as expected in physics, see below for details.

The notion of full vertex algebras essentially appears in \cite{HK}. 
There they showed that as long as a full field algebra is an extension of a tensor product of regular vertex algebras, the notions of full field algebras and full vertex algebras are equivalent \cite[Theorem 2.11]{HK}. 

Before going into the details, let us summarize the key ingredients of the paper:
\begin{itemize}
\item
to introduce a reasonable space $\Cor_4$ of real analytic functions serving as the four point correlation functions (Section \ref{intro_function space}),
\item
to define and analyze parenthesized correlation functions $S_A$ (Section \ref{intro_parenthesized}),
\item
to formulate and prove the consistency of four point functions in the language of full vertex algebras (Section \ref{intro_def}).
\end{itemize}
As a byproduct, we give an example of a deformable family of full vertex algebras corresponding to the irrational conformal field theory appearing in the toroidal compactification of string theory \cite{P}. Mathematically, this is a generalization of a lattice vertex algebra \cite{B,FLM}. We remark that a deformation of a theory is one of the most important ingredient of quantum field theory and conformal field theory \cite{IZ}. This kind of deformations will be studied in more general setting in our next paper \cite{M2}. 
The author hopes that our approach will be beneficial in the future investigation on the conformal field theories in higher dimensions. 

\subsection{Space of four point correlation function}\label{intro_function space}
%As mentioned above, one aim of quantum field theory is to calculate the correlation functions
%and the consistency of quantum field theory can be formulated by the correlation functions,
%especially some four point function.
In order to formulate the consistency of conformal field theory on the projective space $\C P^1$,
we need to fix a space of functions ``consisting of four point correlation functions''.
%define a space of .
%For $n \in \Z_{\geq 0}$,
Let $X_4=\{(z_1, \dots, z_4)\in (\C P^1)^4\;|\; z_i\neq z_j \fora i\neq j \}$
be a space of ordered configurations of four points in $\C P^1$.
Then, a four point correlation function is, roughly, a $\C$-valued real analytic function on $X_4$ with possible singularities along the diagonals,
$\{(z_1,\dots,z_4)\in (\C P^1)^4\;|\;z_i=z_j\}$ for $1\leq i<j\leq 4$.
Recall that the automorphism group of $\C P^1$ is $\mathrm{PSL}_2 \C$, whose action is called the linear fractional transformation and which is the global conformal symmetry in two dimension.
Since the real analytic function 
$$\xi :X_4 \rightarrow \C P^1 \setminus \{0,1,\infty\},\;
(z_1,z_2,z_3,z_4)\mapsto \frac{(z_1-z_2)(z_3-z_4)}{(z_1-z_3)(z_2-z_4)}$$
is invariant under the diagonal action of $\mathrm{PSL}_2 \C$ on $X_4 \subset (\C P^1)^4$,
it gives a homeomorphism
from the coset space $\mathrm{PSL}_2 \C \backslash X_4$ to $\C P^1\setminus \{0,1,\infty\}$.
Any four point function of quasi-primary states in 2d conformal field theory possesses the global conformal symmetry.
Thus, by the homeomorphism, a four point function can be regarded as a real analytic function on $\C P^1\setminus \{0,1,\infty\}$ with possible singularities at $\{0,1,\infty \}$, which we call {\it conformal singularities.}
% Nakatsuka comment
% |z|^r is not real analytic function
We remark that $|z|^r$ is not a real analytic function at $0$,
where $|z|=z\z$, the square of the absolute value, and $r\in \R$.
A function with a conformal singularity at $0$ has an expansion
\begin{align}
\sum_{r,s \in \R} a_{r,s}z^r \z^s, \label{eq_CS}
\end{align}
where $a_{r,s}\in \C$ and $a_{r,s}=0$ if $r-s \notin \Z$.
This series is absolutely convergent in an annulus $0<|z|<R$ (for the precise definition, see Section \ref{sec_four_point}).
By the assumption, $a_{r,s}z^r\z^s=a_{r,s}z^{r-s}|z|^s$ is a single-valued function around $0$.
Denote by $F_{0,1,\infty}$ the space of real analytic functions on $\C P^1 \setminus \{0,1,\infty \}$
with possible conformal singularities at $\{0,1,\infty\}$.
The space $\F$ contains a monodromy invariant linear combination of solutions of 
differential equations with possible regular singularities at $\{0,1,\infty \}$.
% general 4-point function and expansion
%
%Let $M$ be a one dimensional complex manifold.
%Roughly speaking, a 
%A $\C$-valued real analytic function $f(p)$ defined around $\{p\in \C\;|\; 0 < |p| < 1 \}$
%is said to have a conformal singularity at $0$
%if there exists formal power series
%such that the series
\begin{comment}
is absolutely convergent to $f(p)$ in some neighborhood of $0$.
We assume that the above some is countable and $a_{r,s} =0$ for sufficiently small $r$ or $s$, 
$a_{r,s}=0$ unless $r-s \in \Z$ (see Section \ref{}).
More generally, the conformal singularity is defined at any point by using a complex analytic chart of $\C P^1$.
\end{comment}
For example, a combination of hypergeometric functions,
\begin{align}
f_{\mathrm{Ising}}(z)=|1-\sqrt{1-z}|^{1/2}+|1+\sqrt{1-z}|^{1/2}, \label{Ising}
\end{align}
has conformal singularities at $\{0,1,\infty\}$,
which appears as a four point function of the two dimensional Ising model.
%We remark that the exponent of the series (CS) is independent of the choice of the chart.
%An exponent of a function $f \in \F$ is the maximal number of the exponents of the series expansion of $f$ at $0,1,\infty$.
The expansion of $f_{\mathrm{Ising}}$ at $z=0$ is
\begin{align*}
2+|z|^{1/2}/2-z/4-\z/4+|z|^{1/2}(z+\z)/16+ z\z/32-5z^2/64-5\z^2/64+\dots.%,\tag{z} \\
%&2+|z|^{1/2}/2-z/4-\z/4+|z|^{1/2}(z+\z)/16+ z\z/32-5z^2/64-5\z^2/64+\dots,\tag{1-z}\\
%.\tag{1/z}
\end{align*}
%Since $f_{\mathrm{Ising}}(p)$ satisfies $f_{\mathrm{Ising}}(p)=f_{\mathrm{Ising}}(1-p)=
%(z\z)^{1/4}f_{\mathrm{Ising}}(1/p),$
%the exponent of $f_{\mathrm{Ising}}$ is $2$ (see also Section \ref{example_ising}).
We introduce a space $\Cor_4$ spanned by real analytic functions on $X_4$ of the form:
%A space of four point correlation functions, denoted by $\Cor_4$, is spanned by the real analytic functions on $X_4$ of
%the following form:
\begin{align}
 \Pi_{1\leq i<j \leq 4} (z_i-z_j)^{\al_{ij}}(\bar{z}_i-\bar{z}_j)^{\be_{ij}} f\circ \xi(z_1,z_2,z_3,z_4), \label{eq_co}
\end{align}
where $f \in \F$ and $\al_{ij}, \be_{ij} \in \R$ such that $\al_{ij}-\be_{ij}\in \Z$ for any $1 \leq i<j\leq 4$.
Here, we allow these functions to have singularities
around $\{z_i=\infty\}_{i=1,2,3,4}$.
By the condition $\al_{ij}-\be_{ij}\in \Z$, we have
$$ (z_i-z_j)^{\al_{ij}}(\bar{z}_i-\bar{z}_j)^{\be_{ij}}=|z_i-z_j|^{\al_{ij}}(\bar{z}_i-\bar{z}_j)^{\be_{ij}-\al_{ij}},$$
which implies that $ \Pi_{1\leq i<j \leq 4} (z_i-z_j)^{\al_{ij}}(\bar{z}_i-\bar{z}_j)^{\be_{ij}}$ is a single valued real analytic function on $X_4$.
Physically, a four point correlation function of quasi-primary states in compact conformal field theory in two dimensions
is of the form (\ref{eq_co}). Thus, we call $\Cor_4$ a space of four point correlation functions.
We remark that non-compact conformal field theories, e.g., the Liouville field theory and 
non-compact WZW conformal field theory are excluded from our consideration,
where we need to consider some measure on the space of irreducible representations and integral over there.
\begin{comment}
We remark that for a rational conformal field theory (or more generally quasi-rational conformal field theory),
$f\in \F$ appeared in the four point correlation function is of the form
\begin{align*}
\sum_{i=1}^N \sum_{n,m \in \Z_{\geq 0}} a_{n,m}^i z^n \z^m |z|^{r_i}, \tag{Expansion} \label{eq_expansion}
\end{align*}
where $N$ is the number of intermediate states, which is mathematically the number of the irreducible components of the tensor product of modules.
\end{comment}
%The exponent of the series in (\ref{eq_CS}) is almost equal to the number of intermediate states in
%the s,t,u-channel reaction, which is mathematically the number of the irreducible components of the tensor product of modules.
%Thus, in (\ref{eq_CS}) and (\ref{eq_four}), we assume that the number of 
%the exponent is always finite.
%, we assume that the singular term $(p\p)^{r_i}$ is finite $(i=1,\dots,l)$.
%This assumption is satisfied for 2d conformal field theory whose
%fusion rule of conformal blocks is finite %(see Section \ref{sec_correlation}, especially Lemma \ref{discrete}),
%e.g., rational conformal field theory. 
%We remark that there is a non-rational conformal field theory satisfying this assumption,
%e.g., toroidal compactification of string theory and its orbifold (see Section \ref{sec_lattice})
%and also
%We remark that

\subsection{Consistency and parenthesized correlation functions}\label{intro_parenthesized}
Roughly speaking, a conformal field theory on $\C P^1$, considered in this paper, consists of 
an $\R^2$-graded $\C$-vector space $F=\bigoplus_{h,\h\in \R}F_{h,\h}$ with a distinguished vector $\1 \in F$,
 a symmetric bilinear form $(-,-):F\times F \rightarrow \C$,
which is normalized as $(\1,\1)=1$, and a linear map 
$$Y(-,\uz):F \rightarrow \End (F)[[z^\pm,\z^\pm,|z|^\R]],\; a\mapsto Y(a,\uz)=\sum_{r,s \in \R}a(r,s)z^{-r-1}\z^{-s-1},$$
where $\uz$ means the pair of the formal variables $z$ and $\z$.
Physically,  a vector in $F_{h,\h}$ is a state with the spin $h-\h$ and the energy $h+\h$, $\1$ corresponds
 to the vacuum state and $Y(a,\uz)$ is the field located at $z$ associated with a state $a\in F$,
called a {\it vertex operator}. It is convenient to write the vertex operator $Y(a,\uz)$ as $a(z)$.
For example, $(a_1(x_1)a_2)(x_2)$ means $Y(Y(a_1,\ux_1)a_2,\ux_2)$
and $a_1(x_1)a_2(x_2)$ means $Y(a_1,\ux_1)Y(a_2,\ux_2)$.
%We remark that to give a vertex operator is equivalent to
%determine the three point correlation functions.
A four point correlation function is a linear map $$S:F^{\otimes 4}\rightarrow \Cor_4, \;(a_1,a_2,a_3,a_4)\mapsto S(a_1,a_2,a_3,a_4)(z_1,z_2,z_3,z_4)$$
which physically calculates the vacuum expectation value of states $a_1,a_2,a_3,a_4 \in F$ at $(z_1,z_2,z_3,z_4) \in X_4$.
A four point correlation can be calculated in many ways by using the vertex operator, e.g.,
the formal power series of the variables $(z_1,\z_1,\dots,z_4,\z_4)$,
\begin{align*}
(\1,a_1(z_1)a_2(z_2)a_3(z_3)a_4(z_4)\1),
\end{align*}
is absolutely convergent to $S(a_1,a_2,a_3,a_4)(z_1,z_2,z_3,z_4) \in \Cor_4$ in the region $\{|z_1|>>|z_2|>>|z_3|>>|z_4|\}$.
Another choice of compositions of vertex operators is
\begin{align}
\Bigl(\1, \biggl(\Bigl(a_4(x_4)a_2\Bigr)(x_2)a_1\biggr)(x_1)(a_3(x_3)\1)\Bigr). \label{eq_example1}
\end{align}
Such compositions are described by parenthesized products of symbols $\{1,2,3,4,\star \}$, in this case, $((42)1)(3\star)$.
To be more precise, let $n \in \Z_{>0}$
and $Q_n$ be the set of parenthesized products of $n+1$ elements $1,2,\dots,n,\star$
with $\star$ at right most, e.g., $(((31)6)(24))(5\star) \in Q_6$.
We can naturally associate a formal power series to each element in $Q_n$,
which we call a {\it parenthesized $n$ point correlation function}.
To give a simple example, an element $(((31)6)(24))(5\star) \in Q_6$ defines the following %binary tree:
parenthesized $6$ point correlation function:
\begin{align*}
S_{(((31)6)(24))(5\star)}(a_1,\dots,a_6;x_1,\dots,x_6):=
\Bigl(\1, \Biggl[\biggl( \Bigl(a_3(x_3)a_1\Bigr)(x_1)a_6\biggr)(x_6) a_2(x_2)a_4\Biggr](x_4) a_5(x_5)\1 \Bigr),
\end{align*}
where $a_1,\dots,a_6 \in F$.
For $A\in Q_n$, we combinatorially give a space $T_A$ of formal variables $x_1,\x_1,\dots,x_n,\x_n$
which contains the image of all the parenthesized $n$ point functions.
%The image of the map, a space of , is denoted by $T_A$.
Thus, the parenthesized correlation function defines the map,
$S_A:F^{\otimes n} \rightarrow T_A$ for each $A\in Q_n$.

The consistency of quantum field theory says that
all parenthesized four point correlation functions associated with all the elements in $Q_4$ and $a_1,a_2,a_3,a_4\in F$
 (or more generally $n$ point correlation functions)
are expansions of the same function $S(a_1,a_2,a_3,a_4)(z_1,z_2,z_3,z_4)$ in different regions 
after taking the change of variables.
For example, the parenthesized correlation function $S_{((42)1)(3\star)}(a_1,a_2,a_3,a_4;x_1,x_2,x_3,x_4)$ in (\ref{eq_example1})
is the expansion of $S(a_1,a_2,a_3,a_4)(z_1,z_2,z_3,z_4)$ around $\{|x_1|>>|x_2|>>|x_4|,|x_3|\}$,
where $(x_1,x_2,x_3,x_4)=(z_1,z_2-z_1,z_3,z_4-z_2)$.
We remark that for a vertex algebra, 
the composition of vertex operators
$Y(Y(a_1,\ux_0)a_2,\ux_2)$ is an analytic continuation of 
$Y(a_1,\ux_1)Y(a_2,\ux_2)$ after taking the change of variables $x_0=x_1-x_2$.
Physically, the change of variables comes from the fact that
any state in $F$ is located at $0$, i.e., $\lim_{z \to 0}Y(a,\uz)\1=a$ (the state-field correspondence).
\begin{comment}
We combinatorially give the rule for the change of variables for any parenthesized correlation function
associated with $A \in Q_n$ (see Appendix).
Especially, we explicitly give a power series expansion of a function in $\Cor_4$
in some open domain in $X_4$ for each $A \in Q_4$,
that is, $$e_A: \Cor_4 \rightarrow T_A.$$
\end{comment}
In Appendix, for each $A\in Q_4$,
we give a combinatorial rule for the change of variables and the region by using trees
and in Section \ref{sec_expansion}
we define the power series expansion of a function in $\Cor_4$, 
which gives $e_A: \Cor_4 \rightarrow T_A.$

Then, our formulation of the consistency of a four point function is stated as follows:
For any $A\in Q_4$, 
the composition of the four point function $S:F^{\otimes 4}\rightarrow \Cor_4$ and
the expansion map $e_A:\Cor_4 \rightarrow T_A$ is equal to the parenthesized correlation function
$S_A:F^{\otimes 4}\rightarrow T_A$, that is,
\begin{align}
S_A = e_A \circ S. \label{eq_consistency}
\end{align}

In physics, the bootstrap hypothesis says that all consistencies follows from the consistency for special parenthesized correlation functions, $S_{(21)(34)}$ (s-channel), $S_{(41)(23)}$ (t-channel) and $S_{(31)(24)}$ (u-channel).
The purpose of this paper is to define a full vertex algebra by slightly modifying the bootstrap hypothesis
as a generalization of a $\Z$-graded vertex algebra
 and to prove the consistency as a consequence of the axiom of a full vertex algebra.
In the theory of vertex algebras,
it is commonly used the limit of a four point correlation function as $(z_1,z_4)\rightarrow (\infty,0)$,
which we call a {\it generalized two point function}.
Before explaining the definition of a full vertex algebra,
we briefly describe the limit $(z_1,z_4)\rightarrow (\infty,0)$ in relation to the consistency.
%use this generalized two point function in the definition of full vertex algebra.
%To be more precise,
Let $a_1,a_2,a_3,a_4 \in F$ with $a_1 \in F_{h_1,\h_1}$.
Then, the limit $(z_1,z_4)\rightarrow (\infty,0)$ of $z_1^{2h_1}\z_1^{2\h_1}S(a_1,a_2,a_3,a_4)$
exists and is a linear combination of
%By multiplying by $z_1^{-\al_{12}-\al_{13}-\al_{14}}\z_1^{-\be_{12}-\be_{13}-\be_{14}}$,
%the limit of a function in \ref{eq_four} as $(z_1,z_4)\rightarrow (\infty,0)$ exists
%and equal to
\begin{align}
(z_2-z_3)^{\al_{23}}z_2^{\al_{24}}z_3^{\al_{34}}
(\z_2-\z_3)^{\be_{23}}\z_2^{\be_{24}}\z_3^{\be_{34}}f(z_3/z_2), \label{eq_gco},
\end{align}
which is a real analytic function on 
$Y_2=\{(z_2,z_3)\in \C^2\;|\; z_2\neq z_3, z_2\neq 0,z_3 \neq 0 \}$.
Denote by $\GCor_2$ the space of real analytic function spanned by (\ref{eq_gco}).

Since $(\infty,z_2,z_3,0) \in X_4$ with $|z_2|>|z_3|$ is in the convergent region of
$S_{1(2(3(4\star)))}(a_1,a_2,a_3,a_4)$, which is $\{|z_1|>|z_2|>|z_3|>|z_4|\}$,
the limit of $(-1)^{h_1-\h_1}z_1^{2h_1}\z_1^{2\h_1}S(a_1,a_2,a_3,a_4)$ as $(z_1,z_4)\rightarrow (\infty,0)$ with $|z_2|>|z_3|$ is 
equal to the formal limit,
$$
\lim_{(z_1,z_4)\to (\infty,0)} (-1)^{h_1-\h_1}z_1^{2h_1}\z_1^{2\h_1}S_{1(2(3(4\star)))}(a_1,a_2,a_3,a_4)
=(a_1,Y(a_2,z_2)Y(a_3,z_3)a_4).
$$
The convergent region of $S_{1(3(2(4\star)))}(a_1,a_2,a_3,a_4)$ and
$S_{1((23)(4\star))}(a_1,a_2,a_3,a_4)$
are $\{|z_1|>|z_3|>|z_2|>|z_4|\}$ and $\{|z_1|>|z_3|>|z_2-z_3|, |z_4|\}$.
Thus, the limits of $(-1)^{h_1-\h_1}z_1^{2h_1}\z_1^{2\h_1}S(a_1,a_2,a_3,a_4)$ as $(z_1,z_4)\rightarrow (\infty,0)$ with $|z_3|>|z_2|$
and $|z_3|>|z_2-z_3|$
%and $(z_1,z_2)\rightarrow (\infty,0)$ with $|z_4|>|z_3|$
%are $(a_1,Y(a_3,z_3)Y(a_2,z_2)a_4)$ and $(a_1,Y(Y(a_2,z_{23})a_3,z_3)a_2)$, respectively,
are given by $(a_1,Y(a_3,z_3)Y(a_2,z_2)a_4)$ and $(a_1,Y(Y(a_2,z_{23})a_3,z_3)a_4)$, respectively,
where $z_{23}=z_2-z_3$.
Furthermore, in the regions, $\{|z_2|>|z_3|\}$, $\{|z_3|>|z_2|\}$ and $\{|z_3|>|z_2-z_3|\}$,
 $\frac{z_3}{z_2}$ is close to $0$, $\infty$ and $1$, respectively. Thus, the expansions of (\ref{eq_gco}) in these regions
are determined by the expansion of $f \in \F$ around $\{0,1,\infty\}$.

\begin{comment}
Since $(\infty,z_2,z_3,0) \in X_4$ with $|z_2|>|z_3|$ is in the convergent region of
$S_{1(2(3(4\star)))}(a_1,a_2,a_3,a_4)$,
the suitable limit of $S(a_1,a_2,a_3,a_4)\in \Cor_4$ as $(z_1,z_4)\rightarrow (\infty,0)$
with $|z_2|>|z_3|$ is $(a_1,Y(a_2,z_2)Y(a_3,z_3)a_4)$
and similarly the limits as $(z_1,z_4)\rightarrow (\infty,0)$ with $|z_3|>|z_2|$
%and $(z_1,z_4)\rightarrow (\infty,0)$ with $|z_3|>|z_2-z_3|$
and $(z_1,z_2)\rightarrow (\infty,0)$ with $|z_4|>|z_3|$
%are $(a_1,Y(a_3,z_3)Y(a_2,z_2)a_4)$ and $(a_1,Y(Y(a_2,z_{23})a_3,z_3)a_2)$, respectively,
are $(a_1,Y(a_3,z_3)Y(a_2,z_2)a_4)$ and $(a_1,Y(a_4,z_4)Y(a_3,z_3)a_2)$, respectively.
%where $z_{23}=z_2-z_3$.
Those limits are related with each other by $S_4$-symmetry.
More precisely, if we assume that 
then $(2,3)\cdot f(\xi)=f(\xi^{-1})$ and $(24)\cdot f(\xi)=f(1-\xi)$.
\end{comment}

\subsection{Definition of full vertex algebra and main result}\label{intro_def}
Let us describe the precise definition of a full vertex algebra.
A full vertex algebra is an $\R^2$-graded vector space $F=\bigoplus_{h,\h\in \R} F_{h,\h}$ with a distinguished vector $\1 \in F_{0,0}$ and a vertex operator 
$$Y(-,\uz):F\rightarrow \End\;F[[z,\z,|z|^\R]],\; a\mapsto Y(a,\uz)=
\sum_{r,s\in \R}a(r,s)z^{-r-1}\z^{-s-1}$$
satisfying the following conditions:
For any $a,b\in F$, there exists $N\in \R$ such that
$a(r,s)b=$ unless $r\leq N$ and $s\leq N$;
$F_{h,\h}=0$ unless $h-\h \in \Z$;
The vacuum vector $\1$ satisfies $Y(\1,\uz)=\mathrm{id}_F$
and $Y(a,\uz)\1\in F[[z,\z]]$ and $\lim_{z\to 0}Y(a,\uz)\1=a$ for any $a\in F$;
$F_{h,\h}(r,s)F_{h',\h'}\subset F_{h+h'-r-1,\h+\h'-s-1}$ for
any $h,h',\h,\h',r,s\in\R$;
For any $a,b,c \in F$ and $u \in F^\vee=\bigoplus_{h,\h\in \R} F_{h,\h}^*$,
there exists $\mu(z_1,z_2) \in \GCor_2$ such that
\begin{align}
u(Y(a,\uz_1)Y(b,\uz_2)c)&=\mu(z_1,z_2)|_{|z_1|>|z_2|},\nonumber \\
u(Y(Y(a,\uz_0)b,\uz_2)c)&=\mu(z_0+z_2,z_2)|_{|z_2|>|z_0|}, \label{eq_Borcherds} \\
u(Y(b,\uz_2)Y(a,\uz_1)c)&=\mu(z_1,z_2)|_{|z_2|>|z_1|},\nonumber
\end{align}
where $F_{h,\h}^*$ is the dual vector space
and $\mu(z_1,z_2)|_{|z_1|>|z_2|}$ is the expansion of $\mu$ in
$\{|z_1|>|z_2|\}$.
For a full vertex algebra $F$, we define linear maps $D, \D \in \End \,F$ by
$$Y(a,\uz)\1=a+Da z+\D a\z+\dots.$$ Then, one can show that $[D,Y(a,\uz)]=Y(Da,\uz)=d/dzY(a,\uz)$ and $[\D,Y(a,\uz)]=Y(\D a,\uz)=d/d\z Y(a,\uz)$ (Proposition \ref{translation}).
Thus, if $a,b \in \ker \D$, then the generalized correlation function (\ref{eq_Borcherds})
satisfies $d/d\z_1\mu=d/d\z_2 \mu=0$, that is,
$\mu$ is a holomorphic function on $Y_2$.
We prove that if $\mu \in \GCor_2$ is holomorphic,
then $\mu \in \C[z_1^\pm,z_2^\pm,(z_1-z_2)^\pm]$,
which implies that $\ker \D$ is a vertex algebra (Proposition \ref{vertex_algebra}).
This reflects the fact that
a function in $\F$ is holomorphic if and only if $f \in \C[z^\pm,(1-z)^{-1}]$.
%a holomorphic function on $\CPm$ with a conformal singularity at $\{0,1,\infty\}$ is in $
Thus, in the definition of a full vertex algebra, we replace the function space of a vertex algebra $\C[z^\pm,(1-z)^{-1}]$
with $\F$,
which implies that the notion of a full vertex algebra is a generalization of the notion of a $\Z$-graded vertex algebra.

We remark that, for a vertex algebra, the pole of $u(Y(a,z_1)Y(b,z_2)c)$ is
$(z_1-z_2)^{-k}$ for some $k \in \Z$,
which implies that $(z_1-z_2)^N [Y(a,z_1),Y(b,z_2)]=0$ for a sufficiently large $N \in \Z_{\geq 0}$.
However, for a full vertex algebra,
the pole of $u(Y(a,\uz_1)Y(b,\uz_2)c)$ is a linear sum of
$(z_1-z_2)^n |z_1-z_2|^r$ for $r \in \R$ with $n \in \Z$,
e.g., $1+|z_1-z_2|^{1/2}+\dots$. Hence, we could not simultaneously cancel all poles by multiplication by a polynomial.
The number of the index $r \in \R/ \Z$
appeared in $(z_1-z_2)^n |z_1-z_2|^r$ reflects the number of intermediate states,
which is mathematically the number of the irreducible components of the tensor product of modules.
In the case of a full vertex algebra, the function space $\F$ is much more complicated,
however, we can still describe all the consistencies of a four point function in term of the expansions of the same function $f \in \F$ around
$\{0,1,\infty \}$.
%In the case of a conformal field theory (or a full vertex algebra), compared with a vertex algebra,
%correlation functions are much more complicated, e.g., combinations of hypergeometric functions or solutions of KZ-equations.
%For holomorphic functions on $\CPm$, the residues of $\{0,1,\infty\}$ is related by
%the Cauchy integral formula,
%which gives an algebraic definition of a vertex algebra, Borcherds identity.

In order to prove the consistency, we need the global conformal symmetry.
A full vertex algebra with an energy-momentum tensor, a pair of holomorphic and anti-holomorphic conformal vectors,
which we call a full vertex operator algebra (see Section \ref{sec_VOA}).
A full vertex operator algebra admits an action of the Virasoro algebras, $\mathrm{Vir}\oplus \mathrm{Vir}$
and, in particular, a subalgebra, $\mathrm{sl}_2 \C \oplus \mathrm{sl}_2 \C=\bigoplus_{i=-1,0,1} L(i)\oplus \bigoplus_{j=-1,0,1}\Ld(j)$.
A vector $v \in F_{h,\h}$ is called a {\it quasi-primary} if $L(1)v=\Ld(1)v=0$
and a full vertex operator algebra is called QP-generated if 
$F$ is generated by quasi-primary vectors as an $\mathrm{sl}_2 \C \oplus \mathrm{sl}_2 \C$-module.
A notion of a dual module, introduced in \cite{FHL} for a vertex operator algebra, can be generalized to a full vertex operator algebra and its modules.
We give a criterion when a full vertex operator algebra is self-dual and QP-generated, which generalize a result of Li \cite{Li}.
The main result of this paper is that
for a QP-generated self-dual full vertex operator algebra, all the consistencies of four point correlation functions hold
(Theorem \ref{consistency}).

\subsection{Locality and the proof}
One can relax the assumption of a full vertex algebra, (\ref{eq_Borcherds}), to the existence of $D,\D \in \End\,F$ and
the following condition:
For $u \in F^\vee$ and $a_1,a_2,a_3\in F$, there exists $\mu \in \GCor_2$ such that
\begin{align*}
u(Y(a,\uz_1)Y(b,\uz_2)c)&=\mu(z_1,z_2)|_{|z_1|>|z_2|},\\
u(Y(b,\uz_2)Y(a,\uz_1)c)&=\mu(z_1,z_2)|_{|z_2|>|z_1|}.
\end{align*}
It is a generalization of the Goddard's axiom of a vertex algebra (Proposition \ref{locality}).
By using this description, we construct a QP-generated self-dual full vertex operator algebra associated with an even lattice.

Hereafter, we recall the $S_4$-symmetry of the correlation function and then
explain the key idea of the proof of the consistency.
The symmetric group $S_4$ acts on $X_4 \subset (\C P^1)^4$ by the permutation,
thus, on $\Cor_4$. As a consequence of the consistency,
four point correlation functions $S:F^{\otimes 4}\rightarrow \Cor_4$ posses an $S_4$-symmetry,
\begin{align*}
\si \cdot S(a_1,a_2,a_3,a_4)=S(a_{\si^{-1} 1},a_{\si^{-1} 2},a_{\si^{-1} 3},a_{\si^{-1} 4}), \tag{$S_4$-symmetry} \label{eq_S4}
\end{align*}
for any $a_1,a_2,a_3,a_4\in F$ and $\si \in S_4$.
Interestingly, in the Goddard's axiom, we only assume that the correlation function is invariant under the permutation $(23) \in S_4$. In contrast, four point correlation functions should be invariant under the $S_4$-symmetry.
We observe that
since the $S_4$-action commutes with the diagonal action
of $\mathrm{PSL}_2 \C$ on $X_4 \subset (\CP)^4$,
the homeomorphism $\xi: \mathrm{PSL}_2 \C \backslash X_4 \rightarrow \C P^1 \setminus \{0,1,\infty \}$
induces a $S_4$-action on $\C P^1 \setminus \{0,1,\infty \}$,
which permutes $\{0,1,\infty \}$. 
Thus, the $S_4$-action is degenerate to the $S_3$-action,
that is, the action of the Klein subgroup $C_2\times C_2 \subset S_4$ is a part of the global conformal symmetry.
Since $(34)\cdot \xi= \frac{\xi}{\xi-1}$
and both $\xi$ and $\frac{\xi}{\xi-1}$ go to zero as $\xi \rightarrow 0$,
there is a simple relation between 
the expansion of $f\circ \xi$ in the domain $|z_1|>|z_2|>|z_3|>|z_4|$ and $|z_1|>|z_2|>|z_4|>|z_3|$.
%$|z_{\si }|>|z_{\si }|>|z_{\si }|>|z_{\si }|$
This relation gives the skew-symmetry of a full vertex algebra $F$, i.e.,
$$Y(a,\uz)b=\exp(Dz+\D \z)Y(b,-\uz)a
$$
for any $a,b\in \F$.
Since $S_4$ is generated by the Klein subgroup together with $(23),(34) \in S_4$,
the axiom of a full vertex algebra and the global conformal symmetry
implies that the $S_4$-symmetry of the correlation function (Theorem \ref{S4_symmetry}).

Another important point is that in the definition of full vertex algebra, we consider the generalized correlation function,
whereas in the consistency, we need the four point correlation functions.
This problem is solved by the differential equations derived from the global conformal symmetry.

The bootstrap hypothesis is also proved by using this differential equation.
More precisely, consider the following parenthesized correlation functions:
\begin{comment}
\begin{align*}
&(\1,\Bigl(\Bigl(a_1(\ux_1)a_2\Bigr)(\ux_2)a_3(\ux_3)a_4\Bigr)(\ux_4)\1) \tag{s-channel} \label{eq_s}\\
&(\1,\Bigl(\Bigl(a_1(\ux_1)a_4\Bigr)(\ux_4)a_2(\ux_2)a_3\Bigr)(\ux_3)\1)\tag{t-channel} \label{eq_t} \\
&(\1,\Bigl(\Bigl(a_1(\ux_1)a_3\Bigr)(\ux_3)a_2(\ux_2)a_4\Bigr)(\ux_4)\1)\tag{u-channel} \label{eq_u}.
\end{align*}
\end{comment}
\begin{align*}
&(\1,\Bigl(\Bigl(a_2(x_2)a_1\Bigr)(x_1)a_3(x_3)a_4\Bigr)(x_4)\1) \tag{s-channel} \label{eq_s}\\
&(\1,\Bigl(\Bigl(a_4(x_4)a_1\Bigr)(x_1)a_3(x_3)a_2\Bigr)(x_2)\1)\tag{t-channel} \label{eq_t} \\
&(\1,\Bigl(\Bigl(a_3(x_3)a_1\Bigr)(x_1)a_2(x_2)a_4\Bigr)(x_4)\1)\tag{u-channel} \label{eq_u}.
\end{align*}
The equality of s-channel and t-channel (or s-channel and u-channel) are called 
a {\it bootstrap equation}. 
By using the differential equation, we can relate them with the generalized two point functions
in (\ref{eq_Borcherds}) and one can show that under some assumptions
if a vertex operator $Y$ on $F$ satisfies the bootstrap equation,
then $F$ is a full vertex algebra (Proposition \ref{bootstrap}).

\subsection{Outline}
In Section 1, the definition of the space of correlation functions $\Cor_4$ and the expansions $e_A:\Cor_4\rightarrow T_A$
associated with parenthesized products are given.
Properties of the functions and expansions,
e.g., relations among the expansions in different regions
and differential equations, holomorphic correlation functions,
are studied. In Section 2, we define a full vertex algebra and study the consequence of the axiom
by using the results in the previous section.
In particular, we study a property of the translation map $D,\D$ and 
show that $\ker \D$ and $\ker D$ is a vertex algebra. A tensor product of full vertex algebras is also constructed here.
Section 3 is devoted to studying the parenthesized correlation functions and its consistency.
The existence of bilinear form on a full vertex operator algebra are also studied.
In Section 4, a full vertex algebra is constructed from the Goddard's axiom
and the bootstrap equation and In Section 5, we construct an example of a QP-generated self-dual full vertex operator algebra
from an even lattice, which generalize lattice vertex algebras. In Appendix, the combinatorial rule for the expansions is given.

\section*{Acknowledgements}
The author would like to express his gratitude to his supervisor Masahito Yamazaki
and Professor Yuji Tachikawa for useful discussions
and to Shigenori Nakatsuka for reading the manuscript and valuable comments.
This work was supported by World Premier International Research Center Initiative 
(WPI Initiative), MEXT, Japan. 
The author was also supported by the Program for Leading Graduate Schools, MEXT, Japan.

\tableofcontents

\section*{Preliminaries and Notations}\label{sec_preliminary}
We assume that the base field is $\C$ unless otherwise stated. 
Throughout of this paper, $z$ and $\z$ are independent formal variables.
We will use the notation $\underline{z}$ for the pair $(z,\z)$ and $|z|$ for $z\z$.
For a vector space $V$,
we denote by $V[[z_1,\z_1,|z_1|^\R,\dots, z_n,\z_n,|z_n|^\R ]]$ the set of formal sums 
$$\sum_{s_1,\s_1,\dots,s_n,\s_n \in \R} a_{s_1,\s_1,\dots,s_n,\s_n}
z^{s_1} \bar{z}^{\s_1}\dots z^{s_n}\z^{\s_n}$$ such that
\begin{enumerate}
\item
$a_{s_1,\s_1,\dots,s_n,\s_n}=0$ unless $s_i-\s_i \in \Z$ for all $i=1,\dots,n$;
\item
$\{({s_1,\s_1,\dots,s_n,\s_n}) \in \R^{2n} \;|\; a_{s_1,\s_1,\dots,s_n,\s_n} \neq 0 \}$ is a countable set.
\end{enumerate}
We also denote by
$V((z_1,\z_1,|z_1|^\R,\dots,z_n, \z_n,|z_n|^\R))$ the subspace of $V[[z_1,\z_1,|z_1|^\R,\dots, z_n,\z_n,|z_n|^\R ]]$ spanned by the series satisfying
\begin{enumerate}
\item[(3)]
There exists $N \in \Z$ such that
$a_{s_1,\s_1,\dots,s_n,\s_n}=0$ unless $s_i,\s_i \geq N$ for all $i=1,\dots,n$.
\end{enumerate}
For a series $f(\uz)=\sum_{r,s \in\R}a_{r,s}z^r\z^s \in V((z,\z,|z|^\R))$,
The number of the elements $\{r \in \R \;|\; a_{r,s} \neq 0 \}$ in the image of
the natural map $\R \rightarrow \R/\Z$ is called an {\it exponent of the series}.
If the exponent of $f(\uz)$ is $l \in \Z_{\geq 0}$, then
$f(\uz)$ could be written as
$$\sum_{i=1}^l \sum_{n,m=0}^\infty a_{n,m}^i z^{n}\z^{m}|z|^{r_i},
$$
where $r_i \in \R$ and $r_i-r_j \notin \Z$ for any $i\neq j$.

We will consider the following subspaces of $V[[z,\z,|z|^\R]]$:
\begin{align*}
V[[z,\z]]&= \{ \sum_{s,\s \in \Z_{\geq 0}} a_{s,\s} z^s\z^\s \;|\; a_{s,\s} \in V \}, \\
V[z^\pm,\z^\pm]&=\{\sum_{s,\s \in \Z}a_{s,\s}z^s \z^\s \;|\;
a_{s,\s} \in V,
\text{ all but finitely many } a_{s,\s}=0 \}, \\
V[|z|^\R] &= \{\sum_{r \in \R}a_{r}z^r \z^r \;|\;
a_{r} \in V,
\text{ all but finitely many } a_{r}=0 \}.
\end{align*}
We will also consider their combinations, e.g.,
%$$V\{ y/x \}[x^{\pm},|x|^\R,y^\pm,|y|^\R]
%=\{\sum_{i=1}^l \sum_{n,m \in \N}a_{n,m}^i z^n \z^m |z|^{r_i} \;|\;
%l \in \Z_{>0} \tand  r_i \in \R, a_{n,m}^i \in V \fora i=1,\dots,l \},
%$$
%and
$
V((y/x,\y/\x,|y/x|^\R))[x^\pm,\x^{\pm},|x|^\R],
$
which is a subspace of $V[[x,y,\x,\y,|x|^\R,|y|^\R]]$ spanned by
$$\sum_{i=1}^k \sum_{n,m=-l}^l\sum_{r,s\in \R}
a_{n,m,r,s}^i x^{n+r_i} \x^{m+r_i} (y/x)^r(\y/\x)^s$$
%\sum_{k=1}^N \sum_{i =0}^\infty\sum_{n,\n\geq 0} a_{n,\n}^{k,i}x^{s_k}\x^{\s_k} (y/x)^{n+r_i}(\y/\x)^{\n+r_i}$$
%a_{s_k-n-r_i,\s_k-\n-r_i,n+r_i,\n+r_i} x^{s_k}\x^{\s_k} (y/x)^{n+r_i}(\y/\x)^{\n+r_i}$$
for some $k,l \in \Z_{>0}$ and $r_i \in \R$ and $a_{n,m,r,s}^i \in V$
%s_k-n-r_i,\s_k-\n-r_i,n+r_i,\n+r_i} \in V$
such that $a_{n,m,r,s}^i=0$ unless $r-s\in\Z$
and there exists $N$ such that $a_{n,m,r,s}^i=0$ unless
$r\geq N$ and $s \geq N$.

%$s_k-\s_k \in \Z$ and $\{r_i\}_{i=0,1,\dots}$ is bounded below and $r_i - r_j \notin \Z$ for any $i \neq j$.
Let $\frac{d}{dz}$ and $\frac{d}{d\z}$ be formal differential operators
acting on $V[[z,\z,|z|^\R]]$ by
\begin{align*}
\frac{d}{dz}\sum_{s,\s \in \R} a_{s,\s}z^s \bar{z}^\s
&= \sum_{s,\s \in \R} s a_{s,\s}z^{s-1} \bar{z}^\s \\
\frac{d}{d\z}\sum_{s,\s \in \R} a_{s,\s}z^s \bar{z}^\s
&= \sum_{s,\s \in \R} \s a_{s,\s}z^{s} \bar{z}^{\s-1}.
\end{align*}
Since $\frac{d}{dz}|z|^s=s|z|^s z^{-1}$, the differential operators
$\frac{d}{dz}$ and $\frac{d}{d\z}$ acts on all the above vector spaces.
Define a linear map
$\lim_{z \mapsto f(w)}$ from a space of the formal variable $(z,\z)$
to a space of the formal variable $(w,\bar{w})$
%$ V\{ \uz \} \rightarrow V\{\underlline{w} \}$
if the substitution of $(f(w),f(\w))$ into $(z,\z)$ is well-defined,
that is, each coefficient is a finite sum.
For example,
$\lim_{z \mapsto -z} V((z,\z,|z|^\R)) \rightarrow V((z,\z,|z|^\R))$
is given by
$$
\lim_{\uz \mapsto -\uz}  \sum_{r,s \in \R} a_{r,s}z^r \bar{z}^s
=\sum_{r,s \in \R}(-1)^{r-s} a_{r,s}z^r \bar{z}^s.
$$
Since $a_{r,s}=0$ for any $r-s \notin \Z$,
$\lim_{z \mapsto -z}$ is well-defined.
Another example is
\begin{align*}
\lim_{x \to x+y}:
V((y/x,\y/\x,|y/x|^\R)) \rightarrow V((y/x,\y/\x,|y/x|^\R)),
\end{align*}
which is defined by
\begin{align*}
&\lim_{x \to x+y}\Bigl(\sum_{r,s\in\R}
a_{r,s}(y/x)^r(\y/\x)^s\Bigr)
&=
\sum_{r,s\in\R}\sum_{i,j\in \Z_{\geq 0}}
\binom{-r}{i}\binom{-s}{j}
a_{r,s}(y/x)^{r+i}(\y/\x)^{s+j},
\end{align*}
where we used $(x+y)^s=\sum_{m \geq 0}\binom{s}{m}x^{s-m}y^m$.
It is well-defined since $a_{r,s}=0$ for sufficiently small $r$ or $s$.
%each coefficient is a finite sum.
%
% well-defined always define \lim
%
\begin{comment}
For $n \in \Z_{>0}$, we introduce the following important spaces of formal power series:
\begin{align*}
T(z_1,\dots,z_n)&=\C[[\underline{z_2/z_1}]]((\underline{z_3/z_2},\underline{z_4/z_3}, \dots,\underline{z_n/z_{n-1}} ))[\uz_1^\pm,|z_1|^\R],\\
%\C[[z_2/z_1,z_2/z_1,\dots,z_n/z_{n-1},z_n/z_{n-1}]][z_1^\pm,\z_1^\pm,|z_1|^\R,\dots,z_n^\pm,\z_n^\pm,|z_n|^\R],\\
U(z_1,\dots,z_n)&=\C((\underline{z_2/z_1},\underline{z_3/z_2}, \dots,\underline{z_n/z_{n-1}} ))[\uz_1^\pm,|z_1|^\R].
%\C[[z_2/z_1,z_2/z_1,\dots,z_n/z_{n-1},z_n/z_{n-1}]][z_1^\pm,\z_1^\pm,|z_1|^\R,\dots,z_{n-1}^\pm,\z_{n-1}^\pm,|z_n-1|^\R].
\end{align*}
Let $n \in \Z_{\geq 2}$ and $g(z_1,\dots,z_n) \in \C\{z_1,\z_1,\dots,z_{n},\z_{n}\}$.
Set $g(z_1,\dots,z_n)=\sum_{r,s \in \R}  g_{r,s}(z_1,\dots,z_{n-1}) z_n^{r}\z_n^s,$
where $g_{r,s}(z_1,\dots,z_{n-1}) \in \C\{z_1,\z_1,\dots,z_{n-1},\z_{n-1}\}$.
We could not define a product on 
but 
We remark that $T(z_1,\dots,z_n)$ is not $\C$-algebra but
the subspace $$
T^f(z_1,\dots,z_n) = \C[[\underline{z_2/z_1},\underline{z_3/z_2}, \dots,\underline{z_n/z_{n-1}}]][\uz_1^\pm,|z_1|^\R]
\subset T(z_1,\dots,z_n)
$$ is a $\C$-algebra.
Furthermore, we have:
\begin{lem}\label{ring_expansion}
$T(z_1,\dots,z_n)$ is a module over the $\C$-algebra $T^f(z_1,\dots,z_n)$.
\end{lem}
\end{comment}
We end this section by discussing a convergence of a formal power series
in $\C((z_1,\z_1,|z_1|^\R,\dots,z_n,\z_n,|z_n|^\R))$.
Let $f\in \C((z_1,\z_1,|z_1|^\R,\dots,z_n,\z_n,|z_n|^\R))$.
Then, there exists $N \in \R$ such that
\begin{align}
|z_1|^N|z_2|^N \cdots|z_n|^N f(z_1,\dots,z_n)
=
\sum_{\substack{s_1,\s_1,\dots,s_n,\s_n \in \R
\\  s_1,\s_1,\dots,s_n,\s_n \geq 0}} a_{s_1,\s_1,\dots,s_n,\s_n}
z^{s_1} \bar{z}^{\s_1}\dots z^{s_n}\z^{\s_n}.
\end{align}
We say the series $f$ is absolutely convergent around $0$
if $|z_1|^N|z_2|^N \cdots|z_n|^N f(z_1,\dots,z_n)$ is
absolutely convergent in 
$\{(z_1,z_2,\dots,z_n)\in \C^n\;|\; |z_1|,|z_2|,\dots,|z_n|<R\}$
for some $R \in \R_{>0}$.
In this case, $f(z_1,\dots,z_n)$ is compactly absolutely-convergent
to a continuous function in the annulus $\{(z_1,z_2,\dots,z_n)\in \C^n\;|\;0<|z_1|,|z_2|,\dots,|z_n|<R\}$.

\section{Correlation functions and formal calculus}\label{sec_formal}
In this section, we define the space of four point correlation functions, $\Cor_4$ and develop formal calculus which we need to study a full vertex algebra.
In Section \ref{sec_four_point}, we define $\Cor_4$
and in Section \ref{sec_S4}, we study an action of the symmetric group $S_4$ on $\Cor_4$.
Section \ref{sec_expansion} and \ref{sec_expansion2} is devoted to studying series expansions of a function in $\Cor_4$.
We study formal differential equations in Section \ref{sec_differential} and
holomorphic correlation function in Section \ref{sec_holomorphic_func}.
Certain limit of a four point correlation function is studied in Section \ref{sec_gen}.
The reader who is only interested in the
definition of a full vertex algebra can skip Section \ref{sec_expansion},
\ref{sec_relation_expansion}, \ref{sec_differential}, \ref{sec_holomorphic_func} and \ref{sec_expansion2}.

\subsection{Definition of four point functions}\label{sec_four_point}
For $n \in \Z_{\geq 0}$,
set $X_n=\{(z_1, \dots, z_n)\in (\C P^1)^n\;|\; z_i\neq z_j \fora i\neq j \}$, called a space of ordered configurations of $n$ points in the projective space $\C P^1$.
In this section, we define and study a space of four point correlation functions in two dimensional conformal field theory,
which are $\C$-valued real analytic functions on $X_4$ with possible singularity along $\{(z_1,\dots,z_4)\in (\C P^1)^4\;|\;z_i=z_j\}$ for $1\leq i<j\leq 4$.
Recall that the automorphism group of $\C P^1$ is $\mathrm{PSL}_2 \C$, whose action is called the linear fractional transformation and also called a global conformal symmetry. 
Define the real analytic function $\xi :X_4 \rightarrow P^1\C\setminus \{0,1,\infty\}$
by $$\xi(z_1,z_2,z_3,z_4)=\frac{(z_1-z_2)(z_3-z_4)}{(z_1-z_3)(z_2-z_4)}.$$
It is easy to show that the map $\xi :X_4 \rightarrow \C P^1\setminus \{0,1,\infty\}$ is invariant under the diagonal action of $\mathrm{PSL}_2 \C$ on $X_4 \subset (\C P^1)^4$, in fact,
which gives a homeomorphism $\mathrm{PSL}_2 \C \backslash X_4 \rightarrow P^1\C\setminus \{0,1,\infty\}$
(see, for example, \cite{Y}).
\begin{comment}
Since any four point function of quasi-primary states in 2d CFT possesses  global conformal symmetry,
the four point functions can be regarded as a real analytic function on $\C P^1\setminus \{0,1,\infty\}$ with some singularities
 at $\{0,1,\infty \}$, which we call a conformal singularity.
\end{comment}
\begin{rem}
\label{rem_hypersurface}
The hypersurfaces in $(\C P^1)^4$,
$$\{z_1=z_2\} \cup \{z_3=z_4\}$$
and $$\{z_1=z_4\} \cup \{z_2=z_3\}$$
and $$\{z_1=z_3\} \cup \{z_2=z_4\}$$
are maps to $0$ and $1$, $\infty$ respectively by $\xi:(\C P^1)^4\rightarrow \C P^1$.
\end{rem}

\begin{comment}
In Section \ref{sec_expansion}, we define their (Taylor) expansions in some region,
which have $D_4$-symmetry. This symmetry is studied in Section \ref{sec_C2} and \ref{sec_involution}.
In Section \ref{sec_formal_hol}, holomorphic four point functions are studied
and in Section \ref{sec_differential}, some

They are  four point configulation space over $\C$
with some singularity and symmetry.
More precisely, 
\end{comment}

% general 4-point function and expansion
%
%Let $M$ be a one dimensional complex manifold.

%
% vertex operator to correlation function
% motivation for expansion.
%
%

Let $\al_1,\dots,\al_n \in \C P^1$ and
$f$ be a $\C$-valued real analytic function on $\C P^1\setminus \{\al_1,\dots,\al_n\}$.

A chart $(\chi,\alpha)$ of $\C P^1$ at a point $\alpha \in \C P^1$ is a biholomorphism $\chi$
from an open subset $U$ of $\C P^1$ to an open subset of $\C$ such
that $\alpha \in U$ and $\chi(\alpha)=0$.
We say that $f$ has a {\it conformal singularity} at $\alpha_i$
if for any chart $(\chi,\alpha_i)$ of $\C P^1$ at $\alpha_i$,
there exists a formal power series
\begin{align}
\sum_{r,s \in \R} a_{r,s} p^r \p^s \in \C((p,\p,|p|^\R)) \label{eq_CS2}
\end{align}
such that it is compactly absolutely-convergent to $f\circ \chi^{-1}(p)$
in the annulus $\{p\in \C\;|\;0<|p|<R\}$ for some $R\in \R_{>0}$
(see Section \ref{sec_preliminary}).
It is clear that the above condition is independent of a choice of a chart
and the coefficients of the series is uniquely determined by the chart.

\begin{comment}
\label{unique_expansion}
If $r_k$ is the smallest number among $\{r_i\}_{i=1,\dots, l}$,
then $a_{0,0}^k= \lim_{p \rightarrow 0}(p\bar{p})^{-r_k}f(\phi^{-1}(p))$.
By Changing $f(\phi^{-1}(p))$ to $f(\phi^{-1}(p))-(p\bar{p})^{r_k}a_{0,0}^{r_k}$ and repeating the procedure, we get the coefficient $a_{n,m}^i$.
Hence, the expansion is uniquely determined by $f$ and the chart $\phi$.
\end{comment}

%A real analytic function $f$ on $\{z \in \C \,|\, |z|>N \}$ has a conformal singularity at $\infty$
%if $f(1/z)$ has conformal singularity at $0$.
Denote by $F_{0,1,\infty}$ the space of real analytic functions on $\C P^1 \setminus \{0,1,\infty \}$
with possible conformal singularities at $\{0,1,\infty\}$.
Let $p$ be a canonical coordinate of $\C \subset \C P^1$.
A non-trivial example of such functions is 
\begin{align}
f_{\mathrm{Ising}}(p)=|1-\sqrt{1-p}|^{1/2}+|1+\sqrt{1-p}|^{1/2},  \label{eq_Ising}
\end{align}
which appears in a four point function of the 2 dimensional Ising model.
We remark that the exponent of the series (\ref{eq_CS2}) is independent of the choice of the chart.
An exponent of a function $f \in \F$ is the maximal number of the exponents of the series expansion of $f$ at $\{0,1,\infty\}$.
For example, the expansion of $f_{\mathrm{Ising}}$ for the charts $p$ is
\begin{align}
2+|p|^{1/2}/2-p/4-\p/4+|p|^{1/2}(p+\p)/16+ p\p/32-5p^2/64-5\p^2/64+\dots.%,\tag{z} \\
%&2+|z|^{1/2}/2-z/4-\z/4+|z|^{1/2}(z+\z)/16+ z\z/32-5z^2/64-5\z^2/64+\dots,\tag{1-z}\\
%.\tag{1/z}
\end{align}
Since $f_{\mathrm{Ising}}(p)$ satisfies $f_{\mathrm{Ising}}(p)=f_{\mathrm{Ising}}(1-p)=
(z\z)^{1/4}f_{\mathrm{Ising}}(1/p),$
the exponent of $f_{\mathrm{Ising}}$ is $2$.
We remark that for a rational conformal field theory (or more generally quasi-rational conformal field theory), the exponent is always finite.

In this paper, we consider the following
special charts of $\{0,1,\infty\}$:
\begin{align*}
\mathrm{Chart}(0,1,\infty)=\{p, \frac{p}{p-1}, 1-p, 1-\frac{1}{p}, \frac{1}{p},  \frac{1}{1-p}\}.
\end{align*}
Recall that the ring of regular functions on the affine scheme
$\CPm$ is $\C[p^\pm,(1-p)^\pm]$
and $\Ch$ is a generator of the ring.
It is easy to show that
a function in $\C[p^\pm,(1-p)^\pm]$ has conformal singularities at $\{0,1,\infty\}$. Thus, $\C[p^\pm,(1-p)^\pm]\subset \F$.
Conversely, from the existence of the expansion the following proposition follows:
\begin{prop}\label{holomorphic_F}
If $f \in \F$ is a holomorphic function on $\C P^1 \setminus \{0,1,\infty \}$,
then $f \in \C[p^\pm,(1-p)^\pm]$.
\end{prop}
%Let $V$ be a vector space. We denote by $V\{z,\z \}$
%the set of formal sum $\sum_{n,m \in \R}v_{n,m}z^n\z^m $ with $v_{n,m} \in V$.
%%%%%% we remark that
%%%%%%  z_12 + z_23 =z_13. hence it is important to recoginize the variables are formal or not. 
%%%%%%
%%%%%%

%by setting the grading of $\zij^{\al_{ij}}\zijd^{\be_{ij}}$ as $(\al_{ij}, \be_{ij}) \in \R^12$.

We will consider the real analytic functions on $X_4$ of
the form:
\begin{align}
 \Pi_{1\leq i<j \leq 4} (z_i-z_j)^{\al_{ij}}(\bar{z}_i-\bar{z}_j)^{\be_{ij}} f\circ \xi(z_1,z_2,z_3,z_4), \label{eq_four2}
\end{align}
where $f \in \F$ and $\al_{ij}, \be_{ij} \in \R$ such that $\al_{ij}-\be_{ij}\in \Z$ for any $1 \leq i<j\leq 4$.
In this paper, we always allow to have singularities around $\{z_i=\infty\}_{i=1,2,3,4}$;
The asymptotic behavior of (\ref{eq_four2}) as $z_1\rightarrow \infty$
is $z_1^{\al_{12}+\al_{13}+\al_{14}}\z_1^{\be_{12}+\be_{13}+\be_{14}}$.
By the condition $\al_{ij}-\be_{ij}\in \Z$, we have
$$ (z_i-z_j)^{\al_{ij}}(\bar{z}_i-\bar{z}_j)^{\be_{ij}}=|z_i-z_j|^{\al_{ij}}(\bar{z}_i-\bar{z}_j)^{\be_{ij}-\al_{ij}},$$
which implies that $ \Pi_{1\leq i<j \leq 4} (z_i-z_j)^{\al_{ij}}(\bar{z}_i-\bar{z}_j)^{\be_{ij}}$ is a single valued real analytic function on $X_4$.

%\begin{rem}\label{rationality}
%Functions in (\ref{eq_four}) appear as four point correlation functions of quasi-primary states in 2d conformal field theory (see e.g., \cite{}).
%The exponent of the series in (\ref{eq_CS}) is almost equal to the number of intermediate states in
%the s,t,u-channel reaction, which is mathematically the number of the irreducible components of the tensor product of modules.
%Thus, in (\ref{eq_CS}) and (\ref{eq_four}), we assume that the number of intermediate states is countable.
%Hence, Liouville field theory is excluded from our consideration,
%where we should consider some measure on the space of irreducible representations and integral over there,
%which beyond the scope of this paper.

%, we assume that the singular term $(p\p)^{r_i}$ is finite $(i=1,\dots,l)$.
%This assumption is satisfied for 2d conformal field theory whose
%fusion rule of conformal blocks is finite %(see Section \ref{sec_correlation}, especially Lemma \ref{discrete}),
%e.g., rational conformal field theory. 
%We remark that there is a non-rational conformal field theory satisfying this assumption,
%e.g., toroidal compactification of string theory and its orbifold (see Section \ref{sec_lattice})
%and also
%We remark that
%\end{rem}

Let $\Cor_4$ be the space of $\C$-valued real analytic functions spanned by (\ref{eq_four2}).
We also consider the spaces of functions on $X_2$ and $X_3$ and $X_4$ spanned by
\begin{align*}
(z_1-z_2)^{\al_{12}}(\bar{z}_1-\bar{z}_2)^{\be_{12}}, \\
\Pi_{1\leq i<j \leq 3} (z_i-z_j)^{\al_{ij}}(\z_i-\z_j)^{\be_{ij}}, \\
\Pi_{1\leq i<j \leq 4} (z_i-z_j)^{\al_{ij}}(\z_i-\z_j)^{\be_{ij}} %, \tag{CO2} \label{eq_free}
\end{align*}
where $\al_{ij}, \be_{ij} \in \R$ such that $\al_{ij}-\be_{ij}\in \Z$ for any $1 \leq i<j\leq 4$.
We denote them by $\Cor_2, \Cor_3$ and $\Cor_4^f$, respectively.
By the definition, we have:
\begin{lem}
\label{four_translation}
A function $\phi \in \Cor_4$ is translation invariant, i.e., 
$\phi(z_1,z_2,z_3,z_4)=\phi(z_1+\al,z_2+\al,z_3+\al,z_4+\al)$ for any $\al \in \C$.
\end{lem}

\subsection{$S_4$-symmetry}\label{sec_S4}
Define the left action of the symmetric group $S_4$ on $(\C P^1)^4$ by the permutation,
$\sigma\cdot (z_1,z_2,z_3,z_4) = (z_{\si^{-1} 1},z_{\si^{-1} 2},z_{\si^{-1} 3},z_{\si^{-1} 4})$ for $(z_1,z_2,z_3,z_4)  \in (\C P^1)^4$ and $\si \in S_4$.
Since the action commutes with the diagonal action of $\mathrm{PSL}_2 \C$ on $X_4 \subset (\C P^1)^4$,
we have an $S_4$-action on $\C P^1 \setminus \{0,1,\infty \} $.
It is easy to show that each image of $\sigma  \in S_4$ in $\Aut \C P^1 \setminus \{0,1,\infty \}$
is a linear fractional transformation and preserves $\{0,1,\infty \}$ (see Remark \ref{rem_hypersurface}).
Let $\Aut (0,1,\infty)$ be the subgroup of the linear fractional transformations $\mathrm{PSL}_2 \C$ which preserves $\{0,1,\infty\}$.
Since the action of $\mathrm{PSL}_2 \C$  on $\C P^1$ is strictly $3$-transitive,
an element of $\Aut (0,1,\infty)$ is uniquely determined by the action on the $3$ points $\{0,1,\infty \}$.
Thus, $\Aut (0,1,\infty)$  is isomorphic to the permutation group $S_3$.
Hence, we have a group homomorphism $t: S_4 \rightarrow \Aut (0,1,\infty)$
and the kernel of this map is the Klein subgroup $K_4= \{1, (12)(34), (13)(24), (14)(23)\}$.

The action of $S_4$ on $X_4$ induces the dual action on the function space $\Cor_4$ by $\sigma f(-) = f(\sigma^{-1}-)$ for $f \in \Cor_4$ and $\sigma \in S_4$,
which also induces the action on the function space of $\C P^1 \setminus \{0,1,\infty \}$, $\F$.
\begin{rem}
For $\si \in S_4$ and $\Pi_{1\leq i<j \leq 4} (z_i-z_j)^{\al_{ij}}(\z_i-\z_j)^{\be_{ij}} \in \Cor_4^f$,
$$\si \cdot \Pi_{1\leq i<j \leq 4} (z_i-z_j)^{\al_{ij}}(\z_i-\z_j)^{\be_{ij}}
=\Pi_{1\leq i<j \leq 4} (z_{\si i}-z_{\si j})^{\al_{ij}}(\z_{\si i}-\z_{\si j})^{\be_{ij}}.
$$
\end{rem}
Regard the chart $p\in \Ch$ as a function on $\C P^1 \setminus \{0,1,\infty\}$, for $\si\in S_4$,
%$\chi(p) \in \Ch$, regarding $\chi(p)$ as a rational polynomial of the canonical coordinate $p$,
we have $\sigma \cdot p(-)=p(t_\si^{-1}-)=t_\si^{-1}.p$,
where $\left(
    \begin{array}{cc}
      a & b \\
      c & d
    \end{array}
  \right).p$ is defined by $\frac{ap+b}{cp+d} \in \C(p)$ for $\left(
    \begin{array}{cc}
      a & b \\
      c & d
    \end{array}
  \right) \in \mathrm{PSL}_2 \C$.
%We denote the orbits $\{\sigma p \}_{\si \in S_4}$ by $G_p$.
\begin{lem}
\label{orbit_xi}
The explicit expressions of the map $t: S_4 \rightarrow \Aut (0,1,\infty)\subset \mathrm{PSL}_2 \C$
and the orbit $\sigma \cdot p= t_\si^{-1}.p$ is
%orbit of $\xi$ under the action of $S_4$ is
% anti-homomorphism S_4 \rightarrow  \subset \mathrm{PSL}_2\C$ is
\begin{align*}
K_4 &\mapsto (1) = \left(
    \begin{array}{cc}
      1 & 0 \\
      0 & 1
    \end{array}
  \right),  &p& \mapsto p, \\
(12)K_4 &\mapsto  (1 \infty)= \left(
    \begin{array}{cc}
      1 & 0 \\
      1 & -1
    \end{array}
  \right), &p& \mapsto \frac{p}{p-1}, \\
(23)K_4 &\mapsto
(0 \infty) = \left(
    \begin{array}{cc}
      0 & 1 \\
      1 & 0
    \end{array}
  \right), &p& \mapsto \frac{1}{p}, \\
(13)K_4 &\mapsto 
(01) = \left(
    \begin{array}{cc}
      -1 & 1 \\
      0 & 1
    \end{array}
  \right), &p& \mapsto 1-p, \\
(123)K_4 &\mapsto (01\infty) = \left(
    \begin{array}{cc}
      0 & 1 \\
      -1 & 1
    \end{array}
  \right), &p& \mapsto \frac{p-1}{p}, \\
(132)K_4 &\mapsto (0\infty 1)= \left(
    \begin{array}{cc}
      1 & -1 \\
      1 & 0
    \end{array}
  \right), &p& \mapsto \frac{1}{1-p}.
\end{align*}
Furthermore, $\si\cdot \xi=t_\si^{-1}.\xi$ for any $\si \in S_4$.
\end{lem}
Thus, the $S_4$-action induces an action on $\Ch$,
i.e., $\si \cdot \chi(p)=\chi(p(\si^{-1}-))=\chi(t_\si^{-1}.p)$ for $\chi(p) \in \Ch$ and $\si \in S_4$.
Hence, we have:
\begin{lem}
\label{inv_S4}
The function spaces  $\Cor_4$ and $\Cor_4^f$ and $\F$ are preserved by the action of $S_4$.
\end{lem}
\begin{comment}
A group structure on the set $p,1-p, p^{-1}, 1-p^{-1}, \frac{1}{(1-p)},\frac{p}{(p-1)}$ is induced by considering
elements in $G_p$ as a coordinate transformation, which is the opposite product of linear fractional transformations.
More specifically, for $t_1(p), t_2(p) \in G_p$, which are rational polynomials in $\C(p)$,
the product $t_1(p)\cdot t_2(p)$ is defined by $t_2(t_1(p))$.
Then, we have a group homomorphism $S_4 \rightarrow G_p$.
\end{comment}
%\begin{rem}\label{rem_xi_S4}
%By the definition of the $S_4$-action on $\C P^1$,
%We remark that $\Ch$ forms a group under the opposite composition,
%that is, $\chi_1(p) \ast \chi_2(p)=\chi_2(\chi_1(p))$ for $\chi_1(p),\chi_2(p)\in \Ch$,
%which is isomorphic to $S_3$.
%Then, $F_\mathrm{ch}: \Aut (0,1,\infty) \rightarrow \Ch, t \rightarrow t^{-1}.p$ gives an isomorphism as groups,
%since for $t_1,t_2\in \Aut (0,1,\infty)$, $F_\mathrm{ch}(t_1) \ast F_\mathrm{ch}(t_2)=t_1^{-1}.p \ast t_2^{-1}.p=t_2^{-1}.t_1^{-1}.p=F_\mathrm{ch}(t_1t_2).$
%\end{rem}

%Let $p$ be a standard chart of $\C \subset \C P^1$.
%Define a map $s_0:G_p \rightarrow \{0,1,\infty \}$ by substitution of $p=0$ into an element of $G_p$.
%Then, by Remark \ref{Gp}, we can regard an element of $G_p$ as a chart of $\C P^1$

By the definition of a conformal singularity, for a chart $\chi(p) \in \Ch$ and $f\in \F$,
we have the expansion of $f \circ \chi^{-1}(p)$ around $p=0$,
which defines a linear map $j(\chi(p), - ): \F \rightarrow \C((p,\bar{p},|p|^\R))$, that is,
$$j(\chi(p),f)=\sum_{r,s\in \R}a_{r,s}p^r\bar{p}^s,$$
where the right-hand-side is compactly absolutely convergent to $f(\chi^{-1}(p))$ around $p=0$.

For $t \in \Aut(0,1,\infty)$,
$t \cdot p= t^{-1}.p \in \Ch$ is a chart at $t.0$
and the inverse $(t \cdot p)^{-1}: \C P^1 \rightarrow \C P^1$ is given by $t^{-1} \cdot p=t.p$,
e.g., $\frac{1}{1-p}$ is a chart at $\infty \in \C P^1$ whose inverse is $\frac{p-1}{p}$.
\begin{comment}
We remark that $j_\si$ only depends on the image $\overline{\sigma}$ of $S_4 \rightarrow S_3$,
defined in Lemma \ref{orbit_xi}.
Thus, we denote $j_\si$ by $j_{\overline{\si}}$.
Since the permutation group  $S_3$ canonically acts on $\{0,1,\infty \}$, we denote the image by this permutation representation. For example, $j_{\overline{(243)}}$ is denoted by $j_{(0\infty 1)}$.
\end{comment}
\begin{lem}
\label{F_covariance}
For $\sigma \in S_4$, $\chi(p) \in \Ch$ and $f \in \F$,
$j(\chi(p), \sigma \cdot f)=j( \si^{-1} \cdot \chi(p),f).$
\end{lem}
\begin{proof}
We may assume that $\chi(p)=t.p$ for some $t\in \A$.
The series $j(t.p, \sigma f)$ is absolutely convergent to 
$f(t_\sigma^{-1} t^{-1}.p)$ around $p=0$.
Since $f(t_\sigma^{-1} t^{-1}.p)=f((tt_\si)^{-1}.p)$,
we have $j(t.p, \sigma f)=j(t t_\si.p, f)=j(\si^{-1} \cdot t.p,f)$.
\end{proof}

The stabilizer subgroup of a point $0\in \C P^1$ in $S_4$ (resp. $\A$) is
generated by the Klein subgroup $K_4$ and $(12)$ (resp. $\{1,(0 \infty)\}$), which is isomorphic to the dihedral group $D_4$.
For $\si \in K_4$, the action of $\si$ on $\F$ is trivial
and for $\si \in (12)K_4$,
the action of $\si$ on $\F$
is $(12)K_4 \cdot f(p)=f(\frac{p}{p-1})$.
Define the map $\lim_{p \mapsto \frac{p}{p-1}}:\C((p,\p,|p|^\R)) \rightarrow
\C((p,\p,|p|^\R))$
by substituting $-\sum_{n\geq 1} p^{n}$ into 
$p$.
Then, we have:
\begin{lem}
\label{dihedral}
For $\si \in K_4$,
$j(\si \cdot p,-)=j(p,-)$ and
$j((12)\cdot p,-)=\lim_{p \mapsto \frac{p}{p-1}} j(p,-)$.
\end{lem}

\begin{comment}
The orbit of $\xi \in C^\omega(X_4)$ by the action of $S_4$
is $\xi,1-\xi, 1/\xi, 1-1/\xi, 1/(1-\xi),\xi/(\xi-1)$.
By Remark \ref{invariant}, the transformations on the `function space' of $\mathrm{PSL}_2 \C \backslash X_4=\C P^1 \setminus\{0,1,\infty \}$ induces the transformations
on the `space' $\C P^1 \setminus\{0,1,\infty \}$, which is regarded as a linear fractional transformation.
Thus, we have an `anti-homomorphism' $S_4 \rightarrow \Aut \C P^1$ (and homomorphism $S_4 \rightarrow \Aut C^\omega(\C P^1 \setminus \{0,1,\infty\})$).
, denoted by $\sigma \mapsto \bar{\sigma}$,
where $\Aut (0,1,\infty)$ acts on $\{0,1,\infty\}$ by the usual left action.
\begin{rem}
A linear fractional transformation is a left action on $\C P^1$, wheres the $S_4$ action on $X_4$ is a right action.
That is why we have the anti-homomorphism.
\end{rem}
In fact, by Remark \ref{rem_xi}, the right $S_4$-action on $\mathrm{PSL}_2 \C \backslash X_4$
factors through $S_4/K_4=S_3$.
Hence, we have:
\end{comment}

\subsection{Expansions} \label{sec_expansion}
In this section, we consider expansions of  a function in $\Cor_4$.
By Lemma \ref{four_translation},
a function $\phi \in \Cor_4$ can be expanded in three variables,
e.g., $z_2-z_3, z_3-z_4, z_1-z_4$.
Since $\phi$ is a multi-variable function,
we need to determine the order of the expansion.
For example, we expand $\phi$ around $z_2=z_3$
first and then around $z_3=z_4$, and $z_1=z_4$.
By setting $x=z_1-z_4,y=z_3-z_4,z=z_2-z_3$,
the resulting series is a formal power series of the variables
$(x,\x,y,\y,z,\z)$, which is
 absolutely convergent in 
$\{|x|>|y|>|z|\}$.
All such series expansions are described by
parenthesized products of four symbols.
The above case corresponds to $1((23)4)$.
The innermost product $(23)$ means that
we expand $\phi$ around $z_2-z_3$ first.
Let $P_4$ be the set of parenthesized products of four elements $1,2,3,4$,
e.g., $1(2(34))$.
 \begin{comment}

As mentioned in Introduction \ref{intro_parenthesized},
in order to formulate the consistency of conformal field theory,
it is necessary to define expansions of a function in $\Cor_4$ around the poles $\{z_i-z_j\}_{i \neq j}$.
In this section, we will see that
the coefficient of the expansion is explicitly determine by the expansions of a function in $f \in \F$
around $\{0,1,\infty \}$, studied in the previous section.
In Introduction, we consider the expansions in four variables associated with an element in $Q_4$, however,
it is convenient to consider expansions in three variables,
which is possible .
All such expansions are defined associated with a parenthesized product of four elements.
\end{comment}
We remark that
the permutation group $S_4$ naturally acts on $P_4$
and $P_4$ consists of the permutations of the following elements, called standard elements of $P_4$:
\begin{align}
(12)(34), \nonumber \\
((12)3)4, \tag{standard elements of $P_4$}  \nonumber \\
(1(23))4,  \nonumber \\
1((23)4),  \nonumber \\
1(2(34)).  \nonumber
\end{align}

We introduce formal variables $(x_A,y_A,z_A,\x_A,\bar{y}_A,\z_A)$ for each $A\in P_4$.
In this section, for each $A\in P_4$,
we define a series expansion
$$e_A: \Cor_4 \rightarrow T_A,
$$
where $T_A$ is a space of formal variables $(x_A,y_A,z_A,\x_A,\bar{y}_A,\z_A)$ defined below and show that
the expansions are completely described by
$j(\chi(p),-):\F\rightarrow \C((p,\p,|p|^\R))$
for some $\chi(p)\in \Ch$.
\begin{rem}
In introduction \ref{intro_parenthesized},
we consider the expansions in four variables associated with an element in $Q_4$. Such expansions are described in Section \ref{sec_expansion2}
by using the results in this section.
\end{rem}

\begin{comment}
For example, in the case that $A$ is a standard elements in $P_4$,
we consider the following change of the variables:
%in some open domain in $X_4$.
%For each standard element in $P_4$, we consider the following change of variables:
\begin{align}
(12)(34):(x_{(12)(34)}, y_{(12)(34)},z_{(12)(34)})=(z_2-z_4,z_1-z_2,z_3-z_4), \nonumber \\
((12)3)4:(x_{((12)3)4}, y_{((12)3)4},z_{((12)3)4})=(z_3-z_4,z_2-z_3,z_1-z_2), \nonumber \\
(1(23))4:(x_{(1(23))4}, y_{(1(23))4},z_{(1(23))4})=(z_3-z_4,z_1-z_3,z_2-z_3), \nonumber \\
1((23)4):(x_{1((23)4)}, y_{1((23)4)},z_{1((23)4)})=(z_1-z_4,z_3-z_4,z_2-z_3), \nonumber \\
1(2(34)):(x_{1(2(34))}, y_{1(2(34))},z_{1(2(34))})=(z_1-z_4,z_2-z_4,z_3-z_4). \nonumber
\end{align}
\end{comment}

Set
\begin{align*}
T(x,y,z)&=\C[[y/x,\overline{y}/\overline{x}]]((z/y,\overline{z}/\overline{y},|z/y|^\R))
[x^{\pm}, y^\pm,\overline{x}^\pm,\overline{y}^\pm,|x|^\R,|y|^\R],\\
T^f (x,y,z)&=\C[[y/x,\overline{y}/\overline{x},z/y,\overline{z}/\overline{y}]]
[x^{\pm}, y^\pm,z^\pm,\overline{x}^\pm,\overline{y}^\pm,\z^\pm,|x|^\R,|y|^\R,|z|^\R],
\end{align*}
for the formal variables $x,y,z,\bar{x},\bar{y},\bar{z}$.
\begin{comment}
For each $A \in P_4$ which is not $S_4$-conjugate to $(12)(34)$,
set $T_A=T(x_A,y_A,z_A)$ and $T_A^f=T^f(x_A,y_A,z_A).$
For $A \in P_4$ which is $S_4$-conjugate to $(12)(34)$,
we will define $T_A$ and $T_A^f$ later.
In this section, for each $A\in P_4$,
we will define a map
\end{comment}

We first examine the case of $A={1(2(34))}$ in detail.
In this case, we first expand a function $\phi \in \Cor_4$
in the variable $z_{1(2(34))}=z_3-z_4$ (related to the innermost product $(34)$),
then in $y_{1(2(34))}=z_2-z_4$ (related to $2(34)$) and $x_{1(2(34))}=z_1-z_4$ 
(For the general rule for the change of variables, see Appendix).
The resulting formal power series is convergent in
some open region in $\{|x_{1(2(34))}|>|y_{1(2(34))}|>|z_{1(2(34))}|\}$,
which define a linear map
$$e_{1(2(34))}:\Cor_4 \rightarrow T(x_{1(2(34))},y_{1(2(34))},z_{1(2(34))})
$$
such that for $\phi(z_1,z_2,z_3,z_4) \in \Cor_4$, the formal power series $e_{1(2(34))}(\phi)$ is
absolutely convergent to the function $f(z_1,z_2,z_3,z_4)$ by taking the change of variables $(z_1-z_4,z_2-z_4,z_3-z_4)=(x_{1(2(34))},y_{1(2(34))},z_{1(2(34))})$.

\begin{comment}
Set
\begin{align*}
T(x,y,z)&=\C[[y/x,\overline{y}/\overline{x}]]((z/y,\overline{z}/\overline{y},|z/y|^\R))
[x^{\pm}, y^\pm,z^\pm,\overline{x}^\pm,\overline{y}^\pm,\z^\pm,|x|^\R,|y|^\R,|z|^\R],\\
T^f (x,y,z)&=\C[[y/x,\overline{y}/\overline{x},z/y,\overline{z}/\overline{y}]]
[x^{\pm}, y^\pm,z^\pm,\overline{x}^\pm,\overline{y}^\pm,\z^\pm,|x|^\R,|y|^\R,|z|^\R],
\end{align*}
for the formal variables $x,y,z,\bar{x},\bar{y},\bar{z}$.
We will define a linear map
$$e_{1(2(34))}:\Cor_4 \rightarrow T(x_{1(2(34))},y_{1(2(34))},z_{1(2(34))})
$$
such that for $\phi(z_1,z_2,z_3,z_4) \in \Cor_4$, the formal power series $e_{1(2(34))}(\phi)$ is
absolutely convergent in the open region in $\{|x_{1(2(34))}|>|y_{1(2(34))}|>|z_{1(2(34))}|\}$ to the function $f(z_1,z_2,z_3,z_4)$ by taking the change of variables $(z_1-z_4,z_2-z_4,z_3-z_4)=(x_{1(2(34))},y_{1(2(34))},z_{1(2(34))})$.
The point is that the coefficient of this function is completely determined by
the expansion $j(1,f)$, the expansion of $f \in \F$ around $0$, if $\phi$ is of the form \ref{eq_four},
$$\phi(z_1,z_2,z_3,z_4)
= \Pi_{1\leq i<j \leq 4} (z_i-z_j)^{\al_{ij}}(\bar{z}_i-\bar{z}_j)^{\be_{ij}} f\circ \xi(z_1,z_2,z_3,z_4).$$
\end{comment}
For a function in $\Cor_4^f \subset \Cor_4$, the map $e_{1(2(34))}$ is defined as a binomial expansion
in the domain $|x_{1(2(34))}|>>|y_{1(2(34))}|>>|z_{1(2(34))}|$.
For example, $$|z_1-z_3|^r\mapsto|x_{1(2(34))}-z_{1(2(34))}|^r=
|x_{1(2(34))}|^{r}\sum_{i,j \geq 0} (-1)^i\binom{r}{i}\binom{r}{j}
x_{1(2(34))}^{-i}\overline{x}_{1(2(34))}^{-j}z_{1(2(34))}^i \overline{z}_{1(2(34))}^j,
$$
where the right hand side is absolutely convergent to $|z_1-z_3|^r$ in the open domain $|z_1-z_4|>|z_3-z_4|$
and is an element of $T^f(x_{1(2(34))},y_{1(2(34))},z_{1(2(34))})$.
We denote this map by $e_{1(2(34))}^f: \Cor_4^f \rightarrow T^f(x_{1(2(34))},y_{1(2(34))},z_{1(2(34))}).$
Let $f\in \F$.
For $f\circ \xi \in \Cor_4$, expand $\xi=\frac{(z_1-z_2)(z_3-z_4)}{(z_1-z_3)(z_2-z_4)}$
in  $|x_{1(2(34))}|>>|y_{1(2(34))}|>>|z_{1(2(34))}|$, e.g.,
\begin{align}
\xi=\frac{(z_1-z_2)(z_3-z_4)}{(z_1-z_3)(z_2-z_4)}
\mapsto \frac{(x_{1(2(34))}-y_{1(2(34))})z_{1(2(34))}}{y_{1(2(34))}(x_{1(2(34))}-z_{1(2(34))})}
&=z_{1(2(34))}/y_{1(2(34))}(1-y_{1(2(34))}/x_{1(2(34))}) \nonumber \\
&\times \sum_{i \geq0}(z_{1(2(34))}/x_{1(2(34))})^{i}, \nonumber
\end{align}
which is equal to $e_{1(2(34))}^f(\frac{(z_1-z_2)(z_3-z_4)}{(z_1-z_3)(z_2-z_4)})$, defined above.
Define the map $s_{1(2(34))}:\C((p,\bar{p},|p|^\R))\rightarrow T(x_{1(2(34))},y_{1(2(34))},z_{1(2(34))})$
by substituting  $p=e_{1(2(34))}^f(\frac{(z_1-z_2)(z_3-z_4)}{(z_1-z_3)(z_2-z_4)})$
into $\sum_{r,s\in \R}a_{r,s}p^r\bar{p}^s$.
Since $p=e_{1(2(34))}^f(\frac{(z_1-z_2)(z_3-z_4)}{(z_1-z_3)(z_2-z_4)})$
is an element of $$z_{1(2(34))}/y_{1(2(34))} \C[[y_{1(2(34))}/x_{1(2(34))},z_{1(2(34))}/y_{1(2(34))}]],$$
by the definition of $\C((p,\p,|p|^\R))$, $s_{1(2(34))}$ is well-defined.
The following lemma is obvious:
\begin{lem}\label{ring_expansion}
$T(x,y,z)$ is a module over the $\C$-algebra $T^f(x,y,z)$.
\end{lem}
Then, $e_{1(2(34))}: \Cor_4 \rightarrow T(x_{1(2(34))},y_{1(2(34))},z_{1(2(34))})$
is defined by $$e_{1(2(34))}(\phi \cdot f \circ \xi))=e_{1(2(34))}^f(\phi) s_{1(2(34))}(j(p,f))$$
for $\phi \in \Cor_4^f$ and $f \in \F$,
which is absolutely convergent to $\phi \cdot f\circ \xi$ in $|x_{1(2(34))}|>>|y_{1(2(34))}|>>|z_{1(2(34))}|$.
Here, $|x_{1(2(34))}|>>|y_{1(2(34))}|>>|w_{1(2(34))}|$ means 
some non-empty open domain in $\{(z_1,z_2,z_3,z_4)\;|\; |z_1-z_4|>|z_2-z_4|>|z_3-z_4| \}$.
The map $e_{1(2(34))}$ is independent of a choice of the decomposition $\phi \cdot f\circ\xi \in \Cor_4$,
since the coefficient of a convergent series is unique.

Second, we give the definition of $e_A$ in case of $A=1((23)4)$.
We consider the following change of variables;
$$(z_1-z_4,z_3-z_4,z_2-z_3)\mapsto (x_{1((23)4)}, y_{1((23)4)},z_{1((23)4)}).$$
The map $e_{1((23)4)}^f: \Cor_4^f \rightarrow T^f(x_{{1((23)4)}},y_{{1((23)4)}},z_{1((23)4)})$
 is defined by the binomial expansion
in the domain $|x_{1((23)4)}|>>|y_{1((23)4)}|>>|z_{1((23)4)}|$ as above.
The expansion of $\xi=\frac{(z_1-z_2)(z_3-z_4)}{(z_1-z_3)(z_2-z_4)}$
is equal to
\begin{align}
\xi=\frac{(x_{1((23)4)}-y_{1((23)4)}+z_{1((23)4)})y_{1((23)4)}}{(x_{1((23)4)}-y_{1((23)4)})(z_{1((23)4)}+y_{1((23)4)})}
&={(x_{1((23)4)}-y_{1((23)4)}+z_{1((23)4)})y_{1((23)4)}}\nonumber \\
&\times \sum_{i,j \geq 0} (-1)^j x_{1((23)4)}^{-i-1}y_{1((23)4)}^{i-j-1}z_{1((23)4)}^j. \nonumber
\end{align}
Importantly, this is not in 
$$z_{1((23)4)}/y_{1((23)4)} \C[[y_{1((23)4)}/x_{1((23)4)},z_{1((23)4)}/y_{1((23)4)}]].$$
Thus, the substitution of this into $j(p,f)$ is not well-defined for $f \in \F$.
The reason is that when $z_{1((23)4)}=z_2-z_3$ goes to zero, $\xi$ goes to $1$.
Thus, we have to consider the expansion of $f$ around $1$
and the expansion of $1-\xi=\frac{(z_1-z_4)(z_2-z_3)}{(z_1-z_3)(z_2-z_4)}$.
Then, the expansion is
\begin{align}
1-\xi&=\frac{x_{1((23)4)}z_{1((23)4)}}{(x_{1((23)4)}-y_{1((23)4)})(y_{1((23)4)}+z_{1((23)4)})}\nonumber \\
&= z_{1((23)4)}/y_{1((23)4)}
\sum_{i,j \geq 0} (-1)^j (y_{1((23)4)}/x_{1((23)4)})^{i}(z_{1((23)4)}/y_{1((23)4)})^j, \nonumber
\end{align}
which is an element of $z_{1((23)4)}/y_{1((23)4)}\C[[y_{1((23)4)}/x_{1((23)4)},z_{1((23)4)}/y_{1((23)4)}]].$
Define the map $$s_{1((23)4)}:\C((p,\bar{p},|p|^\R))\rightarrow T(x_{{1((23)4)}},y_{{1((23)4)}},z_{1((23)4)})$$
by substituting  $p=e_{1((23)4)}^f(1-\xi)$
into $\sum_{r,s\in \R}a_{r,s}p^r\bar{p}^s$.
Then, $e_{1((23)4)}: \Cor_4 \rightarrow T(x_{{1((23)4)}},y_{{1((23)4)}},z_{1((23)4)})$
is defined by $$e_{1((23)4)}(\phi f\circ \xi)=e_{1((23)4)}^f(\phi) s_{1((23)4)}(j(1-p,f))$$
for $\phi \in \Cor_4^f$ and $f \in \F$.
%Here, $(01)$ is a permutation of $\{0,1,\infty \}$, by Lemma \ref{orbit_xi}, which corresponds to $1-\xi$.
Since $j({1-p},f)$ is absolutely convergent to $f(1-p)$,
$s_{1((23)4)}(j({1-p},f)$ is absolutely convergent to $f\circ \xi$
in $|x_{{1((23)4)}}|>>|y_{{1((23)4)}}|>>|z_{{1((23)4)}}|$.

\begin{rem}
\label{rem_choice}
In the domain $|x_{1((23)4)}|>>|y_{1((23)4)}|>>|z_{1((23)4)}|$, the difference
$|z_{1((23)4)}|=|z_2-z_3|$ is assumed to be small. Then, by Remark \ref{rem_hypersurface},
$\xi \mapsto 1$. Thus, we have to consider the expansion of $1- \xi$ or $1-\xi^{-1}$.
If we choose $1-\xi^{-1}$,
we substitute $e_{1((23)4)}^f(1-\xi^{-1})$ into $j(\frac{1}{1-p} ,f).$
%We remark that the inverse in $j_{(01\infty)^{-1}}(f)$ is needed.
Then, the resulting series is, roughly, convergent to
$f(\frac{1}{1-(1-p^{-1})})=f(p)$.
%A simple explanation for this fact is, for a convergent series ${(0\infty1)^{-1}}f(p)=f(1-p^{-1})=\sum_n a_n p^n$,
%$\sum_n a_n (\frac{1}{1-p})^n$ is, roughly convergent to, $f(1-\frac{1}{\frac{1}{1-p}})=f(p)$.
Hence, both of the choices $1-\xi$ and ${1-\xi^{-1}}$ give the same result.
\end{rem}

The last example is the case of $A=(12)(34)$,
which has a special property.
%We end this subsection by noting that the expansion $e_{(12)(34)}$ has a special property.
We consider the following change of variables:
$$(z_2-z_4,z_1-z_2,z_3-z_4)\mapsto(x_{(12)(34)}, y_{(12)(34)},z_{(12)(34)}).$$
Then, 
\begin{align*}
z_1-z_3 &\mapsto x_{(12)(34)}+y_{(12)(34)}-z_{(12)(34)}, \\
z_1-z_4 &\mapsto x_{(12)(34)}+y_{(12)(34)}, \\
z_2-z_3 &\mapsto x_{(12)(34)}-z_{(12)(34)}.
\end{align*}
Since there is no $y_{(12)(34)}-z_{(12)(34)}$ or $y_{(12)(34)}+z_{(12)(34)}$ in the above list,
the image of $e_{(12)(34)}^f: \Cor_4^f \rightarrow T^f(x_{(12)(34)},y_{(12)(34)},z_{(12)(34)})$
is in $$\C[[y/x,z/x, \overline{y}/\overline{x},\overline{z}/\overline{x}]]
[x^{\pm}, y^\pm,z^\pm,\overline{x}^\pm ,\overline{y}^\pm,\overline{z}^\pm,|x|^\R,|y|^\R,|z|^\R],$$
that is, both $y_{(12)(34)}$ and $z_{(12)(34)}$ are bounded below,
which happens only for the $S_4$-conjugate elements of ${(12)(34)}$.
Set
\begin{align*}
T_{A}^f(x,y,z)
= \C[[y/x,z/x, \overline{y}/\overline{x},\overline{z}/\overline{x}]]
[x^{\pm}, y^\pm,z^\pm,\overline{x}^\pm ,\overline{y}^\pm,\overline{z}^\pm,|x|^\R,|y|^\R,|z|^\R],
\end{align*}
for any $A \in P_4$ which is $S_4$-conjugate of $(12)(34)$.
Then, the binomial expansion defines $e_{(12)(34)}:\Cor_4 \rightarrow T_{(12)(34)}^f(x_{(12)(34)},y_{(12)(34)},z_{(12)(34)})$
and
\begin{align*}
e_{(12)(34)}(\xi)&=\frac{y_{(12)(34)}z_{(12)(34)}}{x_{(12)(34)}(y_{(12)(34)}+x_{(12)(34)}-z_{(12)(34)})}\\
&=\frac{z_{(12)(34)}}{x_{(12)(34)}}\frac{y_{(12)(34)}}{x_{(12)(34)}}
\sum_{n \geq0}\sum_{i=0}^n (-1)^{n-i}\binom{n}{i}\Bigl(\frac{y_{(12)(34)}}{x_{(12)(34)}}\Bigr)^{n-i}\Bigl(\frac{z_{(12)(34)}}{x_{(12)(34)}}\Bigr)^i.
%x_{(12)(34)}^{-n} y_{(12)(34)}^{n-i}z_{(12)(34)}^{i}.
\end{align*}
Let $T'(x,y,z)$ be
a subspace of $\C((y/x,z/x,\bar{y}/\x,\z/\x,|y/x|^\R,|z/x|^\R))$
spanned by
$$\sum_{s_1,s_2,\s_1,\s_2 \in \R}c_{s_1,\s_1,s_2,\s_2}
\Bigl(\frac{y}{x}\Bigr)^{s_1}\Bigl(\frac{\y}{\x}\Bigr)^{\s_1}
\Bigl(\frac{z}{x}\Bigr)^{s_2}\Bigl(\frac{\z}{\x}\Bigr)^{\s_2}$$
such that:
\begin{enumerate}
\item
$c_{s_1,\s_1,s_2,\s_2}=0$ unless $s_1-\s_1, s_2-\s_2 \in \Z$;
\item
There exists $N \in \R$ such that 
$c_{s_1,\s_1,s_2,\s_2}=0$ unless
$s_1,\s_1,s_2,\s_2\geq N$;
\item
$c_{s_1,\s_1,s_2,\s_2}=0$ unless $s_1-s_2, \s_1-\s_2 \in \Z$.
\end{enumerate}

Set 
$$T_{A}(x,y,z)=T'(x,y,z)
[x^\pm,y^\pm,z^\pm,\x^\pm,\bar{y}^\pm,\z^\pm,
|x|^\R,|y|^\R,|z|^\R],
$$
for any $A \in P_4$ which is $S_4$-conjugate to $(12)(34)$.
Define the map $$s_{(12)(34)}:\C((p,\bar{p},|p|^\R))\rightarrow T_{(12)(34)}(x_{(12)(34)},y_{(12)(34)},z_{(12)(34)})$$
by substituting  $p=e_{(12)(34)}^f(\xi)$
into $\sum_{r,s\in \R}a_{r,s}p^r\bar{p}^s$.
The following lemma is obvious:
\begin{lem}\label{expansion_channel}
For any $A \in P_4$ which is $S_4$-conjugate of $(12)(34)$,
$T_{A}^f(x,y,z)$ is a subspace of $T_{A}(x,y,z)$
and $T_{A}(x,y,z)$ is a module over the $\C$-algebra $T_{A}^f(x,y,z)$.
\end{lem}
Then, $e_{(12)(34)}:\Cor_4 \rightarrow T_{(12)(34)}(x_{(12)(34)},y_{(12)(34)},z_{(12)(34)})$ can be defined as above.

Now, we will consider the general case.
For standard elements of $P_4$, consider the following change of variables:
\begin{align}
(12)(34):(x_{(12)(34)}, y_{(12)(34)},z_{(12)(34)})=(z_2-z_4,z_1-z_2,z_3-z_4), \nonumber \\
((12)3)4:(x_{((12)3)4}, y_{((12)3)4},z_{((12)3)4})=(z_3-z_4,z_2-z_3,z_1-z_2), \nonumber \\
(1(23))4:(x_{(1(23))4}, y_{(1(23))4},z_{(1(23))4})=(z_3-z_4,z_1-z_3,z_2-z_3), \nonumber \\
1((23)4):(x_{1((23)4)}, y_{1((23)4)},z_{1((23)4)})=(z_1-z_4,z_3-z_4,z_2-z_3), \nonumber \\
1(2(34)):(x_{1(2(34))}, y_{1(2(34))},z_{1(2(34))})=(z_1-z_4,z_2-z_4,z_3-z_4). \nonumber
\end{align}
For general $A\in P_4$, regard the above $x_A,y_A,z_A$ as polynomials of $(z_1,z_2,z_3,z_4)$.
Then, $S_4$ naturally acts on them. For a standard element $A$ and $\sigma \in S_4$,
set $(x_{\sigma  A},y_{\sigma A},z_{\sigma A})=(\sigma x_A,\sigma y_A,\sigma z_A).$
For example, 
\begin{align*}
(x_{2((43)1)},y_{2((43)1)},z_{2((43)1)})=(124)\cdot(z_1-z_4,z_3-z_4,z_2-z_3)=
(z_2-z_1,z_3-z_1,z_4-z_3).
\end{align*}
For each $A \in P_4$ which is not $S_4$-conjugate to $(12)(34)$,
set $T_A(x_A,y_A,z_A)=T(x_A,y_A,z_A)$ and $T_A^f(x_A,y_A,z_A)=T^f(x_A,y_A,z_A).$
Then, we have the binomial expansion convergent in $|x_A|>>|y_A|>>|z_A|$, $e_A^f:\Cor_4^f \rightarrow  T_A^f(x_A,y_A,z_A)$ as above.
Define the map $T_\sigma:T(x_A,y_A,z_A) \rightarrow T(x_{\sigma A},y_{\sigma A},z_{\sigma A})$
by $(x_A,y_A,z_A,\x_A,\bar{y}_A,\z_A) \mapsto (x_{\sigma A},y_{\sigma A},z_{\sigma A},\x_{\sigma A},\bar{y}_{\sigma A},\z_{\sigma A})$.
Then, we have:
\begin{lem}
\label{free_symmetry}
For any $\sigma \in S_4$ and $A \in P_4$ and $\phi \in \Cor_4^f$,
$T_\sigma \circ e_A (\phi)= e_{\sigma A} (\sigma \phi)$.
\end{lem}

By Remark \ref{rem_choice}, we need to choose an appropriate choice of 
$\xi, \frac{\xi}{\xi-1}, 1-\xi, 1-\xi^{-1}, \xi^{-1}, \frac{1}{1-\xi}$ for each $A \in P_4$, when we expand a function $f \in \F$.
%, which are identified as elements of $\Aut (0,1,\infty)$, by Remark \ref{rem_hypersurface}.
The map $\tau: P_4 \rightarrow S_4$ is chosen for standard elements as
\begin{align}
(12)(34) &\mapsto 1, \nonumber \\
((12)3)4 &\mapsto 1, \nonumber \\
(1(23))4 &\mapsto (13),  \nonumber \\
1((23)4) &\mapsto (13),  \nonumber \\
1(2(34)) &\mapsto 1,  \nonumber
\end{align}
and for $\si A \in P_4$ as $\tau(\si A)=\si \cdot \tau(A)$,
where $\si \in S_4$ and $A$ is a standard element of $P_4$.
Then, for standard element $A$, the expansion of $\tau(A)\cdot \xi$ is given by:
\begin{align}
\xi&=\frac{y_{(12)(34)}z_{(12)(34)}}{x_{(12)(34)}(y_{(12)(34)}+x_{(12)(34)}-z_{(12)(34)})}
=z_{(12)(34)}y_{(12)(34)}
\sum_{n \geq0}\sum_{i=0}^n (-1)^{n-i}\binom{n}{i}x_{(12)(34)}^{-n-2} y_{(12)(34)}^{n-i}z_{(12)(34)}^{i},
 \nonumber \\
\xi&=\frac{x_{((12)3)4}z_{((12)3)4}}{(z_{((12)3)4}+y_{((12)3)4})(y_{((12)3)4}+x_{((12)3)4})}
=z_{((12)3)4}/y_{(12)(34)}
\sum_{i,j \geq0}(-1)^{i+j} x_{(12)(34)}^{-i} y_{(12)(34)}^{i-j}z_{(12)(34)}^{j},
 \nonumber \\
1-\xi&=\frac{(x_{(1(23))4}+y_{(1(23))4})z_{(1(23))4}}{y_{(1(23))4}(x_{(1(23))4}+z_{(1(23))4})}
=z_{(1(23))4}/y_{(1(23))4} (1-y_{(1(23))4}/x_{(1(23))4})\sum_{i \geq0}(-1)^{i} x_{(1(23))4}^{-i}z_{(1(23))4}^{i}, \nonumber \\
1-\xi&=\frac{x_{1((23)4)}z_{1((23)4)}}{(x_{1((23)4)}-y_{1((23)4)})(y_{1((23)4)}+z_{1((23)4)})}
= z_{1((23)4)}/y_{1((23)4)}
\sum_{i,j \geq 0} (-1)^j x_{1((23)4)}^{-i}y_{1((23)4)}^{i-j}z_{1((23)4)}^j, \nonumber \\
\xi&=\frac{(x_{1(2(34))}-y_{1(2(34))})z_{1(2(34))}}{y_{1(2(34))}(x_{1(2(34))}-z_{1(2(34))})}
=z_{1(2(34))}/y_{1(2(34))}(1-y_{1(2(34))}/x_{1(2(34))})\sum_{i \geq0}x_{1(2(34))}^{-i}z_{1(2(34))}^{i}.
 \nonumber
\end{align}

Thus, for $\si \in S_4$, by Lemma \ref{free_symmetry},
$e_{\si A}^f(\tau(\si A)\cdot \xi)=e_{\si A}^f(\si \tau(A)\cdot \xi)=T_\si e_A^f(\tau(A)\cdot \xi),$
which implies that $e_{\si A}^f(\tau(\si A)\cdot \xi) \in z_{\si A}/x_{\si A}y_{\si A}/x_{\si A}  \C[[y_{\si A}/x_{\si A},z_{\si A}/x_{\si A}]]$
if $A$ is $S_4$-conjugate to $(12)(34)$
and $e_{\si A}^f(\tau(\si A)\cdot \xi) \in z_{\si A}/y_{\si A}\C[[y_{\si A}/x_{\si A},z_{\si A}/y_{\si A}]]$ otherwise.

\begin{comment}
Since the right hand sides are elements of 

we have a well-defined map $e_A:\Cor_4 \rightarrow T(x_A,y_A,z_A)$.

Thus, if $A=(12)(34), ((12)3)4, 1(2(34))$, then $e_{\sigma A}(\sigma \xi)=\sigma e_A(\xi) \in z_{\sigma A}\C[[y_{\sigma A}/x_{\sigma A},
z_{\sigma A}/y_{\sigma A}]][x_{\sigma A}^{\pm}, y_{\sigma A}^\pm]$
and if $A=1((23)4), (1(23))4$, then  $e_{\sigma A}(\sigma (1-\xi))=\sigma e_A(1-\xi) \in z_{\sigma A}\C[[y_{\sigma A}/x_{\sigma A},
z_{\sigma A}/y_{\sigma A}]][x_{\sigma A}^{\pm}, y_{\sigma A}^\pm]$.
\end{comment}
Define the map $$s_A:\C((p,\bar{p},|p|^\R)) \rightarrow T_A(x_A,y_A,w_A)$$
by substituting  $p=e_A^f(\tau(A)\cdot \xi)$
into $\sum_{r,s\in \R}a_{r,s}p^r\bar{p}^s$.
Then, the composition of $j(\tau(A)\cdot p,- ): \F \rightarrow \C((p,\bar{p},|p|^\R))$ and $s_A$
define the map $e_{ A}: \Cor_4 \rightarrow T_{ A}(x_{A},y_{A},z_{A})$.
Then, we have:
\begin{prop}
\label{expansion}
For $A\in P_4$ and $\phi \in \Cor_4$, 
$e_A(\phi) \in T_A(x_A,y_A,z_A)$ is absolutely convergent to $\phi$ in
some non-empty domain in $|x_A|>|y_A|>|z_A|$ and
$T_\si e_A (\phi)= e_{\si A}(\si\cdot \phi) $.
\end{prop}
\begin{proof}
Let $f \in \F$.
Then, $e_A(f\circ \xi)= s_A(j(\tau(A)\cdot p, f))=s_A(j(t_{\tau(A)}^{-1}.p, f))$ is, by Lemma \ref{orbit_xi},
absolutely convergent to $$f(t_{\tau(A)}.(\tau(A)\cdot \xi))=f(t_{\tau(A)}.t_{\tau(A)}^{-1}.\xi)=f(\xi).$$
Furthermore, for $\si \in S_4$,
since $T_\si e_A^f(\tau(A)\xi)=e_{\si A}f(\tau(\si A)\xi)$, by Lemma \ref{F_covariance},
we have 
\begin{align*}
e_{\si A}(\si \cdot f\circ\xi)= s_{\si A}(j(\si \cdot \tau(A)\cdot p, \si  \cdot f))
=T_\si s_A j(\tau(A)\cdot p, f)=T_\si e_A(f\circ \xi)
\end{align*}
\end{proof}

The symbols $e_A$ and $x_A,y_A,z_A$ is convenient to express the various expansions systematically.
However, in practice it is cumbersome to work in this notation.
For $A=(1(23))4$, $x_A,y_A,z_A$ represents $z_3-z_4$, $z_1-z_3$, $z_2-z_3$.
Thus, we introduce new formal variables $z_{34},z_{13},z_{23}$ and 
set $z_{34}=x_{(1(23))4},z_{13}=y_{(1(23))4},z_{23}=z_{(1(23))4}$.
Then, the map $e_A$ is a map from $\Cor_4$ to $T(z_{34},z_{13},z_{23})$.
For, $\phi \in \Cor_4$, we also denote $e_{(1(23))4}(\phi)$ by $\phi|_{|z_{34}|>|z_{13}|>|z_{23}|}$.
The notations $x_A,y_A,z_A$ and $z_{ij}$ will be used interchangeably in this paper.

We remark that
we can define the expansion maps $\Cor_2, \Cor_3$ in similar way,
e.g., $\Cor_3 \rightarrow T(z_{13},z_{23}), \phi \mapsto \phi|_{|z_{13}|>|z_{23}|}$.

\subsection{Relations among expansions}\label{sec_relation_expansion}
In this section, we establish relations among expansions $e_A$ for $A\in P_4$. For $A,B \in P_4$,
define the map 
$${T'}_A^B:T(x_A,y_A,z_A) \rightarrow T(x_B,y_B,z_B)$$
by sending $(x_A,y_A,z_A,\x_A,\bar{y}_A,\z_A)$ into $(x_B,y_B,z_B,\x_B,\bar{y}_B,\z_B)$.

We first construct a map 
$T_{1((23)4)}^{1(4(23))}: T(x_{1((23)4)},y_{1((23)4)},z_{1((23)4)})
 \rightarrow T(x_{1(4(23))},y_{1(4(23))},z_{1(4(23))})$
such that
$T_{1((23)4)}^{1(4(23))}\circ e_{1((23)4)}=e_{1(4(23))}$.
By definition,
there should be a relation
\begin{align*}
x_{1(4(23))}&=z_1-z_3=x_{1((23)4)}-y_{1((23)4)}, \\
y_{1(4(23))}&=z_4-z_3=-y_{1((23)4)}, \\
z_{1(4(23))}&=z_2-z_3=z_{1((23)4)}.
\end{align*}
Thus,
define $T_{1((23)4)}^{1(4(23))}: T(x_{1((23)4)},y_{1((23)4)},z_{1((23)4)})
 \rightarrow T(x_{1(4(23))},y_{1(4(23))},z_{1(4(23))})$
by
$$T_{1((23)4)}^{1(4(23))}={T'}_{1((23)4)}^{1(4(23))}
\exp(-y_{1((23)4)}\frac{d}{dx_{1((23)4)}})\exp(-\bar{y}_{1((23)4)}\frac{d}{d\bar{x}_{1((23)4)}})\lim_{y_{1((23)4)}\mapsto -y_{1((23)4)}}.
$$
We remark that the operator $\exp(y\frac{d}{dx}):\C((y/x))\rightarrow \C((y/x))$ changes the variable $x$ into
$x+y$, that is,
\begin{align*}
\exp(y\frac{d}{dx})\sum_{n}a_n x^{-n}y^n
=\sum_{n}\sum_{i \geq 0}  \binom{-n}{i} a_n x^{-n-i}y^{n+i}.
\end{align*}
For $\si \in S_4$,
define $T_{\si{1((23)4)}}^{\si{1(4(23))}}:
T(x_{\si1((23)4)},y_{\si1((23)4)},z_{\si1((23)4)})
 \rightarrow T(x_{\si1(4(23))},y_{\si1(4(23))},z_{\si1(4(23))})$
by
$$T_{\si{1((23)4)}}^{\si{1(4(23))}}=T_\si \circ T_{1((23)4)}^{1(4(23))} \circ
T_{\si^{-1}}.$$
Then, we have:
\begin{lem}
\label{ass_1}
For any $\si \in S_4$,
$T_{\si{1((23)4)}}^{\si{1(4(23))}}\circ e_{\si{1((23)4)}}=e_{\si{1(4(23))}}$.
\end{lem}

Second, we construct 
$T_{(1(23))4}^{4(1(23))}: T(x_{(1(23))4},y_{(1(23))4},z_{(1(23))4})
 \rightarrow T(x_{4(1(23))},y_{4(1(23))},z_{4(1(23))})$
and $T_{((12)3)4}^{4(3(12))}: T(x_{((12)3)4},y_{((12)3)4},z_{((12)3)4})
 \rightarrow T(x_{4(3(12))},y_{4(3(12))},z_{4(3(12))})$.
Similarly to the above,
by the relation
\begin{align*}
x_{4(1(23))}&=z_4-z_3=-x_{(1(23))4}, \\
y_{4(1(23))}&=z_1-z_3=y_{(1(23))4}, \\
z_{4(1(23))}&=z_2-z_3=z_{(1(23))4},
\end{align*}
and
\begin{align*}
x_{4(3(12))}&= z_4-z_2= -x_{((12)3)4}-y_{((12)3)4}, \\
y_{4(3(12))}&=z_3-z_2=-y_{((12)3)4}, \\
z_{4(3(12))}&=z_1-z_2=z_{((12)3)4}.
\end{align*}

Thus, for $\si \in S_4$,
set \begin{align*}
T_{(1(23))4}^{4(1(23))}&={T'}_{(1(23))4}^{4(1(23))}
\lim_{x_{(1(23))4}\to -x_{(1(23))4}} \\
T_{((12)3)4}^{4(3(12))}&={T'}_{((12)3)4}^{4(3(12))}
\lim_{(x_{((12)3)4},y_{((12)3)4})\to (-x_{((12)3)4},-y_{((12)3)4})}
\exp{y_{((12)3)4}\frac{d}{dx_{((12)3)4}}}
\exp{\bar{y}_{((12)3)4}\frac{d}{d\bar{x}_{((12)3)4}}}
\end{align*}
and
\begin{align*}
T_{\si(1(23))4}^{\si4(1(23))}&=
T_\si \circ T_{(1(23))4}^{4(1(23))} \circ T_{\si^{-1}},\\
T_{\si{((12)3)4}}^{\si 4(3(21))}&=
T_\si \circ T_{((12)3)4}^{4(3(21))} \circ T_{\si^{-1}}.
\end{align*}
Then, we have
\begin{lem}
\label{ass_2}
For any $\si \in S_4$,
$T_{\si{(1(23))4}}^{\si{4(1(23))}}\circ e_{\si{(1(23))4}}=e_{\si{4(1(23))}}$
and
$T_{\si{((12)3)4)}}^{\si{4(3(12))}}\circ e_{\si{((12)3)4)}}=e_{\si{4(3(12))}}$.
\end{lem}
By the above lemmas, all expansions except for $e_{(12)(34)}$ can be transformed into
$e_{1(2(34))}$. For $e_{(12)(34)}$,
 the relation is
\begin{align*}
x_{2(1(34))}&=z_2-z_4  =x_{(12)(34)}, \\
y_{2(1(34))}&=z_1-z_4=y_{(12)(34)}+x_{(12)(34)}, \\
z_{2(1(34))}&=z_3-z_4=z_{(12)(34)}.
\end{align*}
It seems that we only need to consider the operator $\exp(-x_{(12)(34)}\frac{d}{dy_{(12)(34)}})$,
However, it is not well-defined, since, informally,
$\exp(x\frac{d}{dy})(x-y)^{-1}=-y$ and $\exp(x\frac{d}{dy})\sum_{n \geq 0}x^{-1-n}y^n=
\sum_{n \geq 0}\sum_{k=0}^n x^{-1-n}(y+x)^n$ whose coefficient of $x^{-1}$ is $1+1+\dots$.
We will discuss this issue in Section \ref{sec_differential}.

Finally, we transform $e_{1(2(34))}$ into $e_{1(2(43))}$.
In this case,
the relation is
\begin{align*}
x_{1(2(34))}&=z_1-z_4  =x_{1(2(43))}-z_{1(2(43))}, \\
y_{1(2(34))}&=z_2-z_4=y_{1(2(43))}-z_{1(2(43))}, \\
z_{1(2(34))}&=z_3-z_4=-z_{1(2(43))}.
\end{align*}
Set $$T_{1(2(34))}^{1(2(43))}={T'}_{1(2(34))}^{1(2(43))}\lim_{z \to -z}\exp(z(\frac{d}{dy}+\frac{d}{dx}))
\exp(\z(\frac{d}{d\bar{y}}+\frac{d}{d\bar{x}})).$$
Then,
\begin{lem}
\label{formal_skew}
$T_{1(2(34))}^{1(2(43))} e_{1(2(34))}=e_{1(2(43))}$.
\end{lem}
\begin{rem}
We remark that the existence of this transformation reflects the fact that
both $e_{1(2(34))}$ and $e_{1(2(43))}$ correspond to the expansion
of a function in $\F$ at the same point, $0 \in \C P^1$.
The Borcherds identity, an axiom of a vertex algebra, can be derived from a
property of holomorphic functions, the Cauchy integral formula,
which relate the expansions at the different points.
However, for a real analytic function in $\F$,
it seems difficult to compare the expansions at the different points.
\end{rem}

\subsection{Formal Differential Equation and Convergence}\label{sec_differential}
Since $\xi:X_4 \rightarrow \CPm$ is $\mathrm{PSL}_2 \C$ invariant,
the image of an expansion $e_A$ satisfies differential equations.
We first consider the formal differential equations for $e_{1(2(34))}$.
We abbreviate $(x_{1(2(34))},y_{1(2(34))},z_{1(2(34))})$ to $(x,y,z)$.
Set
\begin{align*}
D_{1}&=x^2d/dx+y^2d/dy+z^2d/dz, \\
\D_{1}&=\x^2 d/d\x+\y^2 d/d\y+\z^2 d/d\z, \\
D_{0}&=xd/dx+yd/dy+z d/dz, \\
\D_{0}&=\x d/d\x+\y d/d\y+\z d/d\z,
\end{align*}
which acts on $T(x,y,z)$
and
Set $S(x,y,z) = \cap_{i\in \{0,1\}} (\ker D_i \cap \ker \D_i)$,
which is the space of the formal solutions of the differential equations.
The following lemma follows from the $\mathrm{PSL}_2\C$-invariance of $\xi$:
\begin{lem}
The image of $s_{1(2(34))}:\C((p,\p,|p|^\R)) \rightarrow T(x,y,z)$ is in $S(x,y,z)$.
\end{lem}
We will show that $s_{1(2(34))}:\C((p,\p,|p|^\R)) \rightarrow S(x,y,z)$ is an isomorphism of vector spaces.
Let $f(x,y,z) \in S(x,y,z)$. 
Since $D_0 f=0$ and $\D_0 f$, we can assume that 
\begin{align*}
f(x,y,z)=\sum_{n,m,r,s \in \R}a_{n,m,r,s}(y/x)^n(z/y)^m (\y/\x)^r(\z/\y)^s\\
\in \C[[y/x,\y/\x]]((z/x,\z/\x,|z/x|^\R))[(y/x)^\pm,(\y/\x)^\pm,|y/x|^\R].
\end{align*}
By $D_1 f=0$ and $\D_1 f=0$, we have
\begin{align*}
(n-m)a_{n,m,r,s}+(m-1)a_{n,m-1,r,s}-(n+1)a_{n+1,m,r,s}=0, \\
(r-s)a_{n,m,r,s}+(s-1)a_{n,m,r,s-1}-(r+1)a_{n,m,r+1,s}=0. \\
\end{align*}
By the definition of $T(x,y,z)$, 
$a_{n,m,r,s}=0$ if one of $n,m,r,s$ is sufficiently small.
Fix $r,s \in \R$ and let $n_0=\min \{n\;| a_{n,m,r,s} \neq 0 \tand m \in\R \}$.
Then, since $n_0 a_{n_0,m,r,s}=(n_0-1-m)a_{n_0-1,m,r,s}+(m-1)a_{n_0-1,m-1,r,s}=0$,
we have $n_0=0$.
If there exists $n,m,r,s \in \R$ such
that $n \in \R \setminus \Z$ and $a_{n,m,r,s}\neq 0$,
then either $a_{n-1,m,r,s}$ or $a_{n-1,m-1,r,s}$ is non-zero,
which contradicts the fact that $a_{n,m,r,s}=0$ for sufficiently small $n$.
Thus, $f(x,y,z)=\sum_{n,r \in \Z_{\geq 0}}\sum_{m,s \in \R}a_{n,m,r,s} (y/x)^n(z/y)^m (\y/\x)^r(\z/\y)^s.$
Define the map $v_{1(2(34))}:S(x,y,z) \rightarrow \C((p,\p,|p|^\R))$ by
$$f(x,y,z) \mapsto \sum_{m,s \in \R} a_{0,m,0,s} p^m \p^s,$$
which is a formal limit of $(x,y,z) \mapsto (\infty,1,p)$.
\begin{lem}
\label{differential_inverse}
The map $v_{1(2(34))}$ is inverse to $s_{1(2(34))}$.
\end{lem}
\begin{proof}
Recall that
$e_{1(2(34))} \xi = (1-y/x) z/y \sum_{m \geq0}(z/x)^{m},$
which implies that $v_{1(2(34))}(s_{1(2(34))}(p))=p$.
Thus, $v_{1(2(34))}$ is a left inverse of $s_{1(2(34))}$.
Hence, it suffices to show that $v_{1(2(34))}$ is injective.
Let $f(x,y,z)=\sum_{n,r \in \Z_{\geq 0}}\sum_{m,s \in \R}a_{n,m,r,s}(y/x)^n(z/y)^m (\y/\x)^r(\z/\y)^s \in S(x,y,z)$
satisfy $v_{1(2(34))}(f(x,y,z))=0$, that is, $a_{0,m,0,r}=0$ for any $m,r \in \R$.
Then, it is easy to show that $f=0$ by the above recurrence formula.
\end{proof}

\begin{cor}
\label{convergence_1}
Let $g \in S(x,y,z)$.
If there exists $f \in \F$ such that $v_{1(2(34))}(g)=j(p,f)$,
then $g = e_{1(2(34))}(f\circ \xi)$. In particular, $g$ is absolutely convergent
to a function in $\Cor_4$ in $|x|>>|y|>>|z|$.
\end{cor}

We second consider the formal differential equations for $e_{(12)(34)}$.
We abbreviate \\
$(x_{(12)(34)},y_{(12)(34)},z_{(12)(34)})$ to $(x,y,z)$ again.
Set
\begin{align*}
D_{1}'&=x^2d/dx+y^2d/dy+z^2d/dz, \\
\D_{1}'&=\x^2 d/d\x+\y^2 d/d\y+\z^2 d/d\z, \\
D_{0}'&=xd/dx+yd/dy+z d/dz, \\
\D_{0}'&=\x d/d\x+\y d/d\y+\z d/d\z,
\end{align*}
and
let $S'(x,y,z)$ be the space of the formal solutions  of $D_0', D_1',\D_0', \D_1'$ in
$T_{(12)(34)}(x,y,z)$.
%\C[[y/x,z/x, \overline{y}/\overline{x},\overline{z}/\overline{x}]]
%[x^{\pm}, y^\pm,z^\pm,\overline{x}^\pm ,\overline{y}^\pm,\overline{z}^\pm,|x|^\R,|y|^\R,|z|^\R].$$
Similarly to the above, we have $$s_{(12)(34)}: \C((p,\p,|p|^\R)) \rightarrow S'(x,y,z).$$
We will show that the map is an isomorphism of vector spaces by constructing the inverse map.
By $D_0'$ and $\D_0'$, any element of $S'(x,y,z)$ is of the form 
$\sum_{n,m,r,s\in \R}a_{n,m,r,s}(y/x)^n(z/x)^m (\y/\x)^r(\z/\x)^s.$
The map $v_{(12)(34)}: S'(x,y,z) \rightarrow \C((p^{1/2},\bar{p}^{1/2},|p|^\R))$
defined by
$$\sum_{n,m,r,s}a_{n,m,r,s}(y/x)^n(z/x)^m (\y/\x)^r(\z/\x)^s \mapsto
\sum_{n,m,r,s}a_{n,m,r,s} p^{\frac{n+m}{2}} \bar{p}^{\frac{r+s}{2}}$$
is well-defined by the definition of $T_{(12)(34)}(x,y,z)$.
%If one of $n,m,r,s$ is sufficiently small, then $a_{n,m,r,s}$ is zero,
%which implies that $v'_{(12)(34)}$ is well-defined.
Furthermore, since $e_{(12)(34)}(\xi)=(y/x) (z/x) (1+(z/x-y/z)+(z/x-y/z)^2+\dots)$, we have
$v_{(12)(34)} \circ s_{(12)(34)} (\xi)=p$, which implies that $e_{(12)(34)}$ is injective.
Thus, it suffices to show that $e_{(12)(34)}$ is surjective.
Let $g(x,y,z)=\sum_{n,m,r,s}a_{n,m,r,s}(y/x)^n(z/x)^m (\y/\x)^r(\z/\x)^s \in S'(x,y,z)$.
Since it is in the kernel of $D_1'$ and $\D_1'$,
\begin{align*}
(n-m)a_{n,m,r,s}+na_{n-1,m,r,s}+ma_{n,m-1,r,s}&=0, \\
(r-s)a_{n,m,r,s}+ra_{n,m,r-1,s}+sa_{n,m,r,s-1}&=0.
\end{align*}
\begin{lem}
\label{recur_int}
For non-zero $g(x,y,z) \in S(x,y,z)$.
Set $N_0 = \min \{\frac{n+m}{2}\;| a_{n,m,r,s}\neq 0 \text{ for some }r,s\in \R \}$.
and $R_0= \min \{ \frac{r+s}{2}\;| a_{n,2N_0-n, r,s}\neq 0 \text{ for some } n,r,s \in \R\}$.
Then, $a_{N_0,N_0,R_0,R_0} \neq 0$ and $a_{n,2N_0-n, r,2R_0-r}=0$ for any $n \neq N_0$ or $r \neq R_0$.
Furthermore, $N_0-R_0 \in \Z$.
\end{lem}
\begin{proof}
Let $n_0, r_0 \in \R$ satisfy $a_{n_0,2N_0-n_0, r_0,2R_0-r_0}\neq 0$.
Then, by the above recurrence formula,
$2(N_0-n_0)a_{n_0,2N_0-n_0, r,2R_0-r_0}=n_0a_{n_0-1,2N_0-n_0, r_0,2R_0-r_0}+(2N_0-n_0)a_{n_0,2N_0-n_0-1, r_0,2R_0-r_0}=0$. Thus, $n_0=N_0$ and $r_0=R_0$.
Since $g(x,y,z) \in T_{(12)(34)}(x,y,z),$
$a_{n,m,r,s} =0$ if $n-r \notin \Z$ or $m-s \notin \Z$, which implies that $N_0-R_0 \in \Z$.
\end{proof}

By the above lemma, $p^{N_0} \p^{R_0} \in \C[p,\p,|p|^\R]$.
Thus, we can replace $g(x,y,z)$ by $g(x,y,z)- e_{(12)(34)}(a_{N_0,N_0,R_0,R_0} p^{N_0} \p^{R_0})$,
that is, $g(x,y,z)- e_{(12)(34)}(a_{N_0,N_0,R_0,R_0} p^{N_0} \p^{R_0}) \in S'(x,y,z)$.
We repeat this process for $R_0$ and $N_0$. Thus, we have $g(x,y,z) \in \mathrm{Im}e_{(12)(34)}$,
which implies that $e_{(12)(34)}$ is surjective and 
the image of $v_{(12)(34)}$ is in $\C((p,\p,|p|^\R)) \subset \C((p^{1/2},\p^{1/2},|p|^\R))$.
\begin{comment}
Set $v_{(12)(34)}(g(x,y,z))=\sum_{N,R \in \R}b_{N,R} p^N \p^R$,
where $b_{N,R}=\sum_{k,l} a_{k,2N-k,l,2R-l}$.

We first prove that  $v_{(12)(34)}$ is injective.
Assume that $v_{(12)(34)}g(x,y,z)=0$ and $g(x,y,z) \neq 0$.
Then, by Lemma \ref{recur_int}, we have $b_{N_0,R_0}=\sum_{k,l} a_{k,2N_0-k,l,2R_0-l}=a_{N_0,N_0,R_0,R_0} \neq 0$, a contradiction. Hence, $g(x,y,z)=0$ and $v_{(12)(34)}$ is injective.

We now prove that the image of $v_{(12)(34)}$ is in $\C((p,\p,|p|^\R))$.
By definition of $S'(x,y,z)$, it is easy to show that $v_{(12)(34)}(g(x,y,z)) \in \C[[p^{\frac{1}{2}},\p^{\frac{1}{2}}]][|p|^\R]$.
Thus, it suffices to show that $b_{N,R} =0$ if $N-R \notin \Z$.
We remark that this condition holds for any image of $s_{(12)(34)}$.
\end{comment}
Hence, we have:
\begin{lem}
\label{differential_inverse_2}
The map $v_{(12)(34)}:S'(x,y,z) \rightarrow \C((p,\p,|p|^\R))$ is the inverse of $s_{(12)(34)}:\C((p,\p,|p|^\R))
\rightarrow S'(x,y,z)$.
\end{lem}

\begin{cor}
\label{convergence_2}
Let $g' \in S'(x,y,z)$.
If there exists $f \in \F$ such that $v_{(12)(34)}(g')=j(p,f)$,
then $g' = e_{(12)(34)}(f\circ \xi)$. In particular, $g'$ is absolutely convergent
to a function in $\Cor_4$ in $|x|>>|y|>>|z|$.
\end{cor}

Finally, we state a similar result on $\Cor_3$.
Set
\begin{align*}
D_{1}''&=x^2d/dx+y^2d/dy, \\
\D_{1}''&=\x^2 d/d\x+\y^2 d/d\y\\
D_{0}''&=xd/dx+yd/dy, \\
\D_{0}''&=\x d/d\x+\y d/d\y,
\end{align*}
and
let $S''(x,y)$ be a space of the formal solutions  of $D_0', D_1',\D_0', \D_1'$ in
$$T(x,y)=\C[[y/x,\bar{y}/\bar{x}\}[x^\pm,\x^\pm,|x|,y^\pm,\y^\pm,|y|].$$
Then, we have:
\begin{lem}
\label{differential_3}
The space $S''(x,y)$ consists of the constant functions.
\end{lem}
\begin{proof}
Let $f \in S''(x,y)$
By $D_0''f=\D''_0 f=0$, we can assume that 
$f(x,y)=\sum_{r,s \in \R} a_{r,s} (y/x)^r (\y/\x)^s$.
By $D_1''f=\D''_1 f=0$, $a_{n,\n}$ satisfies
$ra_{r,s}=(r+1)a_{r+1,s}$ and $s a_{r,s}=(s+1)a_{r,s+1}$.
Thus, $a_{r,s}=0$ unless $(r,s)=(0,0)$.
\end{proof}

\subsection{Holomorphic correlation functions}\label{sec_holomorphic_func}
In this section, we consider holomorphic correlation functions.\begin{comment}
Wirtinger derivatives $d/dz_i$ and $d/d\overline{z}_i$ ($i=1,2,3,4$),
which act on $\Cor_4$.

Regarding $X_4$ as a subspace of $\C^4$, we define the

The derivations $d/dt$ and $d/d\overline{t}$ ($t=x,y,z,w$) on
$T(x,y,z,w)$ are also defined.

It is clear from the construction that 
the derivations and expansions commute with each other.
%for example, $e_{1(2(3(40)))}\circ d/dz_i=d/dz_i \circ e_{1(2(3(40)))}$ for $i=1,2,3,4$.
%and $e_{1304} \circ d/dz_2=d/dz_0 \circ e_{1304}$.
\end{comment}
The following proposition is important:
\begin{prop}
\label{cor_hol}
If $d/d\overline{z}_a \phi=0$ for $\phi \in \Cor_4$ and $a \in \{1,2,3,4\}$,
then $\phi$ is in $\Cor_4^f$.
Furthermore, $\phi$ is a finite sum of functions
$\Pi_{1 \leq i < j\leq 4}(z_i-z_j)^{\al_{ij}}(\overline{z}_i-\overline{z}_{j})^{\be_{ij}}$
such that $\al_{ij} -\be_{ij} \in \Z$ and $\be_{ai}=0$ for $i,j=1,2,3,4$.
\end{prop}

%We will use the following lemma:
%\begin{lem}\label{N_subspace}
%Let $f \in T(x,y,z)$ satisfy $\frac{d}{dz}^N f =0$ and $\frac{d}{d\z} f=0$ for some $N \in \Z_{>0}$.
%Then, $f \in  \bigoplus_{k=0}^{N-1}z^k \C[[y/x, \y/\x]][x^\pm,y^\pm |x|^\R,|y|^\R]$.
%\end{lem}

\begin{proof}
By the $S_4$-symmetry, we may assume that $a=2$.
By the definition of $\Cor_4$,
we can assume that $\phi(z_1,z_2,z_3,z_4)=\sum_{t=1}^N \phi^t(z_1,z_2,z_3,z_4)$ and
\begin{align}
\phi^t(z_1,z_2,z_3,z_4)=\Pi_{1 \leq i < j\leq 4}(z_{i}-z_j)^{\al_{ij}^t}(\overline{z}_{i}-\overline{z}_{j})^{\be_{ij}^t}f^t(\frac{(z_1-z_2)(z_3-z_4)}{(z_1-z_3)(z_2-z_4)}) \label{eq_zeta}
\end{align}
and $f^t \in \F$ 
%\Pi_{1\leq i \leq 3} z_{i}^{\al_i^t} \overline{z}_{i}^{\be_i^t}
 and for $\chi \in \{p,1-p,p^{-1}\}$ 
\begin{align}
j(\chi, f^t ) = \sum_{k=1}^{\infty}|p|^{r_k^{t,\chi}} \sum_{n,m \geq 0} a_{n,m}^{t,\chi} p^n \p^m.
\label{eq_f}
\end{align}
We first prove that there exists integers $n_{12}, n_{23},n_{24}, N_\infty$
such that $d/dz_2^{N_{\infty}} ((z_1-z_2)^{-n_{12}+1}(z_2-z_3)^{-n_{23}+1}(z_2-z_4)^{-n_{24}+1}\phi)=0$.
Fix $z_1,z_3,z_4$.
Since $d/d\overline{z}_2 \phi =0$, $\phi$ is a holomorphic function for the variable $z_2$ on $\C\setminus \{z_1,z_3,z_4 \}$.
Let $n_{12}$ (resp. $n_{23}$, $n_{24}$) be the largest integer smaller than any element of the
set $\{\al_{12}^t+\be_{12}^t+r_k^{t,p} \}_{t=1,\dots,N, k=1,2,3,\dots}$
(resp. $\{\al_{23}^t+\be_{23}^t+r_k^{t,1-p} \}_{t=1,\dots,N, k=1,2,3,\dots}$
and $\{\al_{24}^t+\be_{24}^t+r_k^{t, p^{-1}} \}_{t=1,\dots,N, k=1,2,3,\dots}$),
which exists since $r_k^{t,\chi} \in \R$ is bounded below.
Then, by (\ref{eq_zeta}) and (\ref{eq_f}), $\lim_{z_2 \rightarrow z_1} |z_2-z_1|^{\frac{-n_{12}+1}{2}}f=0$
 and $\lim_{z_2 \rightarrow z_3} |z_2-z_3|^{\frac{-n_{23}+1}{2}}f=0$
and $\lim_{z_2 \rightarrow z_4} |z_2-z_4|^{\frac{-n_{24}+1}{2}}f=0$.
Hence, $(z_1-z_2)^{-n_{12}+1}(z_2-z_3)^{-n_{23}+1}(z_2-z_4)^{-n_{24}+1}\phi$ is a holomorphic function on $\C$.
Since $d/d\overline{z}_2 (z_1-z_2)^{-n_{12}+1}(z_2-z_3)^{-n_{23}+1}(z_2-z_4)^{-n_{24}+1}\phi=0$,
we may assume without loss of generality that $\phi$ is a holomorphic function
for the variable $z_2$ on $\C$.
Let $n_\infty$ be the largest integer smaller than any element
of the set $\{ \al_{12}^t+\al_{23}^t+\al_{24}^t+1 \}_{t=1,\dots,N}$.
Then, by (\ref{eq_zeta}),
 $\lim_{z_2 \rightarrow \infty} |z_2|^{-n_{\infty}}\phi=0$.
Hence, $\phi$ is a rational function on the projective plane $\C P^1$
and holomorphic on $\C$.
Thus, $\phi$ is a polynomial whose degree is less that $n_\infty$,
 which implies that ${d/dz_2}^{n_{\infty}} \phi=0$.

Now, we consider the expansion map
$e_{1(4(23))}: \Cor_4 \rightarrow T(x_{1(4(23))},y_{1(4(23))},z_{1(4(23))})$
and set $(x_{1(4(23))},y_{1(4(23))},z_{1(4(23))},\x_{1(4(23))},\y_{1(4(23))},\z_{1(4(23))})=(x,y,z,\x,\y,\z)$.
Since $T(x,y,z)$ is a completion of
$$\bigoplus_{r,s\in \R} z^r \z^s \C[[y/x,\y/\x]][x^\pm,y^\pm,\x^\pm,\y^\pm,|x|^\R,|y|^\R],
$$
by ${d/dz_2}^{n_{\infty}} \phi=0$ and $d/d\overline{z}_2 \phi =0$,
$$e_{1(4(23))}(\phi) \in \bigoplus_{k=0}^{n_\infty-1}z^k \C[[y/x, \y/\x]][x^\pm,y^\pm,\x^\pm,\y^\pm, |x|^\R,|y|^\R].$$
Denote by $pr_{n_\infty}: T(x,y,z) \rightarrow \bigoplus_{k=0}^{n_\infty-1}z^k \C[[y/x, \y/\x]][x^\pm,y^\pm,\x^\pm,\y^\pm, |x|^\R,|y|^\R]$
the projection of $T(x,y,z)$ onto the coefficients of $1,z,\dots,z^{n_\infty - 1}$.
Set $z_{ij}=z_i-z_j$ for $i,j =1,\dots,4$.
Then, 
\begin{align}
e_{1(4(23))}(\phi) &= pr_{n_\infty} e_{1(4(23))}(\phi) \nonumber \\
&=pr_{n_\infty} \sum_{t=1}^N  e_{1(4(23))}^f (\Pi_{1 \leq i < j\leq 4}(z_{ij})^{\al_{ij}^t}(\bar{z}_{ij})^{\be_{ij}^t})
 s_{1(4(23))}(j(1-p,f^t), \nonumber \\
&=pr_{n_\infty} \sum_{t=1}^N \sum_{k=1}^\infty
 e_{1(4(23))}^f (\Pi_{1 \leq i < j\leq 4}(z_{ij})^{\al_{ij}^t}(\bar{z}_{ij})^{\be_{ij}^t})
s_{1(4(23))}(\sum_{n,m \geq 0} a_{n,m}^{t,1-p} p^n \p^m |p|^{r_k^{t,1-p}} ). \label{eq_bounded}
\end{align}

Since $e_{1(4(23))}^f (\Pi_{1 \leq i < j\leq 4}z_{ij}^{\al_{ij}^t}\bar{z}_{ij}^{\be_{ij}^t})
\in z^{\al_{23}^t}\z^{\be_{23}^t} \C[[y/x,z/y, \y/\x, \z/\y]][x^\pm,y^\pm |x|^\R,|y|^\R]$
and $s_{1(4(32))}(p^a\p^{a'}) \in z^a \z^{a'}\C[[y/x,z/y, \y/\x, \z/\y]][x^\pm,y^\pm |x|^\R,|y|^\R]$ for any $a,a'\in \R$,
$$pr_{n_\infty} e_{1(4(32))}^f (\Pi_{1 \leq i < j\leq 4}z_{ij}^{\al_{ij}^t}\bar{z}_{ij}^{\be_{ij}^t}) s_{1(4(23))}(p^a\p^{a'})$$ is
equal to zero unless $\al_{23}^t+a \in \Z_{\leq n_\infty}$ and $\be_{23}^t+a' \in \Z_{\leq 0}$.
Thus, the right hand side of (\ref{eq_bounded}) is finite sum, that is,
\begin{align*}
=&pr_{n_\infty} \sum_{t=1}^N 
\sum_{\substack{k=1,2,3,\dots \\ r_k^{t,1-p}+\be_{23}^t \in \Z_{\leq 0} \\ r_k^{t,1-p}+\al_{23}^t \in \Z_{\leq n_\infty}}}
 e_{1(4(23))}^f (\Pi_{1 \leq i < j\leq 4}(z_{ij})^{\al_{ij}^t}(\bar{z}_{ij})^{\be_{ij}^t})
s_{1(4(23))}(  \sum_{\substack{n_\infty-\al_{23}^t-r_k^{t,1-p} \geq n \geq 0 \\
-\be_{23}^t-r_k^{t,1-p} \geq m \geq 0}} a_{n,m}^{t,1-p} p^n \p^m |p|^{r_k^{t,1-p}} ).\\
&=pr_{n_\infty}
\sum_{t=1}^N
\sum_{\substack{k=1,2,3,\dots \\ r_k^{t,1-p}+\be_{23}^t \in \Z_{\leq 0} \\ r_k^{t,1-p}+\al_{23}^t \in \Z_{\leq n_\infty}}}
\sum_{\substack{n_\infty-\al_{23}^t-r_k^{t,1-p} \geq n \geq 0 \\
-\be_{23}^t-r_k^{t,1-p} \geq m \geq 0}}
 e_{1(4(32))}^f ( \Pi_{1 \leq i < j\leq 4}z_{ij}^{\al_{ij}^t}\bar{z}_{ij}^{\be_{ij}^t}
(1-\xi)^{r_k^{t,1-p} +n }
(1-\bar{\xi})^{r_k^{t,1-p} +m }).
\end{align*}
Thus, the coefficients of $1,z,\dots,z^{n_\infty-1}$ in $e_{1(4(32))}(\phi)$ is an element in $\Cor_4^f$.
This finishes the proof, the detailed verification of the assertion being left to the
reader.

\end{proof}

\begin{cor}
\label{cor_3}
If a function $f \in \Cor_4$ is independent of the variable $z_4$, i.e.,
$d/dz_4 f=d/d\z_4 f$,
then $f \in \Cor_3$.
\end{cor}

\subsection{Vacuum state and $D_4$-symmetry} \label{sec_expansion2}
In section \ref{sec_expansion}, we consider the expansions of a function in $\Cor_4$ in three variables 
associated with an element in $A \in P_4$.
In this section, we consider expansions in four variables.
Let $Q_4$ be the set of parenthesized products of five elements $1,2,3,4,\star$ 
with $\star$ at right most,
e.g., $3((4((12)\star))$ (see Introduction \ref{intro_parenthesized}). 
The permutation group $S_4$ acts on $Q_4$, which fixes $\star$,
and $Q_4$ consists of the permutations of the following elements, called standard elements of $Q_4$:
\begin{align}
((12)(34))\star, (12)((34)\star), (12)(3(4\star)), \nonumber \\
(((12)3)4)\star,((12)3)(4\star), \tag{standard elements of $Q_4$}  \nonumber \\
((1(23))4)\star,(1(23))(4\star),  \nonumber \\
(1((23)4))\star,1(((23)4)\star),1((23)(4\star)),  \nonumber \\
(1(2(34)))\star,1((2(34))\star),1(2((34)\star)),1(2(3(4\star))).  \nonumber
\end{align}
The rule for the change of variables is given in Appendix.
%By deleting the symbol $\star$, we have $el:Q_4 \rightarrow P_4$.
In this section, we briefly explain that all the expansions associated with $Q_4$
are given by the expansions associated with $P_4$.

We start with the example $1(2(3(4\star))) \in Q_4$.
In this case, we consider the expansion of a function $\Cor_4$
in $|z_1|>|z_2|>|z_3|>|z_4|$.
Set 
$$T^\star(x,y,z,w)=\C[[y/x,w/z, \overline{y}/\overline{x},\bar{w}/\z]]((z/y,\z/\y,|z/y|^\R))
[x^{\pm}, y^\pm,z^\pm,\overline{x}^\pm ,\overline{y}^\pm,\overline{z}^\pm,|x|^\R,|y|^\R,|z|^\R],$$
for the formal variables $x,y,z,w,\bar{x},\bar{y},\bar{z},\bar{w}$. 
We recall that the expansion $e_{1(2(34))}$ is given by
the change of variables $(z_1-z_4,z_2-z_4,z_3-z_4)\mapsto (z_{14}, z_{24},z_{34})$
with the expansion in $|z_{14}|>| z_{24}|>|z_{34}|$.
Then, let $T_{1(2(34))}^{1(2(3(4\star)))}:T(z_{14},z_{24},z_{34}) \rightarrow T^\star(z_1,z_2,z_3,z_4)$ be the linear map
defined by
\begin{align*}
T_{1(2(34))}^{1(2(3(4\star)))} = \lim_{(z_{14}, z_{24},z_{34}) \to (z_1,z_2,z_3)}
\exp(-z_4(d/dz_{14}+d/dz_{24}+d/dz_{34}))\exp(-\z_4(d/d\z_{14}+d/d\z_{24}+d/d\z_{34}))
\end{align*}
and
set
$$e_{1(2(3(4\star)))}=T_{1(2(34))}^{1(2(3(4\star)))}  \circ e_{1(2(34))}:
\Cor_4 \rightarrow T^\star(z_1,z_2,z_3,z_4).
$$
Then, similarly to Lemma \ref{ass_1}, we have:
\begin{lem}
For $\phi \in \Cor_4$, the formal power series
$e_{1(2(3(4\star)))}(\phi)$ is absolutely convergent to $\phi(z_1,z_2,z_3,z_4)$
in $|z_1|>>|z_2|>>|z_3|>>|z_4|$.
\end{lem}

For the case of $1((2(34))\star)\in Q_4$,
we consider the following expansion with the change of variables
$$
(z_1,z_{24},z_{34},z_4)=(z_1,z_2-z_4,z_3-z_4,z_4),\;
\{|z_1|>>|z_{24}|>>|z_{34}|>>|z_4| \}.$$
Similarly to the above, such expansion $e_{1((2(34))\star)}:\Cor_4 \rightarrow T^\ast(z_1,z_{24},z_{34},z_4)$
is given by
\begin{align*}
e_{1((2(34))\star)} &=T_{1(2(34))}^{1((2(34))\star)} \circ e_{1((2(34)))},\\
T_{1(2(34))}^{1((2(34))\star)} &=  \lim_{\uz_{14} \to \uz_1} \exp(-z_4d/dz_{14})\exp(-\z_4d/d\z_{14}).
\end{align*}
In this way, we can define expansions and the space of formal power series for all elements in $Q_4$
by using the expansions defined in the previous section.

Hereafter, we study $e_{1(2(3(4\star)))}$ in more detail.
Since the expansions and the derivations are commute with each other,
by Lemma \ref{four_translation}, the image of $e_{1(2(3(4\star)))}$ is
in the kernel of $\frac{d}{dz_1}+\frac{d}{dz_2}+\frac{d}{dz_3}+\frac{d}{dz_4}$
and $\frac{d}{d\z_1}+\frac{d}{d\z_2}+\frac{d}{d\z_3}+\frac{d}{d\z_4}$.
Set
$$T_0^\star(z_1,z_2,z_3,z_4)=\{f \in T^\star(z_1,z_2,z_3,z_4)\;|\; \frac{d}{dz_1}+\frac{d}{dz_2}+\frac{d}{dz_3}+\frac{d}{dz_4}f=
\frac{d}{dz_1}+\frac{d}{dz_2}+\frac{d}{dz_3}+\frac{d}{dz_4}f=0\}.$$
Then, the image of 
$T_{1(2(34))}^{1(2(3(4\star)))}:T(z_{14},z_{24},z_{34}) \rightarrow T^\star(z_1,z_2,z_3,z_4)$ is in $T_0^\star(z_1,z_2,z_3,z_4)$, since $[\frac{d}{dz_4},  \exp(-z_4 (\frac{d}{dz_1}+\frac{d}{dz_2}+\frac{d}{dz_3}))]=-(\frac{d}{dz_1}+\frac{d}{dz_2}+\frac{d}{dz_3})$.
\begin{lem}
\label{expansion_translation}
The above map $T_{1(2(34))}^{1(2(3(4\star)))}: T(z_{14},z_{24},z_{34})\rightarrow T_0^\star(z_1,z_2,z_3,z_4)$ is an isomorphism.
\end{lem}
\begin{proof}
By substituting $z_4=\z_4=0$, we obtain the left inverse of the map, which implies that the map is injective.
Let $f(z_1,z_2,z_3,z_4)=\sum_{n,m \geq 0}f_{n,m} z_4^n \z_4^m \in T_0^\star(z_1,z_2,z_3,z_4)$. It suffices to show that $f$ is in the image of the map. We may assume that $f_{0,0}=0$.
Then, $nf_{n,m}=(\frac{d}{dz_1}+\frac{d}{dz_2}+\frac{d}{dz_3})f_{n-1,m}$ and
$mf_{n,m}=(\frac{d}{d\z_1}+\frac{d}{d\z_2}+\frac{d}{d\z_3})f_{n,m-1}$.
Thus, $f_{n,m}=0$ for all $n,m \geq 0$.
\end{proof}

%By the above lemma and Lemma \ref{formal_skew},
%we can transform $e_{1(2(3(40)))}$ into $e_{1(2(4(30)))}$,
%which corresponds to the skew-symmetry for vertex algebras.

We end this section by studying the relation between $e_{1(2(3(4\star)))}$ and $e_{4(3(2(1\star)))}$.
%which is related to the existence of an invariant bilinear form for vertex algebras (see Section \ref{}).
The map $e_{1(2(3(4\star)))}$ expands a function in $|z_1|>|z_2|>|z_3|>|z_4|$,
whereas $e_{4(3(2(1\star)))}$ expands in $|z_4|>|z_3|>|z_2|>|z_1|$.
We consider the following involution:
$X_4 \rightarrow X_4,\;(z_1,z_2,z_3,z_4) \mapsto (z_1^{-1},z_2^{-1},z_3^{-1},z_4^{-1})$
and
set 
$$T_d^\star(x,y,z,w)=\C[[y/x,w/z, \overline{y}/\overline{x},\bar{w}/\z]]((z/y,\z/\y,|z/y|^\R))
[y^\pm,z^\pm,\overline{y}^\pm,\overline{z}^\pm,|y|^\R,|z|^\R],$$
which is a subspace of $T^\star(x,y,z,w)$.
Define the map
$I_d: T_d^\star(z_1,z_2,z_3,z_4) \rightarrow T_d^\star(z_4,z_3,z_2,z_1)$
by $(\uz_1,\uz_2,\uz_3,\uz_4) \mapsto (\uz_1^{-1},\uz_2^{-1},\uz_3^{-1},\uz_4^{-1})$.
We observe that
\begin{align*}
I_d (z_1^{-r} e_{1(2(3(4\star)))}{(z_1-z_2)^{r}})&=I_d (\sum_{ i \geq 0}(-1)^i \binom{r}{i} z_1^{-i}z_2^i) \\
&=\sum_{ i \geq 0}(-1)^i \binom{r}{i} z_1^{i}z_2^{-i} \\
&= z_2^{-r}(-1)^r e_{4(3(2(1\star)))}((z_1-z_2)^{r}).
\end{align*}
Let $$
\phi(z_1,z_2,z_3,z_4) = \Pi_{1\leq i<j \leq 4} (z_i-z_j)^{\al_{ij}}(\bar{z}_i-\bar{z}_j)^{\be_{ij}} f\circ \xi(z_1,z_2,z_3,z_4) \in \Cor_4,$$
where $f \in \F$ and $\al_{ij},\be_{ij} \in \R$ satisfy $\al_{ij}-\be_{ij}\in \Z$ for any $1\leq i <j\leq 4$.
Set $P(\al,\be,z)=\Pi_{1 \leq i<j \leq 4} (-1)^{\al_{ij}-\be_{ij}} (z_iz_j)^{\al_{ij}}(\z_i\z_j)^{\be_{ij}}$.
Then, we have:
\begin{lem}
\label{formal_dual}
$I_d (P(\al,\be,z)^{-1} e_{1(2(3(4\star)))}(\phi))=e_{4(3(2(1\star)))}(\phi)$.
\end{lem}
%In Lemma \ref{dihedral}, we introduce the $D_4$-symmetry of the expansion.
The above transformation corresponds to $(14)(23)$ in $S_4$.
It is well-known that the axiom of a vertex algebra possesses $S_3$-symmetry (see \cite{FHL}).
We will later prove that the correlation functions is $S_4$-invariant by using 
the $S_3$-symmetry together with $(14)(23)$-symmetry above (see Proposition \ref{convergence_standard}).

\subsection{Generalized two point Correlation function}\label{sec_gen}
In the theory of a vertex algebra, four point correlation functions in the limit of $(z_1,z_4) \mapsto (\infty,0)$
are used, which we call generalized two point functions (see Introduction \ref{intro_parenthesized}).
This subsection is devoted to studying the property of this generalized two point function.
Set $$U(y,z)=
\C((z/y,\z/\y,|z/y|^\R))[y^\pm,z^\pm,\y^\pm,\z^\pm,|y|^\R,|z|^\R]
$$
and $$Y_2=\{(z_1,z_2) \in \C^2\;|\; z_1\neq z_2, z_1 \neq 0, z_2 \neq 0 \}.$$
%and $$S_2=\C[z_2^\pm ,z_{23}^\pm,z_3^\pm,\z_2^\pm ,\z_{23}^\pm,\z_3^\pm,|z_2|^\R,|z_3|^\R,|z_{23}|^\R].$$
Let %$C^\omega(Y_2)$ be the space of real analytic function on $Y_2$
 $\eta(z_1,z_2):Y_2 \rightarrow \CPm$ be the real analytic function defined by
$\eta(z_1,z_2)=\frac{z_2}{z_1}$.
For $f \in \F$, $f \circ \eta$ is a real analytic function on $Y_2$.
Denote by $\GCor_2$ the space of real analytic functions on $Y_2$
spanned by
\begin{align}
z_1^{\al_1} z_2^{\al_2} (z_1-z_2)^{\al_{12}} \bar{z}_1^{\be_1} \bar{z}_2^{\be_2} (\bar{z}_1-
\bar{z}_2)^{\be_{12}} f\circ \eta(z_1,z_2), \label{eq_GCO}
\end{align}
%and
%$$z_1^{\al_1} z_2^{\al_2} (z_1-z_2)^{\al_{12}} \bar{z}_1^{\be_1} \bar{z}_2^{\be_2} (\bar{z}_1-
%\bar{z}_2)^{\be_{12}},$$
where $f\in \F$ and $\al_1,\al_2,\al_{12},\be_1,\be_2,\be_{12} \in \R$ satisfy
$\al_1-\be_1,\al_2-\be_2,\al_{12}-\be_{12} \in \Z$.
Define the action of $\A$ on $Y_2$ by
\begin{align*}
(01)\cdot (z_1,z_2)&= (z_1,z_1-z_2),\\
(0\infty)\cdot (z_1,z_2)&= (z_2,z_1),\\
(1\infty)\cdot (z_1,z_2)&=(z_2-z_1,z_2),\\
(01\infty)\cdot (z_1,z_2)&=(z_1-z_2,z_1),\\
(10\infty)\cdot (z_1,z_2)&=(z_2,z_2-z_1),
\end{align*}
for $(z_1,z_2)\in Y_2$. Then, we have:
\begin{lem}\label{Y2_covariant}
The map $\eta:Y_2\rightarrow \CPm$ commutes with the action
of $\A$, i.e.,
$\eta(t\cdot (z_1,z_2))=t \cdot \frac{z_2}{z_1}$
for any $(z_1,z_2)\in Y_2$ and  $t\in \A$.
In particular, for  $\mu\in\GCor_2$, 
$\Bigl(t\cdot\mu\Bigr)(-)=\mu(t^{-1}-) \in \GCor_2$ for
any $t\in \A$.
\end{lem}

Let $\mu(z_1,z_2)=z_1^{\al_1} z_2^{\al_2} (z_1-z_2)^{\al_{12}} \bar{z}_1^{\be_1} \bar{z}_2^{\be_2} (\bar{z}_1-
\bar{z}_2)^{\be_{12}} f\circ \eta(z_1,z_2)$ in (\ref{eq_GCO}).
The expansions of $\mu$ in $\{|z_1|>|z_2|\}$
and $\{|z_2|>|z_1|\}$ are respectively
given by
\begin{align*}
z_1^{\al_1+\al_{12}}\bar{z}_1^{\be_1+\be_{12}}z_2^{\al_2} \bar{z}_2^{\be_2} \sum_{i,j \geq 0}\binom{\al_{12}}{i}\binom{\be_{12}}{j}
(-z_2/z_1)^i(-\z_2/\z_1)^j &\lim_{p \to z_2/z_1} j(p,f)\\
(-1)^{\al_{12}-\be_{12}}
z_1^{\al_1}\bar{z}_1^{\be_1}z_2^{\al_2+\al_{12}} \bar{z}_2^{\be_2+\be_{12}} \sum_{i,j \geq 0}\binom{\al_{12}}{i}\binom{\be_{12}}{j}
(-z_1/z_2)^i
(-\z_1/\z_2)^j &\lim_{p \to z_1/z_2} j(p^{-1},f),
\end{align*}
which define maps
$$|_{|z_1|>|z_2|}:\GCor_2 \rightarrow U(z_1,z_2), \mu(z_1,z_2) \mapsto \mu(z_1,z_2)|_{|z_1|>|z_2|}$$
and
$$|_{|z_2|>|z_1|}:\GCor_2 \rightarrow U(z_2,z_1), \mu(z_1,z_2) \mapsto \mu(z_1,z_2)|_{|z_2|>|z_1|}.$$
Since
$f(\frac{z_2}{z_1})=f(\frac{z_2}{z_2+(z_1-z_2)})$,
the expansions of $\mu$ in $\{|z_2|>|z_1-z_2|\}$
is given by
\begin{align*}
z_2^{\al_1+\al_2} \bar{z}_2^{\be_1+\be_2}z_0^{\al_{12}}\z_0^{\be_{12}}
\sum_{i,j \geq 0}\binom{\al_{1}}{i}\binom{\be_{1}}{j}
(z_0/z_2)^i(\z_0/\z_2)^j &\lim_{p \to -z_0/z_2} j(1-p^{-1},f),
\end{align*}
where $z_0=z_1-z_2$.
We denote it by
$$|_{|z_2|>|z_1-z_2|}:\GCor_2 \rightarrow U(z_2,z_0),
\mu(z_1,z_2) \mapsto \mu(z_1,z_2)|_{|z_2|>|z_1-z_2|}.
$$
Then, we have:
\begin{lem}
For $f\in \F$,
\begin{align*}
f\circ \eta |_{|z_1|>|z_2|} &= \lim_{p \to z_1/z_2} j(1,f),\\
f\circ \eta |_{|z_2|>|z_1|} &= \lim_{p \to z_2/z_1} j(p^{-1},f),\\
f\circ \eta |_{|z_2|>|z_1-z_2|} &= \lim_{p \to -z_0/z_2} j(1-p^{-1},f).
\end{align*}
\end{lem}

For $a,a' \in \R$, define a $\C$-linear map 
$C_{a,a'}(x): T(x,y,z) \rightarrow U(y,z)$
by taking the coefficient of $x^a\x^{a'}$.
Then, we have the $\C$-linear map
$C_{a,a'}(x):T(x,y,z) \rightarrow U(y,z)$.

\begin{prop}
\label{gcor}
For $\phi \in \Cor_4$ and $a,a' \in \R$,
there exists $\eta \in \GCor_2$ satisfying the following conditions:
\begin{enumerate}
\item
$C_{a,a'}(x_{1(2(34))})e_{1(2(34))} \phi = \eta|_{|z_1|>|z_2|}$,
\item
$C_{a,a'}(x_{1(3(24))})e_{1(3(24))} \phi = \eta|_{|z_2|>|z_1|}$,
\item
$C_{a,a'}(x_{1((23)4)})e_{1((23)4)} \phi = \eta|_{|z_2|>|z_1-z_2|}$,
\end{enumerate}
where we identify $(y_{1(2(34))},z_{1(2(34))})$ and $(y_{1(3(24))},z_{1(3(24))})$, $(y_{1((23)4)},z_{1((23)4)})$ as $(z_1,z_2)$ and $(z_2,z_1)$, $(z_2,z_0)$, respectively.
\end{prop}

\begin{proof}
Since the expansion is linear, we may assume that
$\phi= \Pi_{1 \leq i<j \leq 4}(z_i-z_j)^{\al_{ij}}(\z_i-\z_j)^{\be_{ij}}f\circ \xi $,
where $f \in \F$ and $\al_{ij}-\be_{ij} \in \Z$.
Set $$T_g(x,y,z)=\C[[y/x,\y/\x]]((z/y,\overline{z}/\overline{y},|z/y|^\R))
[y^\pm,z^\pm,\overline{y}^\pm,\z^\pm,|y|^\R,|z|^\R],$$
which involves only $x^{-n}\x^{-n'}$ with $n,n' \in \Z_{\geq 0}$,
and $c=\al_{12}+\al_{13}+\al_{14}$
and $c'=\be_{12}+\be_{13}+\be_{14}$.
Let $A = 1(2(34)), 1(3(24)), 1((23)4)$.
Then, $e_A(\phi) \subset z_1^c \z_1^{c'}
T_g(x_A,y_A,z_A)$ and, thus, $C_{a,a'}(x_A)e_A (\phi)=0$ unless $c-a,c'-a' \in \Z_{\geq 0}$.
Hence, we may assume that $a=c-n$ and $a'=c'-n'$ for some $n,n' \in \Z_{\geq 0}$.
Since $\lim_{z_1\to \infty}z_1^{-c} \z_1^{-c'} \phi$ exists and
is equal to
$$\Pi_{2 \leq i<j \leq 4}(z_i-z_j)^{\al_{ij}}(\z_i-\z_j)^{\be_{ij}}f(\frac{z_3-z_4}{z_2-z_4}),
$$
$C_{c-n,c'-n'}(x_A)e_{A}(\phi)$ is convergent to the real analytic function
\begin{align}
\eta(z_1,z_2)=\frac{1}{n!n'!}\lim_{(z_1,z_2,z_3,z_4) \rightarrow (\infty,z_1,z_2,0)} 
(-z_1^2 d/dz_1)^n(-\overline{z}_1^2 d/d\overline{z}_1)^{n'} 
(z_1^{-c} \bar{z}_1^{-c'} \phi(z_1,z_2,z_3,z_4))\in \GCor_2,
\nonumber
\end{align}
where we used $-z^2\frac{d}{dz} z^{-k}=kz^{-(k-1)}$.
Since $e_{1((23)4)}$ is the expansion around $|z_1-z_4|>|z_3-z_4|>|z_2-z_3|$,
in the limit of $\lim_{(z_1,z_2,z_3,z_4) \rightarrow (\infty,z_1,z_2,0)}$,
it gives the expansion around $|z_2|>|z_1-z_2|$.
\end{proof}
\begin{comment}
Thus, define a function $p(z_1,z_2)\in \GCor_2^f$
by
\begin{align}
p(z_1,z_2) &= z_1^{\al_{24}}\bar{z}_1^{\be_{24}}
z_2^{\al_{34}}\bar{z}_2^{\be_{34}}(z_1-z_2)^{\al_{23}}(\z_1-\z_2)^{\be_{23}}\nonumber \\
&\lim_{(t,\bar{t}) \mapsto (0,0)}\frac{1}{n!n'!}(d/dt)^n
(d/d\bar{t})^{n'}(1-tz_1)^{\al_{12}}(1-tz_2)^{\al_{13}}(1-\bar{t}\bar{z_1})^{\be_{12}}(1-\bar{t}\bar{z_2})^{\be_{13}}.
\nonumber
\end{align}
Then, the real analytic function $p(z_1,z_2)f\circ \mu$ satisfies the conditions (1),(2),(3).
\end{proof}

By the above proposition, we can define $C_{a,a'}^{1(2(34))}:\Cor_4 \rightarrow \GCor_2$
by $\phi \mapsto \eta$, where $\eta$ is a unique function satisfying
$C_{a,a'}(x_{1(2(34))})e_{1(2(34))} \phi = \eta|_{z_1>z_2}$.
It is clear that $C_{a,a'}^{1(2(34))}$ is surjective and
$C_{a,a'}^{1(2(34))} \circ d/dz_2=d/dz_1 \circ C_{a,a'}^{1(2(34))}$,
$C_{a,a'}^{1(2(34))} \circ d/dz_3=d/dz_2 \circ C_{a,a'}^{1(2(34))}$.Let $\mu= 
z_1^{\al_1} z_2^{\al_2} (z_1-z_2)^{\al_{12}} \bar{z}_1^{\be_1} \bar{z}_2^{\be_2} (\bar{z}_1-
\bar{z}_2)^{\be_{12}} f(\frac{z_2}{z_1}) \in \GCor_2$.
Set $\phi= (z_1-z_4)^{\al_1} (z_2-z_4)^{\al_2} (z_1-z_2)^{\al_{12}} (\bar{z}_1-\z_4)^{\be_1} (\bar{z}_2-\z_4)^{\be_2} (\bar{z}_1-
\bar{z}_2)^{\be_{12}} f\circ \xi(z_0,z_1,z_2,z_3)$.
Then, $\phi \in \Cor_4$ satisfy $C_{0,0}(x_{1(2(34))})e_{1(2(34))} \phi = \mu|_{z_1>z_2}$.
Furthermore, $\frac{d}{d\z_{i}}\mu|_{z_1>z_2}=C_{0,0}(x_{1(2(34))})e_{1(2(34))} \frac{d}{d\z_{i+1}} \phi$ for $i=1,2$.
Combining this wit
\end{comment}
Similarly to the proof of Proposition \ref{cor_hol}, we have:
\begin{lem}
\label{hol_generalized}
If $\mu \in \GCor_2$ satisfies $\frac{d}{d\z_1} \mu=0$,
then $\mu(z_1,z_2)\in \C[z_1^\pm,(z_1-z_2)^\pm,z_2^\pm,\z_2^\pm,|z_2|^\R]$.
Furthermore, if $\frac{d}{d\z_1} \mu=\frac{d}{d\z_2} \mu=0$,
then $\mu(z_1,z_2) \in \C[z_1^\pm,z_2^\pm,(z_1-z_2)^\pm]$.
\end{lem}
Similarly to Section \ref{sec_differential}, 
let $S_g(z_1,z_2)$ be the subspace of $U(z_1,z_2)$ consisting of $f \in U(z_1,z_2)$ with
$(z_1d/dz_1+z_2d/dz_2)f=(\z_1d/d\z_1+\z_2d/d\z_2)f=0.$
Clearly, the map $s_g:\C((p,\p,|p|^\R)) \rightarrow S_g(z_1,z_2) $ defined by the substitution of
$z_2/z_1$ into $p$
and the map $v_g:S_g(z_1,z_2) \rightarrow \C((p,\p,|p|^\R))$ defined by the limit $(z_1,z_2) \rightarrow (1,p)$
are mutually inverse.
Then, we have:
\begin{lem}
\label{generalized_limit}
Let $\mu \in \GCor_2$ satisfy $(z_1d/dz_1+z_2d/dz_2)\mu=(\z_1d/d\z_1+\z_2d/d\z_2)\mu=0$.
Then, there exists $f \in \F$ such that 
\begin{align*}
\mu|_{|z_1|>|z_2|}&= \lim_{\underline{p} \to \uz_2/\uz_1} j(p,f)\\
\mu|_{|z_2|>|z_1-z_2|}&= \lim_{\underline{p} \to - \uz_0/\uz_1} j(1-p^{-1},g)\\
\mu|_{|z_2|>|z_1|}&= \lim_{\underline{p} \to \uz_1/\uz_2}j(p^{-1},g).
\end{align*}
\end{lem}
\begin{proof}
Set $f(p)=\mu(1,p)$,
which is a real analytic function on $\CPm$.
By the existence of the expansion around $\{0,1,\infty\}$,
$f$ has conformal singularities at $\{0,1,\infty\}$. Thus, the assertion holds.
\end{proof}

%%%%%
%%%%% Duality
%%%%%

We end this section by studying the behavior of the expansion $|_{|z_2|>|z_1-z_2|}$ of generalized two point functions
under the involution $I_Y: Y_2 \rightarrow Y_2,\; (z_1,z_2)\mapsto (z_1^{-1},z_2^{-1})$,
which is important to define a dual module for a full vertex algebra.
The involution $I_Y$ acts on $\GCor_2$ by $\mu(z_1,z_2) \mapsto \mu(z_1^{-1},z_2^{-1})$ for $\mu \in \GCor_2$.
Define the map
$$\lim_{(z_0,z_2) \mapsto (\frac{-z_0}{z_2(z_2+z_0)} , \frac{1}{z_2})} :U(z_2,z_0) \rightarrow U(z_2,z_0)$$
by substituting $(-z_0/z_2^2 \sum_{i\geq 0}(-1)^i (z_0/z_2)^i , z_2^{-1})$ into $z_0$ and $z_2$.
The map is well-defined, since $U(z_2,z_0)=\C((z_0/z_2,\z_0/\z_2,|z_0/z_2|))[z_0^\pm,z_2^\pm,\z_0^\pm,\z_2^\pm,|z_0|^\R,|z_2|^\R]$ and
$\lim_{(z_0,z_2) \mapsto (\frac{-z_0}{z_2(z_2+z_0)} , \frac{1}{z_2})} z_0/z_2= -z_0/{z_2}(1-z_0/z_2+ (z_0/z_2)^2- \cdots)$.

\begin{lem}
\label{dual_generalized}
For any $\mu \in \GCor_2$,
$$\lim_{(z_0,z_2) \mapsto (\frac{-z_0}{z_2(z_2+z_0)} , \frac{1}{z_2})} \mu|_{|z_2|>|z_1-z_2|}
= I_Y(\mu) |_{|z_2|>|z_1-z_2|}.
$$
\end{lem}
\begin{proof}
Set $\mu'(z_0,z_2)=\mu(z_0+z_2,z_2)$.
%and $\mu|_{z_2>z_1-z_2}=\sum_{n,m,r,s} a_{n,m,r,s} z_0^n z_2^n \z_0^r \z_2^s \in U(z_2,z_0)$.
Since the series $\mu|_{|z_2|>|z_1-z_2|} \in U(z_2,z_0)$ is absolutely convergent in $|z_2|>>|z_0|$,
$\lim_{(z_0,z_2) \mapsto (\frac{-z_0}{z_2(z_2+z_0)} , \frac{1}{z_2})} \mu|_{|z_2|>|z_1-z_2|}$ is also absolutely
convergent to $\mu'(\frac{-z_0}{z_2(z_2+z_0)},\frac{1}{z_2})$ in $|z_2|>>|z_0|$,
which is equal to $\mu(\frac{-z_0}{z_2(z_2+z_0)}+\frac{1}{z_2},\frac{1}{z_2})=\mu((z_0+z_2)^{-1},z_2^{-1})$.
The series $I_Y(\mu) |_{|z_2|>|z_1-z_2|}$ is also absolutely convergent to $\mu((z_0+z_2)^{-1},z_2^{-1})$ in the same domain.
Hence, they coincide with each other.
\end{proof}

\section{Full vertex algebra}\label{sec_full_vertex}
In this section, we introduce the notion of a full vertex algebra,
which is a generalization of a $\Z$-graded vertex algebra.

\subsection{Definition of full vertex algebra}\mbox{}
For an $\R^2$-graded vector space $F=\bigoplus_{h,\h \in \R^2} F_{h,\h}$, 
set $F^\vee=\bigoplus_{h,\h \in \R^2} F_{h,\h}^*$,
where $F_{h,\h}^*$ is the dual vector space of $F_{h,\h}$.
A full vertex algebra is an $\R^2$-graded $\C$-vector space
$F=\bigoplus_{h,\h \in \R^2} F_{h,\h}$ equipped with a linear map 
$$Y(-,\uz):F \rightarrow \End (F)[[z^\pm,\z^\pm,|z|^\R]],\; a\mapsto Y(a,\uz)=\sum_{r,s \in \R}a(r,s)z^{-r-1}\z^{-s-1}$$
and an element $\1 \in F_{0,0}$ %and a $\R^2$-graded subspace $F^\vee=\bigoplus_{r,s} F_{r,s}^\vee$ of the dual vector space $Hom_\C(F,\C)$ 
satisfying the following conditions:
%
% F^\vee is not nesesary?
\begin{enumerate}
\item[FV1)]
%There exists $N \in \R$ such that $F_{h,\h}=0$ for any $h \geq N$ or $\h \geq N$.
For any $a,b \in F$, there exists $N \in \R$ such that $a(r,s)b=0$
for any $r \geq N$ or $s \geq N$,
that is, $Y(a,\uz)b \in F((z,\z,|z|^\R))$;
\item[FV2)]
$F_{h,\h}=0$ unless $h-\h \in \Z$;
%For any $a \in F$ and $u \in F^\vee$, there exists $N \in \R$ such that
%$u(a(-r,-s)-)=0$ for any $r \geq N$ or $s \geq N$.
\item[FV3)]
For any $a \in F$, $Y(a,\uz)\1 \in F[[z,\z]]$ and $\lim_{z \to 0}Y(a,\uz)\1 = a(-1,-1)\1=a$.
\item[FV4)]
$Y(\1,\uz)=\mathrm{id} \in \End F$;
\item[FV5)]
%convergence
For any $a,b,c \in F$ and $u \in F^\vee$, there exists $\mu(z_1,z_2) \in \GCor_2$ such that
\begin{align*}
u(Y(a,\uz_1)Y(b,\uz_2)c) &= \mu|_{|z_1|>|z_2|}, \\
u(Y(Y(a,\uz_0)b,\uz_2)c) &= \mu|_{|z_2|>|z_1-z_2|},\\
u(Y(b,\uz_2)Y(a,\uz_1)c)&=\mu|_{|z_2|>|z_1|},
\end{align*}
where $z_0=z_1-z_2$.
\item[FV6)]
$F_{h,\h}(r,s)F_{h',\h'} \subset F_{h+h'-r-1,\h+\h'-s-1}$ for any $r,s,h,h',\h,\h' \in \R$.
%\item[FV6)]
%$F^\vee=\bigoplus_{n,m \in \R^2} Hom_\C (F_{n,m},\C) \cap F^\vee$
%finiteness of fusion
%\item[FV8)]
%For $n,m,n',m' \in \R$, there exists $l \in \Z_{>0}$ and $n_1,m_1,\dots,n_l,m_l \in \R$ such that 
%the subspace spanned by $\{a(r,s)b\}$
%for all $a \in F_{n,m}, b\in F_{n',m'}, r,s\in \R$ is contained in $\bigoplus_{i=1}^l F_{n_i,m_i}.$
\end{enumerate}

%\begin{rem}
%\label{axiom}
%The real analytic function $\mu \in \GCor_2$ is unique (see Remark \ref{unique_expansion})
%By (FV1) and (FV6), a full vertex algebra $F$ satisfies the following properties:
%\begin{enumerate}
%\item
%For any $a,b \in F$, there exists $N \in \R$ such that $a(r,s)b=0$
%for any $r \geq N$ or $s \geq N$.
%\item
%For any $a \in F$ and $u \in F^\vee$, there exists $N \in \R$ such that
%$u(a(-r,-s)-)=0$ for any $r \geq N$ or $s \geq N$.
%\end{enumerate}
%\end{rem}

\begin{rem}
Physically, the energy and the spin of a state in $F_{h,\h}$ are $h+\h$ and $h-\h$.
Thus, the condition (FV2) implies that we only consider the particles whose spin is an integer,
that is, we consider only bosons and not fermions.
The notion of a full super vertex algebra can be defined by modifying (FV5) and (FV2).
%We consider the condition (FV8) because of Remark \ref{rationality}
\end{rem}
\begin{rem}\label{L0_operator}
Define the linear map $L(0),\Ld(0)\in \End\,F$ by
 $L(0)v=hv$ and $\Ld(0)a=\h a$ for any $h,\h\in \R$ and $a\in F_{h,\h}$.
Then, the condition (FV6) is equivalent to the the following condition:
For any $h,\h\in \R$ and $a \in F_{h,\h}$,
\begin{align*}
[L(0),Y(a,\uz)]=(zd/dz+h)Y(a,\uz),\\
[\Ld(0),Y(a,\uz)]=(\z d/d\z+\h)Y(a,\uz).
\end{align*}
\end{rem}

Let $(F^1,Y^1,\1^1)$ and $(F^2,Y^2,\1^2)$ be full vertex algebras.
A full vertex algebra homomorphism from $F^1$ to $F^2$ is a linear map 
$f:F^1\rightarrow F^2$ such that
\begin{enumerate}
\item
$f(\1^1)=\1^2$
\item
$f(Y^1(a,\uz)-)=Y^2(f(a),\uz)f(-)$ for any $a\in F^1$.
\end{enumerate}
The notions of a subalgebra and an ideal are defined in the usual way.

A module of a full vertex algebra $F$ is an $\R^2$-graded $\C$-vector space
$M=\bigoplus_{h,\h \in \R^2} M_{h,\h} $ equipped with a linear map 
$$Y_M(-,\uz):F \rightarrow \End (M)[[z^\pm,\z^\pm,|z|^\R]],\; a\mapsto Y_M(a,z)=\sum_{r,s \in \R}a(r,s)z^{-r-1}\z^{-s-1}$$
%and a subspace $M^\vee$ of the dual vector space 
%$Hom_\C(M,\C)$
 satisfying the following conditions:
%A vertex algebra $(V,Y,\1)$ is a vector space $V$ and 
%with a non-zero distinguished vector $\mathbf{1}$ satisfying the following conditions:
\begin{enumerate}
\item[FM1)]
For any $a \in F$ and $m \in M$, there exists $N \in \R$ such that $a(r,s)m=0$
for any $r\geq N$ or $s\geq N$;
\item[FM2)]
$M_{n,m}=0$ unless $n-m \in \Z$;
\item[FM3)]
$Y_M(\1,\uz)=\mathrm{id} \in \End M$;
\item[FM4)]
%convergence
For any $a,b \in F$, $m \in M$ and $u \in M^\vee$, there exists $\mu \in \GCor_2$ such that
\begin{align*}
u(Y_M(a,\uz_1)Y_M(b,\uz_2)m) &= \mu|_{|z_1|>|z_2|}, \\
u(Y_M(Y_M(a,\uz_0)b,\uz_2)m) &= \mu|_{|z_2>|z_1-z_2|},\\
u(Y_M(b,\uz_2)Y_M(a,\uz_1)m)&=\mu|_{|z_2|>|z_1|};
\end{align*}
%\item[FV6)]
% For $m \in M$, if $u(m)=0$ for any $u \in M^\vee$,
%then $m=0$.
\item[FM5)]
$F_{h,\h}(r,s)M_{h',\h'} \subset M_{h+h'-r-1,\h+\h'-s-1}$ for any $r,s,h,h',\h,\h' \in \R$.
%\item
%$M^\vee$ is graded.
%\item[FM7)]
%For $n,m,n',m' \in \R$, there exists $l \in \Z_{>0}$ and $n_1,m_1,\dots,n_l,m_l \in \R$ such that 
%the subspace spanned by $\{a(r,s)b\}$
%for all $a \in F_{n,m}, b\in F_{n',m'}, r,s\in \R$ is contained in $\bigoplus_{i=1}^l F_{n_i,m_i}.$
\end{enumerate}
%Recall that 
%\begin{align*}
%U(z_1,z_2)&=\C((z_2/z_1,\z_2/\z_1,|z_2/z_1|^\R))
%[z_1^\pm,\z_1^\pm,|z_1|^\R,z_2^\pm,\z_2^\pm,|z_2|^\R].
%\end{align*}
As a consequence of (FM1) and (FM5), we have:
\begin{lem}\label{generalized_formal}
Let $h_i,\h_i\in \R$, $a_i \in F_{h_i,\h_i}$ ($i=1,2$), $m \in M_{h_3,\h_3}$ and $u \in M_{h_0,\h_0}^*$.
Then, $u(Y(a_1,\uz_1)Y(a_{2},\uz_{2})m) \in z_1^{h_0-h_1-h_2-h_3}\z_1^{\h_0-\h_1-\h_2-\h_3}
\C((z_2/z_1,\z_2/\z_1,|z_2/z_1|^\R))$.
\end{lem}
\begin{proof}
Set
$$\sum_{s_1,\s_1,\s_2,\s_2\in \R}c_{s_1,\s_1,\s_2,\s_2}z_1^{s_1}\z_1^{\s_1}z_{2}^{s_{2}}\z_{2}^{s_{2}}=u(Y(a_1,\uz_1)Y(a_2,\uz_2)m).$$
Then, $$c_{s_1,\s_1,s_{2},\s_{2}}=
u(a_1(-s_1-1,-\s_1-1)a_2(-s_2-1,-\s_2-1)m).$$
%By (FVO4), there exists $N \in \Z$ such
%that $F_{h,\h}=0$ for any $h < N$ or $\h < N$.
By (FM5),
 $a_1(-s_1-1,-\s_1-1)a_2(-s_2-1,-\s_2-1)m\in 
M_{h_1+h_2+h_3+s_1+s_2,\h_1+\h_2+\h_3+\s_1+\s_2}$.
Hence,
$c_{s_1,\s_1,s_{2},\s_{2}}=0$
unless
%\begin{align*}
$h_0=h_1+h_2+h_3+s_1+s_2$ and
$\h_0=\h_1+\h_2+\h_3+\s_1+\s_2$.
%\end{align*}
Thus, we have 
\begin{align*}
u(Y(a_1,\uz_1)Y(a_2,\uz_2)m)
=z_1^{h_0-h_1-h_2-h_3}\z_1^{\h_0-\h_1-\h_2-\h_3}
\sum_{s_2,\s_2 \in \R}c_{s_1,\s_1,\s_2,\s_2}(z_{2}/z_1)^{s_{2}}
(\z_{2}/\z_1)^{s_{2}},
\end{align*}
where $s_1=h_0-(h_1+h_2+h_3+s_2)$
and $\s_1=\h_0-(\h_1+\h_2+\h_3+\s_2)$.
By (FM1), the assertion holds.
\end{proof}

By Lemma \ref{generalized_formal} and Lemma \ref{generalized_limit},
we have:
\begin{lem}\label{borcherds}
Let $h_i,\h_i\in \R$, $a_i \in F_{h_i,\h_i}$ ($i=1,2$), $m\in M_{h_3,\h_3}$ and $u \in M_{h_0,\h_0}^*$, there exists $f \in \F$ such that
\begin{align*}
z_2^{-h_0+h_1+h_2+h_3}\z_2^{-\h_0+\h_1+\h_2+\h_3}
u(Y(a,\uz_1)Y(b,\uz_2)m) &=\lim_{p\to z_1/z_2} j(p,f), \\
z_2^{-h_0+h_1+h_2+h_3}\z_2^{-\h_0+\h_1+\h_2+\h_3}u(Y(Y(a,\uz_0)b,\uz_2)m) &= \lim_{p\to -z_0/z_2} j(1-p^{-1},f),\\
z_2^{-h_0+h_1+h_2+h_3}\z_2^{-\h_0+\h_1+\h_2+\h_3}u(Y(b,\uz_2)Y(a,\uz_1)m)&=\lim_{p\to z_2/z_1} j(1/p,f).
\end{align*}
\end{lem}

Let $M,N$ be a $F$-module.
A $F$-module homomorphism from $M$ to $N$ is a linear map $f:M\rightarrow N$
such that $f(Y_M(a,\uz)-)=Y_N(a,\uz)f(-)$ for any $a\in F$.

Let $M$ be a $F$-module.
According to \cite{Li},
a vector $v \in M$ is said to be a vacuum-like vector
if $Y(a,\uz)v \in M[[z,\z]]$ for any $a\in F$.
\begin{lem}
\label{vacuum}
Let $v \in M$ be a vacuum-like vector and $a,b \in F$ and $u \in M^\vee$ and 
$\mu \in \GCor_2$ satisfy $u(Y(a_1,\uz_1)Y(a_2,\uz_2)v)=\mu|_{|z_1|>|z_2|}$.
Then, $\mu$ is a  linear combination of the functions of the form
$(z_1-z_2)^{\al_{12}} z_2^{\al_2}(\bar{z}_1-\bar{z}_2)^{\be_{12}} \z_2^{\be_2},$ where $\al_2,\be_2 \in \Z_{\geq 0}$ and $\al_{12},\be_{12} \in \R$ satisfy $\al_{12}-\be_{12}\in \Z$.
Furthermore, the linear function $F_v: F \rightarrow M$ defined by
$a \mapsto a(-1,-1)v$ is a $F$-module homomorphism.
\end{lem}
\begin{proof}
By (FM4), $u(Y(Y(a_1,\uz_0)a_2,\uz_2)v)=\mu|_{|z_1|>|z_1-z_2|}$.
Since $v$ is a vacuum like vector, by Lemma \ref{generalized_formal}
 $p(z_0,z_2)=\mu|_{|z_1|>|z_1-z_2|}
 \in \C[z_0^{\pm},\z_0^{\pm},|z_0|^\R, z_2,\z_2]\subset U(z_2,z_0)$,
which proves the first part of the lemma.
It suffices to show that $F_v(Y(a_1,\uz_0)a_2)=Y(a_1,\uz_0)F_v(a_2)$.
Since
\begin{align*}
u(Y(a_1,\uz_1)Y(a_2,\uz_2)v)&=\mu|_{|z_1|>|z_2|}\\
&=\lim_{z_0\to (z_1-z_2)|_{|z_1|>|z_2|}}p(z_0,z_2),
\end{align*}
we have
\begin{align}
u(Y(a_1,\uz_0)Y(a_2,\uz_2)v)=\exp(-z_2d/dz_0-\z_2d/d\z_0)u(Y(Y(a_1,\uz_0)a_2,\uz_2)v).
\end{align}
Thus,
\begin{align*}
Y(a_1,\uz_0)F_v(a_2)&=\lim_{z_2 \mapsto 0} u(Y(a_1,\uz_0)Y(a_2,\uz_2)v)\\
&=\lim_{z_2 \mapsto 0}\exp(-z_2d/dz_0-\z_2d/d\z_0)u(Y(Y(a_1,\uz_0)a_2,\uz_2)v)\\
&=F_v(Y(a_1,\uz_0)a_2).
\end{align*}
\end{proof}

Let $F$ be a full vertex algebra
and $D$ and $\D$ denote the endomorphism of $F$
defined by $Da=a(-2,-1)\bm{1}$ and $\D a=a(-1,-2)$ for $a\in F$,
i.e., $$Y(a,z)\1=a+Daz+\D a\z+\dots.$$
Define $Y(a,-\uz)$ by
$Y(a,-\uz)=\sum_{r,s}(-1)^{r-s} a(r,s)z^r \z^s$,
where we used $a(r,s)=0$ for $r-s \notin \Z$,
which follows from (FV2) and (FV6).
\begin{prop}
\label{translation}
For $a \in F$, the following properties hold:
\begin{enumerate}
\item
$Y(Da,\uz)=d/dz Y(a,\uz)$ and $Y(\D a,\uz)=d/d\z Y(a,\uz)$;
\item
$D\1=\D\1=0$;
\item
$[D,\D]=0$;
\item
$Y(a,\uz)b=\exp(zD+\z\D)Y(b,-\uz)a$;
\item
$Y(\D a,\uz)=[\D,Y(a,\uz)]$ and $Y(Da,\uz)=[D,Y(a,\uz)]$.
\end{enumerate}
\end{prop}
\begin{proof}
Let $u \in F^\vee$ and $a, b \in F$ and $\mu_1,\mu_2\in \GCor_2$
satisfy 
\begin{align*}
u(Y(a,\uz_1)Y(\1,\uz_2)b)=\mu_1|_{|z_1|>|z_2|},
%u(Y(a_1,z_1)Y(\1,z_2)a_2)=\mu_2|_{z_1>z_2},
u(Y(a,\uz_1)Y(b,\uz_2)\1)=\mu_2 |_{|z_1|>|z_2|}.
\end{align*}
By (FV4) and (FV5), $p_1(z_1)=\mu_1|_{|z_1|>|z_2|} \in \C[z_1^\pm,\z_1^\pm,|z_1|^\R]$.
Then,
\begin{align*}
u(Y(Y(a,\uz_0)\1,\uz_2)b)=\mu_1|_{|z_2|>|z_1-z_2|}
= \lim_{z_1\rightarrow z_2} \exp(z_0\frac{d}{dz_1}) \exp(\z_0\frac{d}{d\z_1})p_1(z_1).
\end{align*}
Thus, $u(Y(Da,\uz_2)b)=\lim_{z_1\rightarrow z_2} \frac{d}{dz_1} p_1(z_1)=
\frac{d}{dz_2}u(Y(a,z_2)b)$,
which implies that $Y(Da,\uz)=\frac{d}{dz}Y(a,\uz)$ and similarly $Y(\D a,\uz)=\frac{d}{d\z}Y(a,\uz)$.

By (FV4), $Y(D\1,\uz)=\frac{d}{dz}Y(\1,\uz)=0$.
Thus, by (FV3), $D\1=\D\1=0$.
Since $Y(D\D a,\uz)=\frac{d}{dz}\frac{d}{d\z}Y(a,\uz)=\frac{d}{d\z}\frac{d}{dz}Y(a,\uz)=Y(\D Da,\uz)$,
we have $[D,\D]=0$.

By Lemma \ref{vacuum}, $\mu_2|_{|z_2|>|z_1-z_2|} \in \C[z_2,\z_2][z_0^{\pm},\z_0^{\pm},|z_0|^\R]$.
Set $p(z_0,z_2)=\mu_2|_{|z_2|>|z_1-z_2|}=u(Y(Y(a,\uz_0)b,\uz_2)\1)$.
Since 
$u(Y(Y(b,-\uz_0)a,\uz_1)\1) = p(z_0,z_1-z_0)|_{|z_1|>|z_0|}$,
we have
\begin{align*}
u(Y(a,\uz_0)b)&=p(z_0,0)= \lim_{z_1\to 0} \exp(z_0\frac{d}{dz_1}+\z_0\frac{d}{d\z_1})p(z_0,z_1-z_0)\\
&=\lim_{z_1\to 0} \exp(z_0\frac{d}{dz_1}+\z_0\frac{d}{d\z_1})u(Y(Y(b,-\uz_0)a,\uz_1)\1) \\
&=\lim_{z_1\to 0} u(Y(\exp(z_0D+\z_0\D) Y(b,-\uz_0)a,\uz_1)\1) \\
&=u(\exp(z_0D+\z_0\D) Y(b,-\uz_0)a).
\end{align*}

Finally,
\begin{align*}
\frac{d}{dz}Y(a,\uz)b&=\frac{d}{dz}\exp(Dz+\D\z)Y(b,-\uz)a \\
&= D\exp(Dz+\D\z)Y(b,-\uz)a -\exp(Dz+\D\z)Y(Db,-\uz)a\\
&=DY(a,\uz)b-Y(a,\uz)Db.
\end{align*}
\end{proof}

Let $(V,Y,\1)$ be a $\Z$-graded vertex algebra.
Then, by a standard result of the theory of a vertex algebra (see for example \cite{FLM,FB}),
$u(Y(a,z_1)Y(b,z_2)c)$ is an expansion of a rational polynomial
in $\C[z_1^\pm,z_2^\pm,(z_1-z_2)^{-1}] \subset \GCor_2$ in $|z_1|>|z_2|$
for any $u\in V^\vee$ and $a,b,c\in V$.
Thus, we have:
%Then, 
\begin{prop}\label{graded_vertex}
A $\Z$-graded vertex algebra is a full vertex algebra.
\end{prop}

Let $(F,Y,\1)$ be a full vertex algebra.
Set $\bar{F}=F$ and
$\bar{F}_{h,\h}=F_{\h,h}$ for $h,\h\in \R$.
Define $\bar{Y}(-,\uz):\bar{F} \rightarrow \End (\bar{F})[[z,\z,|z|^\R]]$
by $\bar{Y}(a,\uz)=\sum_{s,\s \in \R}a(s,\s)\z^{-s-1}z^{-\s-1}$.
Let $C:Y_2\rightarrow Y_2$ be the conjugate map
$(z_1,z_2)\mapsto (\z_1,\z_2)$ for $(z_1,z_2)\in Y_2$.
For $u\in \bar{F}^\vee$ and $a,b,c\in \bar{F}$,
let $\mu \in \GCor_2$ satisfy
$u(Y(a,\uz_1)Y(b,\uz_2)c)=\mu(z_1,z_2)|_{|z_1|>|z_2|}$.
Then, 
$u(\bar{Y}(a,\uz)\bar{Y}(b,\uz)c)=\mu\circ C(z_1,z_2)$.
Since $\mu \circ C \in \GCor_2$, we have:
\begin{prop}\label{conjugate}
$(\bar{F},\bar{Y},\1)$ is a full vertex algebra.
\end{prop}
We call it a conjugate full vertex algebra of $(F,\1,Y)$.

% consequences of the axioms
% vacuum is translation invariant since 4-point function is 
% translation invariant!!.

% S_3 symmetry and 3-point correlation functions

% S_4 symmetry to S_3 symmetry implies that ...

\subsection{Holomorphic vertex operators}
Let $F$ be a full vertex algebra.
A vector $a \in F$ is said to be a holomorphic vector (resp. an anti-holomorphic vector)
if $Da=0$ (resp. $\D a=0$).
Let $a \in \ker \D$.
Then, since $0=Y(\D a,\uz)=d/d\z Y(a,\uz)$,
we have $a(r,s)=0$ unless $s=-1$.
Hence, $Y(a,\uz)=\sum_{n \in \Z} a(n,-1) z^{-n-1}$.

\begin{lem}
\label{hol_commutator}
Let $a,b\in F$.
If $\D a=0$,
then for any $n\in \Z$,
\begin{align*}
[a(n,-1),Y(b,\uz)]&= \sum_{i \geq 0} \binom{n}{i} Y(a(i,-1)b,\uz)z^{n-i},\\
Y(a(n,-1)b,\uz)&= 
\sum_{i \geq 0} \binom{n}{i}(-1)^i a(n-i,-1)z^{i}Y(b,\uz)
-Y(b,\uz)\sum_{i \geq 0} \binom{n}{i}(-1)^{i+n} a(i,-1)z^{n-i}.
\end{align*}
\end{lem}
\begin{proof}
For any $u\in F^\vee$ and $c\in F$,
there exists $\mu \in \GCor_2$ such that (FV5) holds.
Since $\D a=0$, by Proposition \ref{translation},
$d/d\z_1 \mu(z_1,z_2)=0$.
Then, by Lemma \ref{hol_generalized}, $\mu \in \C[z_1^\pm,(z_1-z_2)^\pm,z_2^\pm,\z_2^\pm,|z_2|^\R]$.
Thus, by the Cauchy integral formula, the assertion holds.
\end{proof}

By Proposition \ref{translation}, $\D Y(a,\uz)b =Y(\D a,\uz)b+Y(a,\uz)\D b=0$.
Thus, the restriction of $Y$ on $\ker \D$ define a linear map $Y(-,z): \ker \D \rightarrow \End\;\ker \D[[z^\pm]]$.
By the above Lemma and Lemma \ref{hol_generalized}, we have:
\begin{prop}\label{vertex_algebra}
$\ker \D$ is a vertex algebra and $F$ is a $\ker \D$-module.
\end{prop}
\begin{comment}
\begin{proof}
In order to prove that $\ker \D$ is a vertex algebra,
it suffices to show that $\ker \D$ satisfies the Goddard's axioms \cite{LL}.
Since $[D,\D]=0$, $D$ acts on $\ker \D$.
By Proposition \ref{translation}, it suffices to show that $Y(a,z)$ and $Y(b,w)$ are mutually local for any $a,b \in \ker \D$.
Let $a,b \in \ker \D$ and $v \in F$, $u\in F^\vee$ and $\mu \in \GCor_2$ satisfy $u(Y(a,z_1)Y(a_2,z_2)v) = \mu|_{|z_1|>|z_2|}$.
By Lemma \ref{hol_generalized},
$\mu$ is a polynomial in $\C[z_1^\pm,z_2^\pm,(z_1-z_2)^{-1}]$.
Since $\mu|_{|z_2|>|z_1-z_2|}=u(Y(Y(a_1,z_0)a_2,z_2)v)$ and $a_1(n,-1)a_2=0$ for sufficiently large $n\in \Z$,
there exists $N \in \Z_{\geq 0}$ such that $(z_1-z_2)^N \mu  \in \C[z_1^\pm,z_2^\pm]$.
Thus, $(z_1-z_2)^N u(Y(a,z_1)Y(a_2,z_2)v)=(z_1-z_2)^N u(Y(a_2,z_2)Y(a_1,z_1)v)$ for any $v \in F$ and $u \in F^\vee$,
which implies that $Y(a_1,z_1)$ and $Y(a_2,z_2)$ are mutually local
and $F$ is a $\ker \D$-module (see, for example, \cite[Proposition 4.4.3]{LL}).
\end{proof}
Furthermore, similarly to the proof of the above proposition, we have:
\end{comment}

\begin{lem}
\label{hol_commute}
For a holomorphic vector $a \in F$ and an anti-holomorphic vector $b\in F$, 
$[Y(a,z),Y(b,\bar{w})]=0$, that is,
$[a(n,-1),b(-1,m)]=0$ and $a(k,-1)b=0$ for any $n,m \in \Z$ and $k \in \Z_{\geq 0}$.
\end{lem}
\begin{proof}
By Lemma \ref{hol_commutator},
it suffices to show that $a(k,-1)b=0$ for any $k \geq 0$.
Since $DY(a,z)b=[D,Y(a,z)]b+Y(a,z)Db=d/dzY(a,z)b$,
we have $Da(n,-1)b=-na(n-1,-1)b$ for any $n \in \Z$.
Thus,  the assertion follows from (FV1).
\end{proof}

\subsection{Tensor product of full vertex algebras}
In this section, we define a tensor product of full vertex algebras
and study the subalgebra of a full vertex algebra
generated by holomorphic and anti-holomorphic vectors.

Let $F$ be a full vertex algebra.
For $h,\h \in \R$, the energy and spin of a vector in $F_{h,\h}$ are $h+\h$ and $h-\h$ and
the set $\{(h,\h)\in \R^2\;|\;F_{h,\h}\neq 0\}$ is called
a spectrum.
The spectrum of $F$ is said to be
{\it bounded below} if
there exists $N\in \R$ such that
$F_{h,\h}=0$ for any $h \leq N$ or $\h\leq N$
and {\it discrete} if for any $H \in \R$, $\sum_{h+\h < H} \dim F_{h,\h}$ is finite.

\begin{lem}\label{finite_product}
If the spectrum of $F$ is discrete,
then for any $N \in \Z_{>0}$ and $a,b \in F$,
the number of the set 
$$\{(s,\s)\in \R^2 \;|\; a(s,\s)b \neq 0, -N \leq s,\s \leq N  \}
$$
is finite.
\end{lem}
\begin{proof}
We may assume that $a \in F_{h,\h}$ and $b \in F_{h',\h'}$.
Since $a(s,\s)b \in F_{h+h'-s-1,\h+\h'-\s-1}$,
the energies of vectors $\{a(s,\s)b \;|\;  -N \leq s,\s \leq N  \}$ are bounded by $h+h'+\h+\h'+2N-2$.
Thus, the assertion holds.
\end{proof}

Let $(F^1,Y^1,\1^1)$ and $(F^2,Y^2,\1^2)$  be full vertex algebras
and assume that the spectrum of $F^1$ is discrete
and the spectrum of $F^2$ is bounded below.
Define the linear map $Y(-,\uz):F^1 \otimes F^2 \rightarrow \End F^1 \otimes F^2[[z,\z,|z|^\R ]]$ by 
$Y(a\otimes b,\uz)=Y^1(a,\uz)\otimes Y^2(b,\uz)$ for $a\in F^1$ and $b \in F^2$.
Then, for $a,c \in F^1$ and $b,d \in F^2$,
$$Y(a\otimes b, \uz)c\otimes d
=\sum_{s,\s,r,\bar{r} \in \R}a(s,\s)c\otimes b(r,\bar{r})d\,z^{-s-r-2}\z^{-\s-\bar{r}-2}.$$
By (FV1) and the above lemma,
the coefficient of $z^k\z^{\bar{k}}$ is a finite sum for any $k,\bar{k}\in \R$.
Thus, $Y(-,\uz)$ is well-defined.
For any $h_0,\h_0 \in \R$,
set $(F^1\otimes F^2)_{h_0,\h_0}=\bigoplus_{a,\bar{a} \in \R} F_{a,\bar{a}}^1\otimes F_{h_0-a,\h_0-\bar{a}}^2$.
Since the spectrum of $F^2$ is bounded below, there exists $N\in \R$ such that
$(F^1\otimes F^2)_{h_0,\h_0}=\bigoplus_{a,\bar{a} \leq N} F_{a,\bar{a}}^1\otimes F_{h_0-a,\h_0-\bar{a}}^2$.
Since the spectrum of $F^1$ is discrete, the sum is finite.
Thus, $(F^1\otimes F^2)_{h_0,\h_0}^*=\bigoplus_{a,\bar{a} \in \R} (F_{a,\bar{a}}^1)^*\otimes (F_{h_0-a,\h_0-\bar{a}}^2)^*$,
which implies that $F^\vee=(F^1)^\vee\otimes (F^2)^\vee$.
Let $u_i \in (F^i)^\vee$ and $a_i,b_i,c_i \in F^i$ for $i=1,2$.
Since
$$
u_1\otimes u_2(Y(a_1\otimes a_2,\uz_1)Y(b_1\otimes b_2,\uz_2)c_1\otimes c_2)
=u_1(Y(a_1,\uz_1)Y(b_1,\uz_2)c_1)u_2(Y(a_2,\uz_1)Y(b_2,\uz_2)c_2),
$$
we have:
\begin{prop}
\label{tensor}
Let $(F^1,Y^1,\1^1)$ and $(F^2,Y^2,\1^2)$  be full vertex algebras.
If the spectrum of $F^1$ is discrete
and the spectrum of $F^2$ is bounded below,
then $(F^1 \otimes F^2, Y^1\otimes Y^2 ,\1^1 \otimes \1^2)$ is a full vertex algebra.
Furthermore, if the spectrum of $F^1$ and $F^2$ are bounded below (resp. discrete),
then the spectrum of $F_1\otimes F^2$ is also bounded below (resp. discrete).
\end{prop}

By Proposition \ref{graded_vertex} and Proposition \ref{conjugate},
we have:
\begin{cor}\label{chiral_tensor}
Let $V, W$ be a $\Z_{\geq 0}$-graded vertex algebras
such that $\dim V_n$ and $\dim W_n$ are finite for any $n\in \Z_{\geq 0}$.
Then, $V\otimes \bar{W}$ is a full vertex algebra with a discrete spectrum,
where $\bar{W}$ is the conjugate full vertex algebra.
\end{cor}
Let $F$ be a full vertex algebra.
By Proposition \ref{vertex_algebra},
$\ker \D$ and $\ker D$ are subalgebras of $F$.
Let $\ker \D \otimes \ker D$ be the tensor product full vertex algebra.
Define the linear map $t:\ker \D \otimes \ker D \rightarrow F$
by $(a\otimes b)\mapsto a(-1,-1)b$ for $a\in \ker \D$ and $b \in \ker D$.
Then, we have:
\begin{prop}
Let $F$ be a full vertex algebra.
Then, $t:\ker \D \otimes \ker D \rightarrow F$ is a full vertex algebra homomorphism.
\end{prop}
\begin{proof}
Let $a,c \in \ker \D$, $b,d \in \ker D$.
By Lemma \ref{hol_commute} and Lemma \ref{hol_commutator},
$$Y(a(-1,-1)b,\uz)=Y(a,z)Y(b,\z)=Y(b,\z)Y(a,z).$$
Thus, it suffices to show that
$t(a\otimes b(n,m)c\otimes d)=t(a\otimes b)(n,m)t(c\otimes d)$
for any $n,m\in \Z$.
By Lemma \ref{hol_commutator}
\begin{align*}
t(a\otimes b(n,m)c\otimes d)&=t(a(n,-1)c \otimes b(-1,m)d)\\
&=(a(n,-1)c)(-1,-1)b(-1,m)d\\
&=\sum_{i=0}\binom{n}{i}(-1)^i (a(n-i,-1)c(-1+i,-1)+c(-1+n-i,-1)a(i,-1))b(-1,m)d.
\end{align*}
Since $b(-1,m)d \in \ker D$, by Lemma \ref{hol_commute},
$t(a\otimes b(n,m)c\otimes d)=a(n,-1)c(-1,-1)b(-1,m)d
=a(n,-1)b(-1,m)c(-1,-1)d=t(a\otimes b)(n,m)t(c\otimes d).$
Thus, the assertion holds.
\end{proof}

We remark that if $\ker \D\otimes \ker D$ is simple, then the above map is injective.

\section{Correlation functions and full vertex algebras}

\subsection{Self-duality}\label{sec_VOA}
%%
%% quasi-conformal full vertex algebra and 3-point function
%

A full vertex operator algebra (full VOA) is a full vertex algebra $F$ with a holomorphic vector $\om \in F$ and an anti-holomorphic vector 
$\omb \in F$
satisfying the following conditions:
\begin{enumerate}
\item[FVO1)]
%$\omega$ is a holomorphic field and $\omb$ is a anti-holomorphic filed.
%\item
There exist a pair of scalars $(c,\bar{c}) \in \C$ such that 
\begin{align*}
[L(n),L(m)]=(n-m)L(n+m)+\frac{n^3-n}{12}\delta_{n+m,0}c, \\
[\Ld(n),\Ld(m)]=(n-m)\Ld(n+m)+\frac{n^3-n}{12}\delta_{n+m,0}c
\end{align*}
holds for any $n,m \in \Z$, where $L(n)=\omega(n+1,-1)$ and $\Ld(n)=\omb(-1,n+1)$;
\item[FVO2)]
$D=L(-1)$ and $\D=\Ld(-1)$;
\item[FVO3)]
$L(0)|_{F_{h,\h}}=h$ and $\Ld(0)|_{F_{h,\h}}=\h$ for any $h,\h \in \R^2$;
\item[FVO4)]
$F_{h,\h}$ is a finite dimensional vector space for any $h,\h \in \R$;
\item[FVO5)]
The spectrum of $F$ is bounded below,
that is, there exists $N\in \R$ such that
$F_{h,\h} =0$ for any $h\leq N$ or $\h\leq N$.
\end{enumerate}
The pair of scalars $(c,\bar{c})$ is called a central charge and the pair $(\omega,\omb)$ is called an {\it energy-momentum
tensor} of the full vertex operator algebra $F$.
A module $M$ of a full vertex algebra $F$ is said to be a module of a full vertex operator algebra if it
satisfies
\begin{enumerate}
\item[FVOM1)]
$L(0)|_{M_{h,\h}}=h$ and $\Ld(0)|_{M_{h,\h}}=\h$ for any $h,\h \in \R$;
\item[FVOM2)]
$M_{h,\h}$ is a finite dimensional vector space for any $h,\h \in \R$;
\item[FVOM3)]
The spectrum of $M$ is bounded below.
\end{enumerate}

Let $F$ be a full VOA and $M$ be a $F$-module.
By Lemma \ref{hol_commutator}, we have:
\begin{lem}
\label{global_covariance}
For $h,\h \in \R$ and $a \in F_{h,\h}$,
\begin{align*}
[L(0),Y(a,\uz)]&=Y((L(0)+zL(-1))a,\uz)=(zd/dz+h)Y(a,\uz),\\
[\Ld(0),Y(a,\uz)]&=Y((\Ld(0)+\z\Ld(-1))a,\uz)=(\z d/d\z+\h)Y(a,\uz), \\
[L(1),Y(a,\uz)]&= Y((L(1)+2zL(0)+z^2L(-1)) a,\uz)   =(z^2d/dz+2hz)Y(a,\uz)+Y(L(1)a, \uz), \\
[\Ld(1),Y(a,\uz)]&=Y((\Ld(1)+2\z\Ld(0)+\z^2\Ld(-1)) a,\uz) =(\z^2d/d\z+2\h \z)Y(a,\uz)+Y(\Ld(1)a, \uz). 
\end{align*}
\end{lem}

Set $M^\vee=\bigoplus_{h,\h\in \R} M_{h,\h}^*$ and 
let $<>:M^\vee \times M \rightarrow \C$ be a canonical pairing.
Define $S_z: M \rightarrow M[z^\pm,\z^\pm,|z|^\R]$, by
$$S_z a= \exp({zL(1) +\z\Ld(1)})(-1)^{h-\h}z^{-2h}\z^{-2\h},$$
for $h,\h \in \R$ and $a \in M_{h,\h}$,
where we used the fact that $L(1)$ and $\Ld(1)$ are locally nilpotent by (FVOM3).
Define the vertex operator $Y_{M^\vee}(-,\uz):F \rightarrow 
\End M^\vee[[ z^\pm ,\z^\pm,|z|^\R ]]$ by
$$<Y_{M^\vee}(a,\uz)u, v>\equiv <u,Y_M(S_z a,\uz^{-1})v>,$$
for $a\in F$ and $v \in M$ and $u \in M^\vee$.
%grading is bounder below

We will prove the following Proposition:
\begin{prop}
\label{dual_module}
$(M^\vee,Y_{M^\vee}(-,\uz))$ is a module of the full VOA $F$.
\end{prop}
% examine the functions carefully. 
% then, we get those results.
%
Properties of the operators $\{L(i)\}_{i=-1,0,1}$ acting on a vertex operator algebra which satisfy the equations in Lemma \ref{global_covariance} are studied in \cite{FHL}, which can be easily generalized to a full vertex operator algebra.
\begin{lem}[\cite{FHL}]
\label{sl_2}
For a full vertex operator algebra $F$ and its module $M$, the following equations hold:
\begin{enumerate}
\item
$[L(0),\exp(L(1)z)]= -L(1)z \exp(L(1)z)$.
\item
For any $a \in F$,
$S_{z_2}Y_M(a,\uz_0)=Y_M(S_{z_2+z_0}a,\frac{-\uz_0}{\uz_2(\uz_2+\uz_0)})S_{z_2}$,
\end{enumerate}
where $z_2+z_0$ is expanded in $|z_2|>|z_0|$.
\end{lem}
\begin{proof}[proof of Proposition \ref{dual_module}]
Let $u \in F^\vee$ and $v \in M$ and $a,b \in F$.
By Lemma \ref{hol_commute} and (FVO1),
$L(1)\omega=L(1)\omb=\Ld(1)\om=\Ld(1)\omb=0$.
Thus, $<Y(\omega,\uz)u, v>= <u, z^4Y(\omega,\uz^{-1})v>$,
which implies that $<L(n)u,v>=<u,L(-n)v>$.
By Lemma \ref{global_covariance} and Lemma \ref{sl_2}, Proposition \ref{translation},
\begin{align*}
<[L&(0),Y_{M^\vee}(a,\uz)]u,v>-<Y_{M^\vee}(L(0)a,\uz)u,v>\\
&=<u,Y(S_z a,\uz^{-1})L(0)v>-<L(0)u,Y(S_z a,\uz^{-1})v>-<u,Y(S_z L(0)a,\uz^{-1})v>\\
&=- <u,Y((L(0)+z^{-1}L(-1))S_z a,\uz^{-1})v>-<u,Y(S_z L(0)a,\uz^{-1})v>\\
&=<u,Y(L(1)zS_za,\uz^{-1})v>-<u,Y(S_z 2L(0)a,\uz^{-1})v>
-z^{-1}<u,[L(-1),Y(S_z a,\uz^{-1})]v>.
\end{align*}
Since
\begin{align*}
zd/dz <u,Y(S_z &a,\uz^{-1})v>=<u,Y(L(1)z S_z a,\uz^{-1})v>\\
&-<u,Y(L(1)z S_z 2L(0) a,\uz^{-1})v>
-z^{-1}<u,[L(-1),Y(S_z a,\uz^{-1})]v>,
\end{align*}
we have $[L(0),Y_{M^\vee}(a,\uz)]=Y_{M^\vee}(L(0)a,\uz)+d/dz Y_{M^\vee}(a,\uz)$.
Thus, by Remark \ref{L0_operator}, $F_{h,\h}(r,s)M_{h',\h'}^\vee \subset M_{h+h'-r-1,\h+\h'-s-1}^\vee$.
(FM2), (FM3), (FVOM1), (FVOM2) and (FVOM3) clearly follows.
(FM1) follows from (FM5) together with (FVOM3).
Thus, it suffices to show (FM4).
By (FM4), there exists $\mu \in \GCor_2$ such that
\begin{align*}
<u,Y(S_{{z_2}^{-1}}b,\uz_2)Y(S_{{z_1}^{-1}}a,\uz_1)v>&=\mu|_{|z_2|>|z_1|},\\
<u,Y(S_{{z_1}^{-1}}a,\uz_1)Y(S_{{z_2}^{-1}}b,\uz_2)v>&=\mu|_{|z_1|>|z_2|},\\
 \lim_{z_1 \to (z_2+z_0)|_{|z_2|>|z_0|}}<u,Y(Y( S_{{z_1}^{-1}}a,\uz_0) S_{{z_2}^{-1}}b,\uz_2)v>&=\mu|_{|z_2|>|z_1-z_2|}.
\end{align*}
Then, $<Y(a,\uz_1)Y(b,\uz_2)u, m>=I_Y(\mu)|_{|z_1|>|z_2|}$ and $<Y(b,\uz_2)Y(a,\uz_1)u, m>=I_Y(\mu)|_{|z_2|>|z_1|}$.
To show (FM4), it suffices to show that
$I_Y(\mu)|_{|z_2|>|z_1-z_2|}=<Y(Y(a,\uz_0)b,\uz_2)u,v>$.
By Lemma \ref{dual_generalized},
\begin{align*}
I_Y(\mu) |_{|z_2|>|z_1-z_2|}&=\lim_{(z_0,z_2) \mapsto (\frac{-z_0}{z_2(z_2+z_0)},\frac{1}{z_2})} \mu|_{|z_2|>|z_1-z_2|}\\
&=\lim_{(z_0,z_2) \mapsto (\frac{-z_0}{z_2(z_2+z_0)},\frac{1}{z_2})}\lim_{z_1 \to z_2+z_0} <u,Y(Y( S_{{z_1}^{-1}}a,\uz_0) S_{{z_2}^{-1}}b,\uz_2)v>\\
&=<u,Y(Y(S_{z_2+z_0}a,\frac{-\uz_0}{\uz_2(\uz_2+\uz_0)})S_{z_2}b,\uz_2^{-1})v>.
\end{align*}
By Lemma \ref{sl_2},
\begin{align*}
<Y_M^\vee(Y(a,\uz_0)b,\uz_2)u,v>&=
<u,Y(S_{z_2}Y(a,\uz_0)b,\uz_2^{-1})v>\\
&=<u,Y(Y(S_{z_2+z_0}a,\frac{-\uz_0}{\uz_2(\uz_2+\uz_0)})S_{z_2}b,\uz_2^{-1})v>.
\end{align*}
Thus, the assertion holds.
\end{proof}

An invariant bilinear form on a full vertex operator algebra $F$ is a bilinear form
$(-,-):F \times F \rightarrow \C$ such that
$$(Y(a,\uz)b,c)=(b,Y(S_z a,\uz^{-1})c)$$
holds for any $a,b,c \in F$.
By the proof of Proposition \ref{dual_module},
$(L(n)a,b)=(a,L(-n)b)$, which implies that $F_{h,\h}$ and $F_{h',\h'}$ are orthogonal to each other unless $(h,\h)=(h',\h')$.

The following proposition is a straightforward generalization of Li's result on the invariant bilinear form on a vertex operator algebra (see \cite{Li}):
\begin{prop}
\label{invariant_bilinear}
There is a one-to-one correspondence between
$Hom_\C(F_{0,0}/(L(1)F_{1,0}+\Ld(1)F_{0,1}),\C)$ and 
the space of invariant bilinear forms on $F$.
\end{prop}
According to \cite{Li},
we will use the following lemma:
\begin{lem}
\label{Li_lemma}
Let $M$ be an $F$-module.
A vector $v \in M$ is a vacuum-like vector if and only if $L(-1)v=\Ld(-1)v=0$.
\end{lem}
\begin{proof}
If $v$ is a vacuum-like vector, then $Y(\omega,z)v \in F[[z]]$.
Thus,  $L(-1)v=0$ and similarly $\Ld(-1)v=0$.
Assume that $L(-1)v=\Ld(-1)v=0$.
Then, $L(-1)Y_M(a,\uz)v=[L(-1),Y_M(a,\uz)]v=d/dz Y_M(a,\uz)v$.
Thus, similarly to the proof of Lemma \ref{hol_commute}, the assertion holds.
\end{proof}

\begin{proof}[proof of Proposition \ref{invariant_bilinear}]
Let $I_F$ be the space of invariant bilinear forms on $F$
and $(-,-)\in I_F$.
Then, $(\1,-):F \rightarrow \C,\; a\mapsto (\1,a)$ satisfies
$(\1, F_{h,\h})=0$ for any $(h,\h) \neq (0,0)$.
Since $L(-1)\1=\Ld(-1)\1=0$, we have $(\1, L(1)-)=(\1,\Ld(1)-)=0$.
Thus, we have a linear map $\rho :I_F \rightarrow Hom_\C(F_{0,0}/(L(1)F_{1,0}+\Ld(1)F_{0,1}),\C)$.
Assume that $(\1,-)=0$.
Then, for any $a,b \in F$, 
$(a,b)=\lim_{\uz \to 0} (Y(a,\uz)\1,b)= \lim_{z \to 0} (\1, Y(S_z a, \uz^{-1})b)=0$.
Hence, $\rho$ is injective.

Let $u \in Hom_\C(F_{0,0}/(L(1)F_{1,0}+\Ld(1)F_{0,1}),\C)$.
Then, $u$ is an element of $F_{0,0}^* \subset F^\vee$.
Since $<L(-1)u, a>=<u,L(1)a>$ for any $a \in F$,
by Lemma \ref{Li_lemma}, $u$ is a vacuum-like vector.
Hence,  by Lemma \ref{vacuum}, we have a $F$-module homomorphism 
$F_u: F \rightarrow F^\vee,\; a \mapsto a(-1,-1)u$.
Define the bilinear form $(-,-)_u:F\times F \rightarrow \C$ by
$(a,b)_u = <F_u(a),b>$ for $a,b \in F$.
Then, 
\begin{align*}
(Y(a,\uz)b,c)_u&=<F_u(Y(a,\uz)b), c>=<Y_{F^\vee}(a,\uz)F_u(b),c>\\
&= <F_u(b), Y(S_z a,\uz^{-1})c >=(b,Y(S_z a,\uz^{-1})c )_u,
\end{align*}
which implies that $(-,-)_u$ is an invariant bilinear form.
Since $(\1,-)_u=<F_u(\1),b>=<u,b>$,
$\rho( (-,-)_u)=u$. Thus, $\rho$ is an isomorphism.
\end{proof}

Similarly to \cite{FHL}, we have:
\begin{prop}
\label{symmetric}
An invariant bilinear form on $F$ is symmetric.
\end{prop}
\begin{prop}
A full vertex operator algebra $F$ is simple if and only if $F^\vee$ is simple.
\end{prop}
%%
%% prove this !!
%%
A full vertex operator algebra $F$ is said to be self-dual if $F$ is isomorphic to $F^\vee$ as an $F$-module,
or equivalently, there exists a non-degenerate invariant bilinear form on $F$.
\begin{cor}
\label{existence_bilinear}
If $F$ is simple and $F_{0,0}=\C \1$ and $L(1)F_{1,0}=\Ld(1)F_{0,1}=0$,
then $F$ is self-dual.
\end{cor}

For a self-dual full vertex operator algebra $F$ and $a\in F$,
$(\1,Y(a,\uz)\1)=(\1,\exp(Dz+\D\z)a)=(\exp(L(1)z+\Ld(1)z)\1, a)=(\1,a)$.
Thus, we have:
\begin{lem}\label{one_point_cor}
For $a \in F$,
$(\1,Y(a,\uz)\1)$ is equal to $(\1,a) \in \C$.
\end{lem}

\subsection{Quasi-primary vectors}\label{sec_QP}
For $h,\h \in \R$,
set $$QF_{h,\h}=\{v \in F_{h,\h}\;|\; L(1)v=\Ld(1)v=0 \}.$$
A homogeneous vector in $QF=\bigoplus_{h,\h \in \R}QF_{h,\h}$ is called a quasi-primary vector.
Recall that  $L(i),\Ld(i)$ ($i=-1,0,1$) generate the Lie algebra $\mathrm{sl}_2 \oplus \mathrm{sl}_2$.
Following the terminology of \cite{Li}, we call a full vertex operator algebra $F$ {\it QP-generated} if
$F$ is generated by $QF$ as a $\mathrm{sl}_2 \oplus \mathrm{sl}_2$-module,
or equivalently, $F=\C[L(-1),\Ld(-1)]QF$.

\begin{comment}
For a mobius full vertex algebra,
% Li symmetric invariant 

By Lemma \ref{mobius_cor}, we have a bilinear form
$$QF_{h,\h} \times QF_{h,\h} \rightarrow \C\; (a,b) \mapsto C_{a,b}.$$
\begin{lem}
\label{qp_nondege}
If $F$ is QP-generated, then the bilinear form on $QF_{n,m}$ is symmetric and non-degenerate, that is, there exists $b \in Q_{n,m}$ such that $C_{a,b}\neq 0$ for any
$a \in Q_{n,m}$.
\end{lem}

\begin{proof}
Let $a,b \in QF_{n,m}$.
By Proposition \ref{translation}, $C_{a,b}z^{-2n}\z^{-2m}=<Y(a,z)b>=
<e^{zD+\z\D}Y(b,-z)a>=<Y(b,-z)a>=C_{b,a}(-z)^{-2n}(-\z)^{2m}$.
By (M7), $(-z)^{-2n}(-\z)^{2m}=|-z|^{-2n}(-\z)^{2(n-m)}=z^{-2n}\z^{-2m}.$
Thus, $C_{a,b}=C_{b,a}$.
Suppose that $a \neq 0$.
By (FV7), there exits $b \in F$ such that $<Y(b,z)a> \neq 0$,
by skew-symmetry, which is equivalent to $<Y(a,z)b>\neq 0$.
Since $b$ is in $\C[D,\D]QF$ and $<Y(a,z)D->=-d/dz<Y(a,z)->$,
there exists $b'\in QF$ such that $<Y(a,z)b'>\neq 0$.
By Lemma \ref{mobius_cor}, the bilinear form is non-degenerate.
\end{proof}
\end{comment}

\begin{prop}\label{qp_generated}
A self-dual full vertex operator algebra $F$ is QP-generated if and only if the following conditions hold:
\begin{enumerate}
\item
$F_{h,\h}=0$ if $h \in \frac{1}{2}\Z_{<0}$ or $\h \in \frac{1}{2}\Z_{<0}$.
\item
$L(1)F_{1,n}=0$ and $\Ld(1)F_{n,1}=0$ for any $n \in \Z$.
\item
$L(-1)F_{0,n}=0$ and $\Ld(-1) F_{n,0}=0$ for any $n \in \Z$.
\end{enumerate}
\end{prop}

\begin{lem}
\label{qp_non_degenerate}
If $F$ is self-dual and QP-generated,
then the restriction of the invariant bilinear form on $QF$ is non-degenerate.
Furthermore, $\mathrm{Im} \,L(-1) \cap \ker L(1)=0$ and $\mathrm{Im} \,\Ld(-1) \cap \ker \Ld(1)=0$ hold.
\end{lem}
\begin{proof}
Assume that $a \in QF_{h,\h}$ satisfy $(a,b)=0$ for any $b \in QF_{h,\h}$.
Since $(a,L(-1)b)=(L(1)a,b)=0$ and $(a,\Ld(-1)b)=(\Ld(1)a,b)=0$ for any $b \in F$,
by $F=\C[L(-1),\Ld(-1)]QP$, $a=0$, which proves the first part of the lemma.
Set $QF'=\{v \in F\;|\; L(1)v=0 \}$.
Since $F=\C[L(-1)][\Ld(-1)]QF$ and $L(1)\C[\Ld(-1)]QF=0$,
we have $F=\C[L(-1)]QF'$.
Then, by the same proof, the restriction of the invariant bilinear form on $QF'$ is non-degenerate.
Since $\mathrm{Im} \,L(-1) \cap \ker L(1)$ is in the radical of this invariant bilinear form,
the assertion holds.
\end{proof}

\begin{proof}[proof of Proposition \ref{qp_generated}]
Suppose that $F$ is QP-generated.
Let $0 \neq v \in QF_{-n,m}$ for $n \in \frac{1}{2}\Z_{>0}$.
Then, $L(1)L(-1)^{2n+1}v=0$.
By the representation theory of $\mathrm{sl}_2$ and Lemma \ref{qp_non_degenerate}, we have $0 \neq L(-1)^{2n}v \in F_{n,m}$ and $L(-1)^{2n+1}v=0$.
Since  $L(-1)^{2n+1}v=0$ implies 
$L(0)L(-1)^{2n}v=0$, a contradiction.
Hence, $QF_{-n,m}=0$ for any $n \in \frac{1}{2}\Z_{>0}$.

Let $v \in F_{0,n}$. Since $L(1)L(-1)v=0$, by Lemma \ref{qp_non_degenerate}, $L(-1)v=0$, which implies (3).
Since $F=\C[L(-1),\Ld(-1)]QF$, by (1) and (3), $F_{1,n}=\bigoplus_{k \geq 0}\Ld(-1)^k QF_{1,n-k}$. Since $L(1)$ commutes with $\Ld(-1)$, we have $L(1)F_{1,n}=0$, thus, (2).

Suppose that $F$ satisfies (1), (2) and (3).
Let $a,b \in QF_{n,m}$ with $n,m \in \R \setminus \frac{1}{2}\Z_{<0}$.
Since $$(L(-1)^k\Ld(-1)^l a,L(-1)^k\Ld(-1)^l b) =k!l! \Pi_{\substack{k-1 \geq i \geq 0 \\ l-1 \geq j \geq 0}}(2n+i)(2m+j) (a,b),$$
the restriction of the bilinear form on $L(-1)^k\Ld(-1)^lQ_{n,m}$ is non-degenerate if $(-,-)|_{Q_{n,m}}$ is
non-degenerate.
Let $h,\h \in \R$ satisfy $F_{h,\h}\neq 0$ and $F_{h-k,\h-l}=0$ for any $k,l \in \Z_{>0}$. 
Then, $QF_{h,\h}=F_{h,\h}$, which implies that $(-,-)|_{QF_{h,\h}}$ is non-degenerate.
Since $QF_{h+1,\h}=\{v \in F_{h+1,\h}\;|\; (v,L(-1)F_{h,\h})=0 \}$,
$QF_{h+1,\h}$ is non-degenerate and $F_{h+1,\h}=QF_{h+1,\h}\oplus L(-1)Q_{h,\h}$.
Similarly, we have $F_{a+k,b+l}=\bigoplus_{\substack{0 \leq i \leq k \\ 
0 \leq j \leq l}} L(-1)^{k-i} \Ld(-1)^{l-j} QF_{a+i,b+j}$ by the induction on $k,l \in \Z_{\geq 0}$.
%For $a \in \frac{1}{2}\Z$, by a similar argument with (1), (2) and (3), the assertion holds.
\begin{comment}
Since $L(-1) QF_{0,n}=0$,
$L(-1)F_{k,l}=\bigoplus_{\substack{0 < i \leq k \\ 
0 \leq j \leq l}} D^{k+1-i} \D^{l-j} QF_{i,j}$.
(2), $F_{0,0}=QF_{0,0}$ and $F_{1,0}=QF_{1,0}$ for any $n \geq 0$.
$L(1): L(-1)F_{k,l} \rightarrow F_{k,l}$
\end{comment}
\end{proof}
\begin{comment}
\begin{prop}\label{vacuum_space}
Let $F$ be a self-dual QP-generated full vertex operator algebra
such that $F$ is a countable dimension over $\C$.
Then, $F_{0,0}=\C \1$.
\end{prop}
We will use the following Proposition \cite{M1}:
\begin{prop}[\cite{Mo}]
Let $V$ be a simple vertex algebra with countable dimension over $\C$.
Then, the kernel of the translation map $D:V\rightarrow V,\; v \mapsto v(-2)\1$ is $\C\1$.
\end{prop}
\begin{proof}[proof of Proposition \ref{vacuum_space}]
By Proposition \ref{qp_generated}, $F_{0,0} = \ker D \cap \ker \D$.
Combining this with the above proposition and Lemma \ref{vertex_algebra}, the assertion holds.
\end{proof}
\end{comment}
By Proposition \ref{qp_generated} and Lemma \ref{existence_bilinear}, we have:
\begin{cor}\label{grading_qp}
Suppose that simple full vertex operator algebra $F$ satisfies the following conditions:
\begin{enumerate}
\item[PN1)]
$F_{0,0}=\C\1$;
\item[PN2)]
$F_{n,m}=0$ if $n \in \frac{1}{2}\Z_{<0}$ or $m \in \frac{1}{2}\Z_{<0}$;
\item[PN3)]
$L(1)F_{1,n}=0$ and $\Ld(1)F_{n,1}=0$ for any $n \in \Z$;
\item[PN4)]
$L(-1)F_{0,n}=0$ and $\Ld(-1)F_{n,0}=0$ for any $n \in \Z$.
\end{enumerate}
Then, $F$ is self-dual and QP-generated.
\end{cor}

\subsection{Parenthesized correlation functions and formal power series}
Let $F$ be a self-dual vertex operator algebra and denote $Y(a,\x)$ by $a(x)$ in this section for $a\in F$.
For $n \in \Z_{>0}$, let $P_n$ be the set of parenthesized products of $n$ elements $1,2,\dots,n$, e.g., 
$((32)1)((47)(65)) \in P_7$
and $Q_n$ be the set of parenthesized products of $n+1$ elements $1,2,\dots,n,\star$ with $\star$ at right most,
e.g., $((32)1)((4(65))\star) \in Q_6$.
%We remark that $P_n$ is a $n$th component of the magma operad, thus the set of multi-linear elements in free magma (see \cite{}).
We can naturally associate a parenthesized $n$ point correlation function
 to each element in $P_n$ and $Q_n$ (see also Introduction \ref{intro_parenthesized} and Section \ref{sec_expansion}).
To give a simple example, an element $(((31)6)(24))(57) \in P_7$ and $a_1,\dots,a_7\in F$
 define the following parenthesized $7$ point correlation function:
\begin{align*}
S_{(((31)6)(24))(57)}(a_1,a_2,a_3,a_4,a_5,a_6,a_7):=
\Bigl(\1, \Biggl[\biggl( \Bigl(a_3(x_3)a_1\Bigr)(x_1)a_6\biggr)(x_6) a_2(x_2)a_4\Biggr](x_4) a_5(x_5)a_7\Bigr).
\end{align*}
For an element in $Q_n$ and $a_1,\dots,a_n \in F$, we consider parenthesized $n+1$ point correlation
for states $a_1,\dots,a_n,\1$ where $\1$ corresponds to $\star$,
e.g., for $(((31)6)(24))(5\star) \in Q_6$,
\begin{align*}
S_{(((31)6)(24))(5\star)}(a_1,a_2,a_3,a_4,a_5,a_6):= 
\Bigl(\1, \Biggl[\biggl( \Bigl(a_3(x_3)a_1\Bigr)(x_1)a_6\biggr)(x_6) a_2(x_2)a_4\Biggr](x_4) a_5(x_5)\1\Bigr).
\end{align*}
%A parenthesized $n$ point correlation function is a formal power series of $2(n-1)$ variables.
In this subsection, we study the space of formal variables for each $A \in P_n$ and $A \in Q_n$.

We start from $A=1(2(\dots (n-1 n))\dots )) \in P_n$. In this case, for $a_1,\dots,a_n \in F$,
the parenthesized correlation function is
$$S_{1(2(\dots (n-1 n))\dots ))}(a_1,\dots,a_n)=(\1,a_1(x_1)a_2(x_2)\dots a_{n-1}(x_{n-1})a_n).
$$
Set 
\begin{align*}
&T_{1(2(\dots (n-1 n))\dots ))}(x_1,\dots,x_{n-1}) \\
&=\C[[x_2/x_1,\x_2/\x_1]]((x_3/x_2,\x_3/\x_2,|x_3/x_2|^\R,\dots,x_{n-1}/x_{n-2},\x_{n-1}/\x_{n-2},|x_{n-1}/x_{n-2}|^\R))\\
& \;\;\;\;\;\;\;\;\;\;\;\;         [x_1^\pm,\x_1^\pm,|x_1|^\R,x_2^\pm,\x_2^\pm,|x_2|^\R].
\end{align*}
For $h,\h,h',\h' \in \R$ and $a \in F_{h,\h}$ and $b\in F_{h',\h'}$,
since
\begin{align*}
(\1,Y(a,\uz)b)&=(Y(S_z a,\uz^{-1})\1,b)\\
&=(\exp(D z^{-1}+\D \z^{-1})S_z a, b)
\end{align*}
and $F_{h,\h}$ and $F_{h',\h'}$ is orthogonal if $(h,\h)\neq (h',\h')$,
we have:
\begin{lem}\label{vacuum_fusion}
For $h,\h,h',\h' \in \R$ and $a \in F_{h,\h}$ and $b\in F_{h',\h'}$,
$(\1,Y(a,\uz)b)=0$ unless $h-h' \in \Z$ and $\h-\h'\in \Z$.
\end{lem}

Then, we have:
\begin{lem}\label{axiom}
For $a_1,\dots,a_{n} \in F$,
$$S_{1(2(\dots (n-1 n))\dots ))}(a_1,\dots,a_N;x_1,\dots,x_{n-1}) \in T_{1(2(\dots (n-1 n))\dots ))}(x_1,\dots,x_{n-1}).$$
\end{lem}
\begin{proof}
We may assume that $a_i \in F_{h_i,\h_i}$ for $h_i,\h_i\in \R$ and $i=1,\dots, n$.
Set
$h=h_1+\dots+h_n$ and $\h=\h_1+\dots+\h_n$ and
$$\sum_{s_1,\s_1,\dots, \s_n,\s_n \in \R}c_{s_1,\s_1,\dots,s_{n-1},\s_{n-1}}z_1^{s_1}\z_1^{\s_1}\dots z_{n-1}^{s_{n-1}}\z_{n-1}^{s_{n-1}}=(\1,Y(a_1,\uz_1)Y(a_2,\uz_2)\dots Y(a_{n-1},\uz_{n-1})a_n).$$
Then, $$c_{s_1,\s_1,\dots,s_{n-1},\s_{n-1}}=(\1, a_1(-s_1-1,-\s_1-1)a_2(-s_2-1,-\s_2-1)\dots a_{n-1}(-s_{n-1}-1,-\s_{n-1}-1)a_n).$$
By (FVO5), there exists $N \in \R$ such
that $F_{h,\h}=0$ for any $h \leq N$ or $\h \leq N$.
Since for any $k=1,2,\dots,n-1$, $a_k(-s_k-1,-\s_k-1)a_{k-1}(-s_{k-1}-1,-\s_{k-1}-1)\dots a_{n-1}(-s_{n-1}-1,-\s_{n-1}-1)a_n$\\
is in $F_{h_k+\dots+h_n+s_k+\dots+s_{n-1},\h_k+\dots+\h_n+\s_k+\dots+\s_{n-1}}$,
by Lemma \ref{vacuum_fusion} and (FVO5), $c_{s_1,\s_1,\dots,s_{n-1},\s_{n-1}}=0$
unless
%\begin{align*}
$0=h+s_1+\dots+s_{n-1}$, $0=\h+\s_1+\dots+\s_{n-1}$ and
$(h-h_1)+s_2+\dots+s_{n-1}-h_1 \in \Z$, $(\h-\h_1)+\s_2+\dots+\s_{n-1}-\h_1 \in \Z$ and
$h_k+\dots+h_n+s_k+\dots+s_{n-1} \geq N$, $\h_k+\dots+\h_n+\s_k+\dots+\s_{n-1} \geq N$ 
for any $k=1,\dots, n-1$.
%\end{align*}
Thus, we have 
\begin{align*}
&(\1,Y(a_1,\ux_1)Y(a_2,\ux_2)\dots Y(a_{n-1},\ux_{n-1})a_n)\\
&=x_1^{-h}\x_1^{-\h}
\sum_{\substack{n,m \in \Z \\ s_3,\s_3,\dots, s_n,\s_n \in \R}}c_{s_1,\s_1,\dots,s_{n-1},\s_{n-1}}
\frac{x_2}{x_1}^{-h+2h_1+n}
\frac{\x_2}{\x_1}^{-\h+2h_1+m}
\dots
\frac{x_{n-1}}{x_{n-2}}^{s_{n-1}}\frac{\x_{n-1}}{\x_{n-2}}^{\s_{n-1}},
%x_1^{s_1}\z_1^{\s_1}\dots z_{n-1}^{s_{n-1}}\z_{n-1}^{s_{n-1}}
%  \C((x_2/x_1,\x_2/\x_1,|x_2/x_1|^\R,\dots,x_{n-1}/x_{n-2},\x_{n-1}/\x_{n-2},|x_{n-1}/x_{n-2}|^\R)).
\end{align*}
where $s_1= -h-s_2-\dots-s_{n-1}$, $\s_1=-\h-\s_2-\dots-\s_{n-1}$, $s_2=-h+2h_1-s_3-\dots-s_{n-1}+n$ and 
$\s_2=-\h+2\h_1-\s_3-\dots-\s_{n-1}+m$.
Thus, the assertion follows.
\begin{comment}
\begin{align*}
&(\1,Y(a_1,x_1)Y(a_2,x_2)\dots Y(a_{n-1},x_{n-1})a_n)\\
&=x_1^{h}\x_1^{\h}
\sum_{s_2,\s_2,\dots, s_n,\s_n \in \R}c_{s_1,\s_1,\dots,s_{n-1},\s_{n-1}}\frac{x_2}{x_1}^{s_2+\dots+s_{n-1}}
\frac{\x_2}{\x_1}^{\s_2+\dots+\s_{n-1}}
\dots
\frac{x_{n-1}}{x_{n-2}}^{s_{n-1}}\frac{\x_{n-1}}{\x_{n-2}}^{\s_{n-1}},
%x_1^{s_1}\z_1^{\s_1}\dots z_{n-1}^{s_{n-1}}\z_{n-1}^{s_{n-1}}
%  \C((x_2/x_1,\x_2/\x_1,|x_2/x_1|^\R,\dots,x_{n-1}/x_{n-2},\x_{n-1}/\x_{n-2},|x_{n-1}/x_{n-2}|^\R)).
\end{align*}
where 
By (??), $c_{s_1,\s_1,\dots,s_{n-1},\s_{n-1}}=0$
unless $h-s_2-\dots-s_{n-1}-2h_1 \in \Z$ and $\h-\s_2-\dots-\s_{n-1} -2\h_1 \in \Z$.
\end{comment}
\end{proof}

We have the embedding $P_n \rightarrow Q_n$
defined by $A \mapsto (A)\star$ for $A\in P_n$
and surjection $d_\star:Q_n \rightarrow P_n$
defined by deleting $\star$ from $A\in Q_n$.
By lemma \ref{one_point_cor}, we have:
\begin{lem}\label{embedding}
For $A\in P_n$ and $a_1,\dots,a_n$,
$$S_A(a_1,\dots,a_n)=S_{(A)\star}(a_1,\dots,a_n).
$$
\end{lem}
By Proposition \ref{translation}, $Y(a,\uz)\1=\exp(Dz+\D \z)a$ and $[D,Y(a,\uz)]=d/dz Y(a,\uz)$,
any parenthesized $n$ point correlation function for $A\in Q_n$
 is an expansion of a parenthesized correlation function for $d_\star(A)\in P_n$.

For example, let us consider the case of $A=1(2(3(4\star))) \in Q_4$.
In this case,
$$S_{1(2(3(4\star)))}(a_1,a_2,a_3,a_4;z_1,z_2,z_3,z_4)
=(\1,a_1(z_1)a_2(z_2)a_3(z_3)a_4(z_4)\1).
$$
Since
\begin{align*}
(\1&,a_1(z_1)a_2(z_2)a_3(z_3)a_4(z_4)\1)\\
&=(\1,a_1(z_1)a_2(z_2)a_3(z_3)\exp(Dz_4+\D \z_4)a_4)\\
&=\exp(-z_4(d/dz_1+d/dz_2+d/dz_3))\exp(-\z_4(d/d\z_1+d/d\z_2+d/d\z_3)
(\1,a_1(z_1)a_2(z_2)a_3(z_3)a_4)\\
&=\exp(-z_4(d/dz_1+d/dz_2+d/dz_3))\exp(-\z_4(d/d\z_1+d/d\z_2+d/d\z_3)S_{1(2(34))}(a_1,a_2,a_3,a_4;z_1,z_2,z_3),
\end{align*}
by Lemma \ref{expansion_translation}, we have:
\begin{lem}\label{transform_qp}
For $a_1,a_2,a_3,a_4$,
$$S_{1(2(3(4\star)))}(a_1,a_2,a_3,a_4;z_1,z_2,z_3,z_4)
=T_{1(2(34))}^{1(2(3(4\star)))}S_{{1(2(34))}}(a_1,a_2,a_3,a_4;z_{14},z_{24},z_{34}).$$
\end{lem}

Hereafter, we concentrate on the case of $n=4$ and $A\in P_4$.
Similarly to the proof of Lemma \ref{axiom},
we can show that $S_A(a_1,a_2,a_3,a_4) \in T_A$,
where $T_A$ is the space of formal power series defined in Section \ref{sec_expansion}.

For example, for $(12)(34)$, we have:
\begin{lem}\label{space_channel}
For any $a_1,a_2,a_3,a_4 \in F$,
$(\1,Y(Y(a_1,\ux_1)a_2,\ux_2)Y(a_{3},\ux_{3})a_4)\in T_{(12)(34)}(x_2,x_1,x_3)$.
\end{lem}
\begin{proof}
We may assume that $a_i \in F_{h_i,\h_i}$ for $h_i,\h_i\in \R$ and $i=1,\dots, 4$.
Set $h=h_1+\dots+h_4$ and $\h=\h_1+\dots+\h_4$ and
$$\sum_{s_1,\s_1,\dots, \s_3,\s_3 \in \R}c_{s_1,\s_1,s_2,\s_2,s_{3},\s_{3}}
x_1^{s_1}\x_1^{\s_1}x_{2}^{s_{2}}\x_{2}^{s_{2}} x_{3}^{s_{3}}\x_{3}^{s_{3}}=
(\1,Y(Y(a_1,\ux_1)a_2,\ux_2)Y(a_{3},\ux_{3})a_4).$$
Then, $$c_{s_1,\s_1,\dots,s_{3},\s_{3}}=
(\1, \Bigl(a_1(-s_1-1,-\s_1-1)a_2\Bigr)(-s_2-1,-\s_2-1)a_{3}(-s_{3}-1,-\s_{3}-1)a_4).$$
By (FVO5), there exists $N \in \R$ such
that $F_{h,\h}=0$ for any $h \leq N$ or $\h \leq N$.
Thus, similarly to the proof of Lemma \ref{axiom}, $c_{s_1,\s_1,\dots,s_{3},\s_{3}}=0$
unless $0=h+(s_1+s_2+s_{3})$,
$h_1+h_2+s_1-h_3-h_4-s_3 \in \Z$
and
$h_3+h_4+s_3, h_1+h_2+s_1 \geq N$.

Thus, there exists $N' \in \R$ such that
\begin{align*}
&(\1,Y(Y(a_1,\ux_1)a_2,\ux_2)Y(a_{3},\ux_{3})a_4)\\
&=x_2^{-h}\x_2^{-\h}\Bigl(\frac{x_1}{x_2}\Bigr)^{-h_1-h_2+h_3+h_4}\Bigl(\frac{\x_1}{\x_2}\Bigr)^{-\h_1-\h_2+\h_3+\h_4}
\sum_{\substack{s_1',\s_1', s_3,\s_3 \in \R}}c_{s_1,\s_1,\dots,s_{3},\s_{3}}
\Bigl(\frac{x_1}{x_2}\Bigr)^{s_1'}\Bigl(\frac{\x_1}{\x_2}\Bigr)^{\s_1'}
\Bigl(\frac{x_3}{x_2}\Bigr)^{s_3}\Bigl(\frac{\x_3}{\x_2}\Bigr)^{\s_3}
%x_1^{s_1}\z_1^{\s_1}\dots z_{n-1}^{s_{n-1}}\z_{n-1}^{s_{n-1}}
%  \C((x_2/x_1,\x_2/\x_1,|x_2/x_1|^\R,\dots,x_{n-1}/x_{n-2},\x_{n-1}/\x_{n-2},|x_{n-1}/x_{n-2}|^\R)).
\end{align*}
where $s_2= -h-s_1-s_3$, $\s_2=-\h-\s_1-\s_3$, 
$s_1'=h_1+h_2-h_3-h_4+s_1$ and $\s_1'=\h_1+\h_2-\h_3-\h_4+\s_1$
and the sum runs through $s_1',s_3,\s_1',\s_3 \geq N'$
and $s_1'-s_3 \in \Z$, $\s_1'-\s_3 \in \Z$.
Thus, the assertion holds.
\end{proof}

\subsection{Consistency of four point functions}
Let $F$ be a self-dual full vertex operator algebra.
In this section, we will prove the main result of this paper, 
the consistency of four point correlation functions.
The following lemma follows from Lemma \ref{global_covariance}:
\begin{lem}
\label{quasi_differential}
For $h_i,\h_i \in \R$ and $a_i \in F_{h_i,\h_i}$ and $b_i \in QF_{h_i,\h_i}$,
\begin{enumerate}
\item
\begin{align*}
%(\sum_{i=1}^n d/dz_i)(\1,Y(a_1,z_1)\dots Y(a_n,z_n),\1)&=0,\\
%(\sum_{i=1}^n d/d\z_i)(\1,Y(a_1,z_1)\dots Y(a_n,z_n)\1)&=0,\\
(\sum_{i=1}^n z_id/dz_i+h_i)(\1,Y(a_1,\uz_1)\dots Y(a_{n-1},\uz_{n-1})a_n)&=0,\\
(\sum_{i=1}^n \z_id/d\z_i+\h_i)(\1,Y(a_1,\uz_1)\dots Y(a_{n-1},\uz_{n-1})a_n)&=0;
\end{align*}
\item
\begin{align*}
(\sum_{i=1}^n z_i^2 d/dz_i+2h_iz_i)(\1,Y(b_1,\uz_1)\dots Y(b_{n-1},\uz_{n-1})b_n)&=0,\\
(\sum_{i=1}^n \z_i^2 d/d\z_i+2\h_i\z_i)(\1,Y(b_1,\uz_1)\dots Y(b_{n-1},\uz_{n-1})b_n)&=0.
\end{align*}
\end{enumerate}
\end{lem}
For $h_i,\h_i \in \R$ with $h_i-\h_i \in \Z$ ($i=1,2,3,4$),
set
\begin{align*}
Q(\underline{h},z)&=(z_1-z_2)^{-2h_2}(z_3-z_4)^{h_1-h_2-h_3-h_4}
(z_1-z_4)^{h_2-h_1+h_3-h_4}(z_1-z_3)^{h_2-h_1-h_3+h_4}\\
&(\z_1-\z_2)^{-2\h_2}(\z_3-\z_4)^{\h_1-\h_2-\h_3-\h_4}
(\z_1-\z_4)^{\h_2-\h_1+\h_3-\h_4}(\z_1-\z_3)^{\h_2-\h_1-\h_3+\h_4}
 \in \Cor_4^f.
%\frac{1}{(z_{1}-z_{2})^{h_1+h_2}(z_3-z_4)^{h_3+h_4}}(\frac{z_2-z_4}{z_1-z_4})^{h_1-h_2}
%(\frac{z_1-z_4}{z_1-z_3})^{h_3-h_4}\\
%&\cdot \frac{1}{(\z_{1}-\z_{2})^{\h_1+\h_2}(\z_3-\z_4)^{\h_3+\h_4}}(\frac{\z_2-\z_4}{\z_1-\z_4})^{\h_1-\h_2}
%(\frac{\z_1-\z_4}{\z_1-\z_3})^{\h_3-\h_4}
\end{align*}
Then, by an easy computation, we have:
\begin{lem}\label{free_differential}
For $h_i,\h_i \in \R$ with $h_i-\h_i \in \Z$ ($i=1,2,3,4$),
as a function on $X_4$, $Q(\underline{h},z)$ satisfies
\begin{align*}
(\sum_{i=1}^3 z_id/dz_i+h_i)Q(\underline{h},z)=0.\\
(\sum_{i=1}^3 \z_id/d\z_i+\h_i)Q(\underline{h},z)=0.\\
(\sum_{i=1}^4 z_i^2 d/dz_i+2z_ih_i)Q(\underline{h},z)=0.\\
(\sum_{i=1}^4 \z_i^2 d/d\z_i+2\z_i\h_i)Q(\underline{h},z)=0.
\end{align*}
\end{lem}
\begin{rem}
We can choose another $Q(\underline{h},z) \in \Cor_4$ which satisfies the above differential equations.
In fact, we can multiply $Q(\underline{h},z)$ by $\C[\xi^\pm,\bar{\xi}^\pm,|\xi|^\R]$.
We also remark that
$$\Pi_{1 \leq i <j \leq 4} (z_i-z_j)^{-h_i-h_j+h/3}(\z_i-\z_j)^{-\h_i-\h_j+\h/3},
$$
is commonly used in physics,
where $h=h_1+h_2+h_3+h_4$ and $\h=\h_1+\h_2+\h_3+\h_4$.
However, since $(h-\h)/3$ is not always integer, this function is not in $\Cor_4^f$.
\end{rem}

\begin{prop}
\label{qp_1234}
For $h_i,\h_i \in \R$ and $a_i \in QF_{h_i,\h_i}$ ($i=1,2,3,4$),
the following conditions hold:
\begin{enumerate}
\item
$$(\1,Y(a_1,\uz_{12})a_2)= (a_1,a_2)(-1)^{h_1-\h_1}z_{12}^{-2h_1}\z_{12}^{-2\h_1}.$$
In particular, it is zero unless $h_1=h_2$ and $\h_1=\h_2$.
\item
There exists $C_{a_1,a_2,a_3} \in \C$ such that
$$(\1,Y(a_1,\uz_{13})Y(a_2,\uz_{23})a_3)=
C_{a_1,a_2,a_3}\Pi_{1\leq i <j \leq 3} (z_i-z_j)^{h_1+h_2+h_3-2h_i-2h_j}(\z_i-\z_j)^{\h_1+\h_2+\h_3-2\h_i-2\h_j}
|_{|z_{13}|>|z_{23}|}.$$
\item
There exists $f \in \F$ such that
$$(\1,Y(a_1,\uz_{14})Y(a_2,\uz_{24})Y(a_3,\uz_{34})a_4)=
\Bigl( Q(\underline{h},z) f\circ \xi \Bigr) |_{|z_{14}|>|z_{24}|>|z_{34}|}.$$
In particular, this series is absolutely convergent to the function in $\Cor_4$.
%
% 4-point also ??
\end{enumerate}
\end{prop}
\begin{proof}%[proof of Proposition \ref{qp_1234}]
%
% with a nice factor, prove SL_2 C invariance. 
% then the proof is follows from an analytic lemma.
%
Since $(\1,Y(a_1,\uz_{12})a_2)=(Y(S_{z_{12}}a_1,\uz_{12}^{-1})\1, a_2)
=(-1)^{h_1-\h_1}z_{12}^{-2h_1}\z_{12}^{-2\h_1}(\exp(L(-1)z_{12}^{-1}+\Ld(-1)\z_{12}^{-1})a_1,a_2)
=(-1)^{h_1-\h_1}z_{12}^{-2h_1}\z_{12}^{-2\h_1}(a_1,\exp(L(1)z_{12}^{-1}+\Ld(1)\z_{12}^{-1})a_2)=
(-1)^{h_1-\h_1}z_{12}^{-2h_1}\z_{12}^{-2\h_1}(a_1,a_2)$,
(1) holds.

By Lemma \ref{axiom},
the product of the formal power series $\Pi_{1\leq i <j \leq 3} (z_i-z_j)^{-h_1-h_2-h_3+2h_i+2h_j}(\z_i-\z_j)^{-\h_1-\h_2-\h_3+2\h_i+2\h_j}|_{|z_{13}|>|z_{23}|}$ and $(\1,Y(a_1,\uz_{13})Y(a_2,\uz_{23})a_3)$
is in $$\C[[z_{23}/z_{13},\z_{23}/\z_{13}]]
[z_{13}^\pm,z_{23}^\pm,\z_{13}^\pm,\z_{23}^\pm,|z_{13}|^\R,|z_{23}|^\R]
$$
and is in the kernel of the formal differentials $D_0'',\D_0'',D_1'',\D_1''$ by Lemma \ref{quasi_differential} and
the similar result to Lemma \ref{free_differential}. Hence, it is constant $C_{a_1,a_2,a_3} \in \C$ by Lemma \ref{differential_3},
which implies that (2) holds.

Similarly,
set
$$\phi = Q(\underline{h},z)^{-1}|_{|z_{14}|>|z_{24}|>|z_{34}|}(\1,Y(a_1,\uz_{14})Y(a_2,\uz_{24})Y(a_3,\uz_{34})a_4),$$
which is an element of  $T(z_{14},z_{24},z_{34})$ and in the kernel of the formal differential $D_0,\D_0,D_1,\D_1$.
Thus, by Lemma \ref{differential_inverse}, $v_{1(2(34))}\phi \in \C((p,\p,|p|^\R))$ determines $\phi$.
Recall that $v_{1(2(34))}:S(z_{14},z_{24},z_{34}) \rightarrow \C((p,\p,|p|^\R))$
is defined by the evaluation $(z_{14},z_{24},z_{34}) \mapsto (\infty,1,p)$.
Since 
\begin{align*}
&(\1,Y(a_1,\uz_{14})Y(a_2,\uz_{24})Y(a_3,\uz_{34})a_4)\\
&=(-1)^{h_1-\h_1}z_{14}^{-2h_1}\z_{14}^{-2\h_1}(\exp(L(-1)z_{14}+\Ld(-1)\z_{14})a_1,Y(a_2,\uz_{24})Y(a_3,\uz_{34})a_4)
\end{align*}
and 
\begin{align*}
Q(\underline{h},z)=&z_{14}^{-2h_1}z_{34}^{h_1-h_2-h_3-h_4}
(1-z_{24}/z_{14})^{-2h_2}(1-z_{34}/z_{14})^{h_2-h_1-h_3+h_4}\\
&\z_{14}^{-2\h_1}\z_{34}^{\h_1-\h_2-\h_3-\h_4}
(1-\z_{24}/\z_{14})^{-2\h_2}(1-\z_{34}/\z_{14})^{\h_2-\h_1-\h_3+\h_4}.
%\z_{14}^{-2\h_1} 
%\frac{1}{(1-z_{24}/z_{14})^{h_1+h_2}z_{34}^{h_3+h_4}}
%\frac{{z_{24}}^{h_1-h_2}}{(1-z_{34}/z_{14})^{h_3-h_4}}\\
%&\cdot \frac{1}{(1-\z_{24}/\z_{14})^{\h_1+\h_2}\z_{34}^{\h_3+\h_4}}
%\frac{{\z_{24}}^{\h_1-\h_2}}{(1-\z_{34}/\z_{14})^{\h_3-\h_4}},
\end{align*}
%in $$z_{14}^{2h_1}z_{24}^{h_2-h_1+h/3}\z_{24}^{\h_2-\h_1+\h/3}
%z_{34}^{h_3+h_4-h/3}\z_{34}^{\h_3+\h_4-\h/3}\C[[z_{24}/z_{14},z_{34}/z_{24}]],$$
the formal limit of $\phi$ as $\uz_{14} \rightarrow \infty$
is
$$(-1)^{h_1-\h_1}z_{34}^{-h_1+h_2+h_3+h_4}\z_{34}^{-\h_1+\h_2+\h_3+\h_4}(a_1,Y(a_2,\uz_{24})Y(a_3,\uz_{34})a_4)
$$
and 
$$v_{1(2(34))}(\phi)
=(-1)^{h_1-\h_1}p^{-h_1+h_2+h_3+h_4}\p^{-\h_1+\h_2+\h_3+\h_4}(a_1,Y(a_2,1)Y(a_3,\underline{p})a_4).
$$
By Lemma \ref{borcherds}, there exists $f \in \F$ such that
$v_{1(2(34))}(\phi)=j(p,f)$.
By Corollary \ref{convergence_1},
we have $\phi = e_{1(2(34))}(f)$.
\end{proof}
\begin{comment}
\begin{prop}\label{S_3}
??????
Let $h_i,\h_i \in \R$ and $a_i \in QF_{h_i,\h_i}$ ($i=1,2,3$)
and $C_{ijk}$ be the constant in Proposition \ref{qp_1234},
where $ijk$ is a permutation of $\{1,2,3\}$.
Then,
\begin{align*}
C_{123}=C_{132}=C_{213}
\end{align*}
\end{prop}
\end{comment}
For $h_i,\h_i \in \R$ and $a_i \in QF_{h_i,\h_i}$ ($i=1,2,3,4$),
by the above proof,
\begin{align*}
\lim_{\uz_{14} \to \infty} &Q(h,z)^{-1}(\1,Y(a_1,\uz_{14})Y(a_2,\uz_{24})Y(a_3,\uz_{34})a_4)\\
&=(-1)^{h_1-\h_1}z_{34}^{-h_1+h_2+h_3+h_4}\z_{34}^{-\h_1+\h_2+\h_3+\h_4}(a_1,Y(a_2,\uz_{24})Y(a_3,\uz_{34})a_4)\\
\lim_{\uz_{14} \to \infty} &Q(h,z)^{-1}(\1,Y(a_1,\uz_{14})Y(a_3,\uz_{34})Y(a_2,\uz_{24})a_4)\\
&=(-1)^{h_1-\h_1}z_{34}^{-h_1+h_2+h_3+h_4}\z_{34}^{-\h_1+\h_2+\h_3+\h_4}(a_1,Y(a_3,\uz_{34})Y(a_2,\uz_{24})a_4).
\end{align*}
Thus, by Lemma \ref{borcherds},
we have:
\begin{lem}
\label{generalized_limit2}
For $h_i,\h_i \in \R$ and $a_i \in QF_{h_i,\h_i}$ ($i=1,2,3,4$),
there exists $f \in \F$ such that
\begin{align*}
(\1,Y(a_1,\uz_{14})Y(a_2,\uz_{24})Y(a_3,\uz_{34})a_4)=e_{1(2(34))}(Q(h,z)f\circ \xi)\\
(\1,Y(a_1,\uz_{14})Y(a_3,\uz_{34})Y(a_2,\uz_{24})a_4)=e_{1(3(24))}(Q(h,z)f\circ \xi).
\end{align*}
\end{lem}

\begin{prop}
\label{convergence_standard}
Let $F$ be a self-dual QP-generated full vertex operator algebra.
Then, for any $a_i \in F_i$, there exists $\phi \in \Cor_4$ such that
$$(\1,Y(a_{\sigma1},\uz_{\sigma1})\dots Y(a_{\si 4},\uz_{\si 4})\1)=e_{\si 1(2(3(4\star)))}(\phi),$$
for any $\sigma \in S_4$.
In particular, this series is absolutely convergent to a function in $\Cor_4$.
\end{prop}

\begin{proof}
% By solving the solution of ODE.
First, we assume that $a_i \in QF_{h_i,\h_i}$ for $i=1,2,3,4$.
Set $Q=Q(\underline{h},z)$
and
$h=h_1+h_2+h_3+h_4$ and $\h=\h_1+\h_2+\h_3+\h_4$.
Then, by Proposition \ref{qp_1234}, there exists $f \in \F$
such that $(\1,Y(a_1,\uz_{14})Y(a_2,\uz_{24})Y(a_3,\uz_{34})a_4)=e_{1(2(34))}(Qf)$.
Then, by Lemma \ref{expansion_translation} and Lemma \ref{transform_qp}, we have
\begin{align*}
(\1&,Y(a_1,\uz_{1})Y(a_2,\uz_2)Y(a_3,\uz_3) Y(a_{4},\uz_{ 4})\1)\\
&=T_{{1(2(34))}}^{{1(2(3(4\star)))}}(\1,Y(a_1,\uz_{14})Y(a_2,\uz_{24})Y(a_3,\uz_{34})a_4)\\
&=T_{{1(2(34))}}^{{1(2(3(4\star)))}} e_{1(2(34))}(Qf)\\
&= e_{1(2(3(4\star)))}(Qf).
\end{align*}

Let $f_1,f_2 \in \F$ satisfy
\begin{align*}
e_{1(2(43))}(Q f_1)&=(\1,Y(a_1,\uz_{13})Y(a_2,\uz_{23})Y(a_4,\uz_{43})a_3)\\
e_{4(3(2(10))}(Q f_2)&=(\1,Y(a_4,\uz_{4})Y(a_3,\uz_3)Y(a_2,\uz_2) Y(a_{1},\uz_{1})\1).
%e_{1(3(2(4)))}(Q f_3)&=(\1,Y(a_1,z_{14})Y(a_3,z_{34})Y(a_2,z_{24})a_4).
\end{align*}
We will show that $f$ and $f_1,f_2$ are the same function.
Then, since $(34), (14)(23), (23) \in S_4$ is a generator of $S_4$,
by Lemma \ref{generalized_limit2}, the assertion holds for quasi-primary vectors.
By Proposition \ref{translation} and Lemma \ref{formal_skew},
\begin{align*}
e_{1(2(43))}(Q f_1)&=(\1,Y(a_1,\uz_{13})Y(a_2,\uz_{23})Y(a_4,\uz_{43})a_3)\\
&=(\1,Y(a_1,\uz_{13})Y(a_2,\uz_{23})\exp(z_{43}D+\z_{43}\D) Y(a_3,-\uz_{43})a_4)\\
&=\exp(-z_{43}(d/dz_{13}+d/dz_{23})-\z_{43}(d/d\z_{13}+d/d\z_{23}))(\1,Y(a_1,\uz_{13})
Y(a_2,\uz_{23})Y(a_3,-\uz_{43})a_4)\\
&=T_{1(2(34))}^{1(2(43))} e_{1(2(34))}(Q f),
\end{align*}
which implies that $f_1=f$.

Since
\begin{align*}
e_{1(2(3(4\star))}(Q f )&=(\1,Y(a_1,\uz_{1})Y(a_2,\uz_2)Y(a_3,\uz_3) Y(a_{4},\uz_{ 4})\1)\\
&=(-1)^{h-\h} \Pi_{1\leq i\leq 4} z_i^{-2h_i}\z_i^{-2\h_i}
(\1,Y(a_4,\uz_{4}^{-1})Y(a_3,\uz_3^{-1})Y(a_2,\uz_2^{-1}) Y(a_{1},\uz_{1}^{-1})\1),
\end{align*}
by Lemma \ref{formal_dual},
\begin{align*}
e_{4(3(2(1\star))}(Qf_2)
&=(\1,Y(a_4,\uz_{4})Y(a_3,\uz_3)Y(a_2,\uz_2) Y(a_{1},\uz_{1})\1)\\
&= I_d\Bigl( (-1)^{h-\h} \Pi_{1\leq i\leq 4} z_i^{2h_i}\z_i^{2\h_i} 
(\1,Y(a_1,\uz_{1})Y(a_2,\uz_2)Y(a_3,\uz_3) Y(a_{4},\uz_{ 4})\1)\Bigr)\\
&= I_d\Bigl( (-1)^{h-\h} \Pi_{1\leq i\leq 4} z_i^{2h_i}\z_i^{2\h_i}e_{1(2(3(40))}(Q f)\Bigr) \\
&=e_{4(3(2(1\star))}(Q f),
\end{align*}
which implies $f_2=f$.

We now turn to the case that $a_i \in F_{h_i,\h_i}$.
Since $F$ is QP-generated,
we may assume that $a_i=D^{n_i} \D^{\n_i}b_i$ for some $n_i, \n_i \in \Z_{\geq 0}$ and
$b_i \in QF_{h_i-n_i,\h_i-\n_i}$.
Since
\begin{align*}
(\1,Y(a_1,\uz_{1})&Y(a_2,\uz_2)Y(a_3,\uz_3) Y(a_{4},\uz_{ 4})\1)\\
&=(\1,Y(D^{n_1} \D^{\n_1}b_1,\uz_{1})Y(D^{n_2} \D^{\n_2}b_2,\uz_2)D^{n_3} 
\D^{\n_3}Y(b_3,\uz_3) Y(D^{n_4} \D^{\n_4}b_{4},\uz_{ 4})\1)\\
&=\Pi_{1\leq i\leq 4} d/dz_i^{n_i} d/d\z_i^{\n_i}
(\1,Y(b_1,\uz_{1})Y(b_2,\uz_2)Y(b_3,\uz_3) Y(b_{4},\uz_{ 4})\1),
\end{align*}
by Proposition \ref{qp_1234},
there exists $\phi \in \Cor_4$ such that
$(\1,Y(a_1,\uz_{1})Y(a_2,\uz_2)Y(a_3,\uz_3) Y(a_{4},\uz_{ 4})\1)=e_{1(2(3(4\star)))}(\phi)$.
The $S_4$-invariance follows from the $S_4$-invariance for quasi-primary vectors.
\end{proof}

Hence,
we can define the linear map $S:F^{\otimes 4} \rightarrow \Cor_4$
by $e_{1(2(3(4\star)))}(S(a_1,\dots,a_4))=(\1,Y(a_1,\uz_{1})\dots Y(a_{4},\uz_{ 4})\1)$.
We recall that the symmetric group $S_4$ acts on $\Cor_4$ (see Section \ref{sec_S4}).
Then, the $S_4$-symmetry of the correlation function follows:
\begin{thm}\label{S4_symmetry}
For a self-dual QP-generated full vertex operator algebra $F$
and $a_1,a_2,a_3,a_4 \in F$ and $\si \in S_4$,
$$\si \cdot S(a_1,a_2,a_3,a_4)=S(a_{\si^{-1} 1},a_{\si^{-1} 2},a_{\si^{-1} 3},a_{\si^{-1} 4}).
$$
\end{thm}
\begin{proof}
By Proposition \ref{expansion} and Proposition \ref{convergence_standard},
\begin{align*}
e_{1(2(3(4\star)))}(\si \cdot S(a_1,a_2,a_3,a_4))&=
T_\si e_{\si^{-1} 1(2(3(4\star)))} S(a_1,a_2,a_3,a_4)\\
&=T_\si (\1,Y(a_{\si^{-1} 1},\uz_{\si^{-1} 1})Y(a_{\si^{-1} 2},\uz_{\si^{-1} 2})
Y(a_{\si^{-1} 3},\uz_{\si^{-1} 3})Y(a_{\si^{-1} 4},\uz_{\si^{-1} 4})\1)\\
&=(\1,Y(a_{\si^{-1} 1},\uz_{1})Y(a_{\si^{-1} 2},\uz_{2})
Y(a_{\si^{-1} 3},\uz_{3})Y(a_{\si^{-1} 4},\uz_{4})\1)\\
&=e_{1(2(3(4\star)))}S(a_{\si^{-1} 1},a_{\si^{-1} 2},a_{\si^{-1} 3},a_{\si^{-1} 4}).
\end{align*}
\end{proof}

Now, the result for the consistency of four point correlation functions in conformal field
theory on $\C P^1$ can be stated as follows:
\begin{thm}\label{consistency}
Let $F$ be a self-dual QP-generated full vertex operator algebra.
Then, %for any $a_i \in F_i$, there exists $\phi \in \Cor_4$ such that
$S_A=e_A \circ S$
for any $A \in Q_4$.
\end{thm}

\begin{lem}
\label{vertex_channel}
For $a_i \in QF_{h_i,\h_i}$ ($i=1,2,3,4$),
\begin{align*}
S_{(12)(34)}(a_1,a_2,a_3,a_4)&=(-1)^{h_2-\h_2} z_{12}^{-2h_1}\z_{12}^{-2\h_1}
z_{24}^{-2h_2+2h_1}\z_{24}^{-2\h_2+2\h_1}(1-z_{34}/z_{24})^{-2h_3}(1-\z_{34}/\z_{24})^{-2\h_3}\\
&\Bigl(a_2,Y\Bigl(a_1,-\uz_{24}(1+\uz_{24}/\uz_{12})\Bigr)Y(a_3,\frac{\uz_{34}}{1-\uz_{34}/\uz_{24}})a_4\Bigr)
\end{align*}
\end{lem}
\begin{proof}
By the invariance,
\begin{align*}
(\1,Y(Y(a_1,\uz_{12})a_2,\uz_{24})Y(a_3,\uz_{34})a_4)
&=(Y(S_{z_{24}}Y(a_1,\uz_{12})a_2,\uz_{24}^{-1})\1, Y(a_3,\uz_{34})a_4)\\
&=(\exp(z_{24}^{-1}D+\z_{24}^{-1}\D)S_{z_{24}}Y(a_1,\uz_{12})a_2, Y(a_3,\uz_{34})a_4)\\
&=(a_2, Y(S_{z_{12}}a_1,\uz_{12}^{-1})
S_{z_{24}}^{\dagger}
(\exp(z_{24}^{-1}\Ld(1)+\z_{24}^{-1}\Ld(1))
Y(a_3,\uz_{34})a_4),
\end{align*}
where $S_{z_{24}}^{\dagger}=(-z_{24})^{L(0)}(-\z_{24})^{\Ld(0)}\exp(L(-1)z_{24}+\Ld(-1)\z_{24})$.
By \cite{FHL},
\begin{align*}
z^{L(0)}Y(a,\uz_0)z^{-L(0)}&=Y(z^{L(0)}a,\uz \uz_0)\\
\exp(zL(1))Y(a,\uz_0)\exp(-zL(1))&=Y(\exp(z(1-zz_0)L(1))(1-zz_0)^{-2L(0)}a, \uz_0/(1-\uz\uz_0)).
\end{align*}
Thus, the assertion holds.
\end{proof}

\begin{proof}[proof of Theorem \ref{consistency}]
We only prove the Theorem for $A\in P_4$.
The general result can be obtained by using Proposition \ref{translation}
and the similar argument as Lemma \ref{transform_qp}.
For the case of $A=1((23)4), ((12)3)4$ and $(1(23))4$,
the assertion follows from Lemma \ref{ass_1} and Lemma \ref{ass_2}.
We will prove the theorem for $A=(12)(34)$.
Similarly to the proof of Theorem \ref{convergence_standard},
we may assume that $a_i \in QF_{h_i,\h_i}$ for $i=1,2,3,4$.
Set $Q=Q(\underline{h},z)$ and let $f \in \F$ satisfy
$S(a_1,a_2,a_3,a_4)=Qf(\xi)$.
Then, by Lemma \ref{global_covariance} and Lemma \ref{space_channel},
\begin{align*}
e_{(12)(34)}(Q^{-1})(\1,Y(Y(a_1,\uz_{12})a_2,\uz_{24})Y(a_3,\uz_{34})a_4)
\end{align*}
is in $S'(z_{24},z_{12},z_{34})$.
Thus, by Lemma \ref{differential_inverse_2}, 
$e_{(12)(34)}(Q^{-1})(\1,Y(Y(a_1,\uz_{12})a_2,\uz_{24})Y(a_3,\uz_{34})a_4)$
is uniquely determined by
the limit of $(z_{24},z_{12},z_{34}) \rightarrow (1,q,q)$.
By Lemma \ref{vertex_channel}, the limit is equal to
\begin{align*}
(-1)^{h_2-\h_2}q^{-3h_1+3h_2+h_3+h_4}(1+q)^{h_1-h_2-h_3+h_4}(1-q)^{-2h_3}
\bar{q}^{-3\h_1+3\h_2+\h_3+\h_4}(1+\bar{q})^{\h_1-\h_2-\h_3+\h_4}(1-\bar{q})^{-2\h_3}\\
\cdot (a_2,Y(a_1,-(1+\underline{q}^{-1}))Y(a_3,\underline{q}/(1-\underline{q}))a_4). \tag{q-series} \label{eq_q}
\end{align*}
Since
the formal limit (we omit the anti-holomorphic part)
\begin{align*}
&\lim_{(z_2,z_1,z_3,z_4)\rightarrow (\infty, -(1+q^{-1}),q/(1-q), 0)}
z_2^{-2h_2}\z_2^{-2\h_2}Q(z)^{-1}\\
&=\frac{q}{1-q}^{-h_1+h_2+h_3+h_4}(-q^{-1}(1+q))^{h_1-h_2-h_3+h_4}{(-q(1-q))}^{-h_1+h_2-h_3+h_4}\\
&=(-1)^{2h_1-2h_2}q^{-3h_1+3h_2+h_3+h_4}(1+q)^{h_1-h_2-h_3+h_4}(1-q)^{-2h_3}
\end{align*}
and $2h_1-2h_2-2\h_1-2\h_2 \in 2\Z$,
we can verify that the series \ref{eq_q} is equal to the formal limit of 
$$Q(z)^{-1}|_{{|z_2|>|z_1|>|z_3|>|z_4|}} (\1,Y(a_2,\uz_2)Y(a_1,\uz_{1})Y(a_3,\uz_3) Y(a_{4},\uz_{ 4})\1)$$
as $(z_2,z_1,z_3,z_4)\rightarrow (\infty, -(1+q^{-1}),q/(1-q), 0)$.
Since the limit of $\xi(z_1,z_2,z_3,z_4)$ as $(z_2,z_1,z_3,z_4)\rightarrow (\infty, -(1+q^{-1}),q/(1-q), 0)$
is $q^2$ and 
$$
Q(z)^{-1}|_{|z_2|>|z_1|>|z_3|>|z_4|} (\1,Y(a_2,\uz_2)Y(a_1,\uz_{1})Y(a_3,\uz_3) Y(a_{4},\uz_{ 4})\1)
=e_{2(1(3(4\star)))}(f(\xi)),
$$
\begin{align*}
v_{(12)(34)}\Bigl(e_{(12)(34)}(Q^{-1})(\1,Y(Y(a_1,z_{12})a_2,z_{24})Y(a_3,z_{34})a_4)\Bigr)
=\lim_{p \to q^2} j(p,f).
\end{align*}
Hence, $(\1,Y(Y(a_1,z_{12})a_2,z_{24})Y(a_3,z_{34})a_4)=e_{(12)(34)}(Qf(\xi))$.
\end{proof}

%\subsection{Branch and Reduction}
%????
% depend only on z, then holomorphic

% single branch field = \sum_{n,m\geq 0} a_n,m z^n \z^m
% half single branch field = (z\z)^{1/2} \sum + \sum
%A state $a \in V$ is said to be a half-branch if
%$a(n,m)=0$ for any $n, m \in \R$ such that $n+m, n-m \neq \Z$.

%vertex algebra??
%module??
% Is there any consequence from single branch condition?
%\subsection{Commutation Formula for Half-Branch states}

% VOA vs vertex operator algebra
% vertex algebra automorphism vs full vertex algebra automorphism

\section{Construction}\label{sec_construction}
In this section, we generalize the Goddard's axiom for a full vertex algebra (Section \ref{sec_locality} and Section \ref{sec_locality2}).
We also show that if a vertex operator $Y(-,\uz)$ satisfies the bootstrap equation
then it gives a full vertex algebra.

A full prevertex algebra is an $\R^2$-graded $\C$-vector space
$F=\bigoplus_{h,\h \in \R^2} F_{h,\h}$ equipped with a linear map 
$$Y(-,\uz):F \rightarrow \End (F)[[z^\pm,\z^\pm,|z|^\R]],\; a\mapsto Y(a,\uz)=\sum_{r,s \in \R}a(r,s)z^{-r-1}\z^{-s-1}$$
and an element $\1 \in F_{0,0}$ %and a $\R^2$-graded subspace $F^\vee=\bigoplus_{r,s} F_{r,s}^\vee$ of the dual
such that
\begin{enumerate}
\item[PV1)]
For any $a,b \in F$, there
exists $N\in \R$ such that $a(r,s)b=0$ for any $r \leq N$ or $s \leq N$;
%$F_{h,\h}=0$ for any $h \leq N$ or $\h \leq N$.
\item[PV2)]
$F_{h,\h}=0$ unless $h-\h \in \Z$;
\item[PV3)]
For any $a \in F$, $Y(a,\uz)\1 \in F[[z,\z]]$ and $\lim_{z \to 0}Y(a,z)\1 = a(-1,-1)\1=a$;
\item[PV4)]
$Y(\1,\uz)=\mathrm{id} \in \End F$;
\item[PV5)]
$F_{h,\h}(r,s)F_{h',\h'} \subset F_{h+h'-r-1,\h+\h'-s-1}$ for any $r,s,h,\h,h',\h' \in \R$.
\end{enumerate}

A full prevertex algebra $(F,Y,\1)$ is said to be translation covariant
if there exist linear maps $D,\D \in \End \;F$ such that
\begin{enumerate}
\item[T1)]
$D\1=\D\1=0$;
\item[T2)]
For any $a\in F$,
$[D,Y(a,\uz)]=\frac{d}{dz}Y(a,\uz)$ and $[\D,Y(a,\uz)]=\frac{d}{d\z}Y(a,\uz)$;
\end{enumerate}

\subsection{Locality and Associativity I}\label{sec_locality}
Let $(F,Y,\1,D,\D)$ be a translation covariant full prevertex algebra.
%In this section, we prove that a locality condition of vertex operators are enough to prove the axiom of a full vertex algebra, which is an analogy of Goddard's axiom of a vertex algebra.

In this section, we consider the following conditions:
\begin{enumerate}
\item[FL)]
For any $a_1,a_2,a_3 \in F$ and $u \in F^\vee$, there exists $\mu \in \GCor_2$ such that
\begin{align*}
u(Y(a_2,\uz_2)Y(a_1,\uz_1)a_3) &= \mu|_{|z_1|>|z_2|}\\
u(Y(a_1,\uz_1)Y(a_2,\uz_2)a_3) &=\mu|_{|z_2|>|z_1|}
\end{align*}
\end{enumerate}
and
\begin{enumerate}
\item[FA)]
For any $a_1,a_2,a_3 \in F$ and $u \in F$, there exists $\mu \in \GCor_2$ such that
\begin{align*}
u(Y(a_1,\uz_1)Y(a_2,\uz_2)a_3) &= \mu|_{|z_1|>|z_2|}\\
u(Y(Y(a_1,\uz_0)a_2,\uz_2)a_3) &=\mu|_{|z_2|>|z_1-z_2|}
\end{align*}
\end{enumerate}
and
\begin{enumerate}
\item[FSS)]
$Y(a_1,\uz)a_2=e^{Dz+\D\z}Y(a_2,-\uz)a_1$ for any $a_1,a_2\in F$.
\end{enumerate}

The condition (FL) (resp. (FA) and (FSS)) is a generalization of the locality (resp. associativity,  skew-symmetry) of a vertex algebra.
%and the condition (FA) is a generalization of the associativity of vertex operators.
%The condition (FSS) is called a skew-symmetry.

\begin{lem}\label{local_skew}
If a translation covariant full prevertex algebra $(F,Y,\1,D,\D)$ satisfies condition (FL),
then the skew-symmetry, (FSS), holds.
\end{lem}
\begin{proof}
Since $DY(a,z)\1=d/dzY(a,z)\1$, we have $Y(a,z)\1=\exp(Dz+\D\z)a$,
which implies that $DF_{h,\h}\subset F_{h+1,\h}$ and $\D F_{h,\h} \subset F_{h,\h+1}$.
Let $a_i \in F_{h_i,\h_i}$ ($i=1,2$) and $u \in F_{h_0,\h_0}^*$.
Then,
\begin{align*}
u(Y(a_1,\uz_1)Y(a_2,\uz_2)\1)=u(Y(a_1,\uz_1)\exp(Dz_2+\D\z_2)a_2)\\
=\lim_{z_{12} \to (z_1-z_2)|_{|z_1|>|z_2|}} u(\exp(Dz_2+\D\z_2)a_2)Y(a_1,\uz_{12})a_2).
\end{align*}
Set $h=h_1+h_2-h_0$ and $\h=\h_1+\h_2-\h_0$.
Then, by (PV5),
\begin{align*}
u(\exp(Dz_2+\D\z_2)a_2)Y(a_1,\uz_{12})a_2)
=\sum_{s_1,\s_1 \in \R}\sum_{n,\n \in \Z_{\geq 0}}\frac{1}{n!\n!}
u(D^n\D^\n a_1(s_1,\s_1)a_2)z_{12}^{-\s_1-1}\z_{12}^{-\s_1-1}z_2^n \z_2^\n \\
=\sum_{n,\n \in \Z_{\geq 0}}\frac{1}{n!\n!}
u(D^n\D^\n a_1(h+n-1,\h+\n-1)a_2)z_{12}^{-h-n}\z_{12}^{-\h-\n}z_2^n \z_2^\n.
\end{align*}
By (PV1), there exists an integer $N$ such that $a_1(h+n-1,\h+\n-1)a_2=0$ for any $n \geq N$
or $\n \geq N$.
Thus,
$z_{12}^{N+h}\z_{12}^{\bar{N}+\h}  u(\exp(Dz_2+\D\z_2)a_2)Y(a_1,\uz_{12})a_2) \in \C[z_{12},z_2,\z_{12},\z_2]$.
Hence, by (FL),
$\{(z_1-z_2)^{N+h}(\z_1-\z_2)^{\bar{N}+\h}\}|_{|z_1|>|z_2|}u(Y(a_1,\uz_1)Y(a_2,\uz_2)\1)
=\{(z_1-z_2)^{N+h}(\z_1-\z_2)^{\bar{N}+\h}\}|_{|z_2|>|z_1|}u(Y(a_2,\uz_2)Y(a_1,\uz_1)\1),
$
which implies
\begin{align*}
(-z_2)^{N+h}(-\z_2)^{\bar{N}+\h}u(Y(a_2,z_2)a_1)
&=\lim_{z_1\to 0}\{(z_1-z_2)^{N+h}(\z_1-\z_2)^{\bar{N}+\h}\}|_{|z_1|>|z_2|}u(Y(a_1,z_1)Y(a_2,z_2)\1)\\
&=\lim_{z_1\to 0}\lim_{z_{12} \to (z_1-z_2)|_{|z_1|>|z_2|}}
z_{12}^{N+h}\z_{12}^{\bar{N}+\h}u(\exp(Dz_2+\D\z_2)a_2)Y(a_1,z_{12})a_2)\\
&=(-z_{2})^{N+h}(-\z_{2})^{\bar{N}+\h}u(\exp(Dz_2+\D\z_2)a_2)Y(a_1,-z_{2})a_2).
\end{align*}
Hence, the assertion holds.
\end{proof}

Then, we have:
\begin{prop}\label{locality}
Assume that a translation covariant full prevertex algebra $(F,Y,\1,D,\D)$ satisfies the condition (FL)
and the spectrum of $F$ is bounded below, that is,
there exists $N \in \R$ such that $F_{h,\h}=0$ for any $h \leq N$ or $\h \leq N$.
Then, $(F,Y,\1)$ is a full vertex algebra.
\end{prop}

\begin{proof}
Let $a_1,a_2,a_3 \in F$ and $u \in F^\vee$.
By (FL), there exists $\mu \in \GCor_2$ such that
\begin{align*}
u(Y(a_1,\uz_1)Y(a_2,\uz_2)a_3) &= \mu|_{|z_1|>|z_2|}\\
u(Y(a_2,\uz_2)Y(a_1,\uz_1)a_3)&=\mu|_{|z_2|>|z_1|}.
\end{align*}
By Lemma \ref{local_skew},
\begin{align*}
u(Y(a_1,\uz_1)Y(a_2,\uz_2)a_3)=\lim_{z_{12} \to (z_1-z_2)|_{|z_1|>|z_2|}}
u(\exp(Dz_2+\D\z_2)Y(a_1,\uz_{12})Y(a_3,-\uz_2)a_2).
\end{align*}
By the assumption, $u(D^n\D^\n-)=0$ for sufficiently large $n$ or $\n$ in $\Z_{\geq 0}$.
Then, by (FL), 
\begin{align*}
u(\exp(Dz_2+\D\z_2)Y(a_1,\uz_{12})Y(a_3,-\uz_2)a_2)&=\mu|_{|z_{12}|>|z_2|}\\
u(\exp(Dz_2+\D\z_2)Y(a_3,-\uz_2)Y(a_1,\uz_{12})a_2)&= \mu|_{|z_2|>|z_{12}|}.
\end{align*}
Since $u(\exp(Dz_2+\D\z_2)Y(a_3,-\uz_2)Y(a_1,\uz_{12})a_2)=u(Y(Y(a_1,\uz_{12})a_2,\uz_2)a_3)$,
(FV5) holds.
\end{proof}

In the above proposition, we assume that the spectrum of $F$ is bounded below.
However, in the proof, we only use the property that 
for any $a_1,a_2,a_3\in F$ and $u\in F^\vee$,
there exits $N$ such that 
$u(D^n\D^m Y(a,z_1)Y(b,z_2)c)=0$ for any $n \geq N$ or $m \geq N$.
Thus, we have:
\begin{prop}\label{graded_locality}
Let $(F,Y,\1,D,\D)$ be a translation covariant full prevertex algebra satisfies the condition (FL)
and $M$ be an abelian group.
Suppose that there exists an $M \times \R^2$-grading on $F$, $F=\bigoplus_{\al\in M,h,\h \in \R} F_{h,\h}^\al$ such that:
\begin{enumerate}
\item
$F_{h,\h}=\bigoplus_{\al \in M}F_{h,\h}^\al$ for any $h,\h \in \R$ and $\al \in M$
\item
$F_{h,\h}^\al(r,s)F_{h',\h'}^{\al'} \subset F_{h+h'-r-1,\h+\h'-s-1}^{\al+\al'}$ for any $r,s,h,\h,h',\h' \in \R$ and $\al,\al'\in M$;
\item
For any $\al$, there exists $N$ such that $F_{h,\h}^\al=0$ for any $h\geq N$ or $\h \geq N$.
\end{enumerate}
Then, $F$ is a full vertex algebra.
\end{prop}

\begin{prop}\label{associativity}
Assume that a translation covariant full prevertex algebra $(F,Y,\1,D,\D)$ satisfies
the condition (FA) and the skew-symmetry.
Then, $(F,Y,\1)$ is a full vertex algebra.
\end{prop}
\begin{proof}
We first prove that 
$Y(Da,\uz)=d/dz Y(a,\uz)$ for any $a\in F$.
For any $a,b \in F$ and $u \in F^\vee$,
by (FFA), there exists $\mu \in \GCor_2$ such that
\begin{align*}
u(Y(a,\uz_1)b) &=\mu|_{|z_1|>|z_2|}\\
u(Y(Y(a,\uz_0)\1,\uz_2)b) &= \mu|_{|z_2|>|z_1-z_2|}.
\end{align*}
By (PV5), $u(Y(a,\uz_1)b)=p(z_1) \in \C[z_1^\pm,\z_1^\pm,|z_1|^\R]$.
Thus,
$\mu|_{|z_2|>|z_1-z_2|}=\lim_{z_1\to (z_2+z_0)|_{|z_2|>|z_0|}} p(z_1)
=\exp(z_0d/dz_2+\z_0 d/d\z_2)p(z_2).
$
Hence,
\begin{align*}
u(Y(Y(a,\uz_0)\1,\uz_2)b) &=
\exp(z_0d/dz_2+\z_0 d/d\z_2)u(Y(a,\uz_2)b),
\end{align*}
which implies that
$Y(Da,\uz_2)=d/dz_2 Y(a,\uz_2)$ and similarly $Y(\D a,\uz_2)=d/d\z_2 Y(a,\uz_2)$.

Now, we will show the assertion.
Let $a_1,a_2,a_3 \in F$ and $u \in F^\vee$.
By (FA), there exists $\mu \in \GCor_2$ such that
\begin{align*}
u(Y(a_1,\uz_1)Y(a_2,\uz_2)a_3) &=\mu|_{|z_1|>|z_2|}\\
u(Y(Y(a_1,\uz_0)a_2,\uz_2)a_3) &= \mu|_{|z_2|>|z_1-z_2|}.
\end{align*}
By the skew symmetry,
\begin{align*}
u(Y(Y(a_1,\uz_0)a_2,\uz_2)a_3)&=u(Y(\exp(D z_0+\D \z_0) Y(a_2,-\uz_0)a_1,\uz_2)a_3)\\
&=\lim_{\uz_1\to (\uz_2+\uz_0)|_{|z_2|>|z_0|}} u(Y(Y(a_2,-\uz_0)a_1,\uz_1)a_3).
\end{align*}
Then, applying (FA) again, (FV5) holds.
\end{proof}

\subsection{Locality and Associativity II}\label{sec_locality2}
In this section, we give a sufficient condition in order to construct a full vertex algebra by using four point functions.
We remark that for a self-dual QP-generated full vertex operator algebra $F$,
there is the linear map $(\1,-):F\rightarrow \C$ such that Proposition \ref{convergence_standard} holds.

In this section, we assume that 
$(F,Y,\1,D,\D)$ is a translation covariant full prevertex algebra equipped with
a linear map $<>: F\rightarrow \C$ such that
\begin{enumerate}
\item[E1)]
$F_{h,\h}$ is a finite dimensional vector space for any $h,\h \in \R$;
\item[E2)]
$<\1>=1$ and $<Da>=<\D a>=0$ for any $a \in F$;
\item[E3)]
$<F_{h,\h}>=0$ for any $(h,\h)\neq (0,0)$;
\item[E4)]
There exists $N \in \R$ such that $F_{h,\h}=0$ for $h \leq N$ or $\h \leq N$.
\end{enumerate}
We remark that if $F_{0,0}=\C\1$,
then $<>$ is simply the projection $F \rightarrow F_{0,0}$.

For $(F,Y,\1,D,\D,<>)$, set 
$$R_{F,<>}=\{b \in F\;|\; <Y(a,\uz)b>=0 \fora a\in F \}.
$$
We say $(F,Y,\1,D,\D,<>)$ is non-degenerate if $R_{F,<>}=0$.

In this section, we consider the following conditions:
\begin{enumerate}
\item[FFL)]
For any $a_1,a_2,a_3,a_4 \in F$, there exists $\phi \in \Cor_4$ such that
\begin{align*}
<Y(a_1,\uz_1)Y(a_2,\uz_2)Y(a_3,\uz_3)Y(a_4,\uz_4)\1>&=e_{1(2(3(4\star)))}\phi \\
<Y(a_1,\uz_1)Y(a_3,\uz_3)Y(a_2,\uz_2)Y(a_4,\uz_4)\1>&=e_{1(3(2(4\star)))}\phi
\end{align*}
\end{enumerate}
and
\begin{enumerate}
\item[FFA)]
For any $a_1,a_2,a_3,a_4 \in F$, there exists $\phi \in \Cor_4$ such that
\begin{align*}
<Y(a_1,\uz_1)Y(a_2,\uz_2)Y(a_3,\uz_3)Y(a_4,\uz_4)\1>&=e_{1(2(3(4\star)))}\phi \\
<Y(a_1,\uz_1)Y(a_4,\uz_4)Y(a_3,\uz_3)Y(a_2,\uz_3)\1>&=e_{1(4(3(2\star)))}\phi.
\end{align*}
\end{enumerate}

By Corollary \ref{cor_3} and (PV4), we have:
\begin{lem}\label{s3_4}
Let $(F,Y,\1,D,\D,<>)$ satisfy (FFL) or (FFA).
Then, for $a_1,a_2,a_3 \in F$, there exists $\phi \in \Cor_3$ such that
for any $\si \in S_3$,
\begin{align*}
<Y(a_{\si1},z_{\si1})Y(a_{\si 2},z_{\si2})Y(a_{\si3},z_{\si3})\1>&=e_{\si 1(2(3\star))}\phi.
\end{align*}
\end{lem}
\begin{lem}\label{skew_in}
Let $(F,Y,\1,D,\D,<>)$ satisfy (FFL) or (FFA).
Then, for $a_1,a_2,a_3 \in F$,
$$<Y(a_1,\uz_{12})Y(a_3,\uz_{32})a_2>=<Y(a_1,\uz_{12})\exp(Dz_{32}+\D \z_{32}) Y(a_2,-\uz_{32})a_3>.
$$
\end{lem}
\begin{proof}
Let $\phi \in \Cor_3$ such that 
$<Y(a_1,\uz_{13})Y(a_2,\uz_{23})a_3>=e_{1(23)}(\phi)$.
Similarly to the proof of Lemma \ref{formal_skew},
\begin{align*}
e_{1(32)}=\lim_{\substack{ \uz_{23} \to -\uz_{32} \\ \uz_{13} \to (\uz_{12}-\uz_{32})|_{|z_{12}|>|z_{32}|} }}  e_{1(23)}.
\end{align*}

Thus, by Lemma \ref{s3_4} and (E2),
\begin{align*}
<Y(a_1,\uz_{12})Y(a_3,\uz_{32})a_2>&=e_{1(32)}(\phi)\\
&=\exp(-z_{32}d/dz_{12}-\z_{32}d/d\z_{12})<Y(a_1,\uz_{12})Y(a_2,-\uz_{32})a_3>\\
&=<\exp(-Dz_{32}-\D \z_{32})Y(a_1,\uz_{12})\exp(Dz_{32}+\D \z_{32})Y(a_2,-\uz_{32})a_3>\\
&=<Y(a_1,\uz_{12})\exp(Dz_{32}+\D \z_{32})Y(a_2,-\uz_{32})a_3>.
\end{align*}
\end{proof}

We remark that we cannot obtain the skew-symmetry from the above lemma,
since the skew-symmetry holds only in $<->\in F^\vee$.
However, if $(F,Y,\1,D,\D,<>)$ is non-degenerate, 
then, by (E1) and (PV5), $F^\vee$ is spanned by $\{<a(s,\s)->\}_{a\in F,s,\s \in \R}$.
Thus, we have:
\begin{lem}\label{skew_FFA}
If $(F,Y,\1,D,\D,<>)$ is non-degenerate
and satisfy (FFL) or (FFA), then the skew-symmetry holds.
\end{lem}

\begin{prop}\label{locality_4}
If $(F,Y,\1,D,\D,<>)$ is non-degenerate
and satisfy (FFL) or (FFA), then $(F,Y,\1)$ is a full vertex algebra.
\end{prop}
\begin{proof}
First, we assume that (FFL) holds.
Then, by Lemma \ref{gcor}, the condition (FL) holds.
Thus, by Proposition \ref{locality}, $F$ is a full vertex algebra.
Next, we assume that (FFA) holds.
Then, for any $a_1,a_2,a_3,a_4 \in F$, there exists $\phi \in \Cor_4$ such that
\begin{align*}
<Y(a_1,\uz_{14})Y(a_2,\uz_{24})Y(a_3,\uz_{34})a_4>&=e_{1(2(34))}\phi \\
<Y(a_1,\uz_{12})Y(a_4,\uz_{42})Y(a_3,\uz_{32})a_2>&=e_{1(4(32))}\phi.
\end{align*}
Then, by Lemma \ref{skew_FFA} and Lemma \ref{formal_skew},
$<Y(a_1,\uz_{12})Y(a_4,\uz_{42})Y(a_2,\uz_{23})a_3>=e_{1(4(23))}\phi.
$
By Lemma \ref{skew_FFA} and (E2),
\begin{align*}
<Y&(a_1,\uz_{12})Y(a_4,\uz_{42})Y(a_2,\uz_{23})a_3>\\
&=<Y(a_1,\uz_{12})\exp(Dz_{42}+\D \z_{42})Y(Y(a_2,\uz_{23})a_3,-\uz_{42})a_4>\\
&= \exp(z_{42}d/dz_{12} +\z_{42}d/d\z_{12})<Y(a_1,\uz_{12})Y(Y(a_2,\uz_{23})a_3,-\uz_{42})a_4>.
\end{align*}
Then, by Lemma \ref{ass_1},
$$<Y(a_1,\uz_{14})Y(Y(a_2,\uz_{23})a_3,\uz_{24})a_4>=e_{1((23)4)}\phi.
$$
By Lemma \ref{gcor} again, the condition (FA) holds.
Thus, by Proposition \ref{associativity}, $F$ is a full vertex algebra.
\end{proof}

In the rest of this section, we study the condition that $(F,Y,\1,D,\D,<>)$ is non-degenerate.
A two-sided ideal of $(F,Y,\1,D,\D,<>)$ is a subspace $I \subset F$ such that
$I=\bigoplus_{h,\h\in \R} I\cap F_{h,\h}$, i.e., a graded subspace, and for any  $v\in I$ and $a \in F$,
$Y(a,z)v, Y(v,z)a \in I((z,\z,|z|^\R))$.

Then, we have:
\begin{lem}\label{left_ideal}
If $(F,Y,\1,D,\D,<>)$ satisfy (FFL) or (FFA),
the subspace $R_{F,<>}$ is a two-sided ideal of $F$.
%the maximal two-sided ideal of $(F,Y,\1,D,\D,<>)$,
%that is, any two-sided ideal of $(F,Y,\1,D,\D,<>)$ is a subset of $R_F$.
\end{lem}
\begin{proof}
By (PV5) and (E3), $R_{F,<>}$ is a graded subspace of $F$.
Let  $a_1,a_2 \in F$ and $v\in R_{F,<>}$.
It suffices to show that
$<Y(a_1,\uz_1)Y(a_2,\uz_2)Y(v,\uz_3)\1>=<Y(a_1,\uz_1)Y(v,\uz_3)Y(a_2,\uz_2)\1>=0$.
By Lemma \ref{skew_in} and (E2),
we have $<Y(a,\uz)b>=<Y(b,-\uz)a>$ for any $a,b \in F$.
Since
 $<Y(v,\uz_3)Y(a_1,\uz_1)Y(a_2,\uz_2)\1>=<Y(Y(a_1,\uz_1)Y(a_2,\uz_2)\1,-\uz_3)v>=0$,
by Lemma \ref{s3_4}, the assertion holds.
\end{proof}
By the above lemma, 
the vertex operator $Y(-,\uz):F/R_{F,<>} \rightarrow
\End\;(F/R_{F,<>})[[z,\z,|z|^\R]]$, $D,\D \in F/R_{F,<>}$
and $<>:F/R_{F,<>}\rightarrow \C$
are induced from $(F,Y,\1,D,\D,<>)$
and are non-degenerate.
We remark that by the construction, a parenthesized correlation function of $F/R_{F,<>}$ and $F$ are the same.
Thus, we have:
\begin{prop}\label{quotient}
Let $(F,Y,\1,D,\D,<>)$ satisfy (FFL) or (FFA).
Then, $F/R_{F,<>}$ is a full vertex algebra.
\end{prop}

$(F,Y,\1,D,\D,<>)$ is said to be simple if
there is no proper left ideal of $F$.
By the above lemma, we have:
\begin{prop}\label{radical}
Let $(F,Y,\1,D,\D,<>)$ satisfy (FFL) or (FFA).
If $(F,Y,\1,D,\D,<>)$ is simple, then it is non-degenerate.
Conversely, if $(F,Y,\1,D,\D,<>)$ is non-degenerate and $F_{0,0}=\C\1$,
then $(F,Y,\1,D,\D,<>)$ is simple.
\end{prop}
%Existence of maximal ideal
\begin{proof}
By Lemma \ref{left_ideal},
$R_{F,<>}$ is a left ideal.
If $F$ is simple, then $R_{F,<>}=0$ by (E2), thus non-degenerate.
Suppose that $(F,Y,\1,D,\D,<>)$ is non-degenerate and $F_{0,0}=\C\1$
and let $I \subset F$ be a left ideal which does not contain $\1$.
Since $I\cap F_{0,0}=0$,
$<Y(a,\uz)v>=0$ for any $a\in F$ and $v \in I$.
Hence, $I \subset R_{F,<>}=0$.
\end{proof}

\subsection{Bootstrap}
In this section, we study the bootstrap equation.
Set $L(-1)=D$ and $\Ld(-1)=\D$
and let $L(0),\Ld(0) \in \End \;F$ be linear maps defined by
$L(0)|_{F_{h,\h}}=h$ and $\Ld(0)|_{F_{h,\h}}=\h$ for any $h,\h\in \R$.
We assume that $(F,Y,\1,D,\D)$ is a translation covariant full prevertex algebra equipped with
linear maps $L(1),\Ld(1) \in \End\, F$ and a non-degenerate symmetric bilinear form $(-,-):F \times F \rightarrow \C$
such that:
\begin{enumerate}
\item
There exists $N\in \R$ such that $F_{h,\h}=0$ for any $h \leq N$ or $\h \leq N$;
\item
$\dim F_{h,\h}$ is finite for any $h,\h\in \R$;
\item
For any $i,j=-1,0,1$,
\begin{align*}
[L(i),L(j)]=(i-j)L(i+j),\\
[\Ld(i),\Ld(j)]=(i-j)\Ld(i+j),\\
[L(i),\Ld(j)]=0;
\end{align*}
\item
For any $i=-1,0,1$ and $a,b\in F$,
\begin{align*}
(L(i)a,b)&=(a,L(-i)b)\\
(\Ld(i)a,b)&=(a,\Ld(-i)b);
\end{align*}
\item
For $k=-1,0,1$,
\begin{align*}
[L(k),Y(a,\uz)]=\sum_{i=0}^{k+1} \binom{k+1}{i+1} Y(L(i+1)a,\uz)z^{-i}\\
[\Ld(k),Y(a,\uz)]=\sum_{i=0}^{k+1} \binom{k+1}{i+1} Y(\Ld(i+1)a,\uz)\z^{-i};
\end{align*}
\item
$L(1)\1=\Ld(1)\1=0$;
\item
For any $a,b,c \in F$,
$$(a,Y(b,\uz)c)=(Y(S_{z}b,\uz^{-1})a,c).$$
\end{enumerate}

Let $(F,Y,\1,L(i),\Ld(i),(-,-))$ satisfy the above condition.
For any $a_1,a_2,a_3,a_4 \in F$,
define s,t,u-channels by
%\begin{align*}
%S_{(12)(34)}(a_1,a_2,a_3,a_4)&=\Bigl(\1,\Bigl(\Bigl(a_1(z_{12})a_2\Bigr)(z_{24})a_3(z_{34})a_4\Bigr) \tag{s-channel}\\
%S_{(14)(23)}(a_1,a_2,a_3,a_4)&=\Bigl(\1,\Bigl(\Bigl(a_1(z_{14})a_4\Bigr)(z_{43})a_2(z_{23})a_3\Bigr) \tag{t-channel}\\
%S_{(13)(24)}(a_1,a_2,a_3,a_4)&=\Bigl(\1,\Bigl(\Bigl(a_1(z_{13})a_3\Bigr)(z_{34})a_2(z_{24})a_4\Bigr) \tag{u-channel}.
%\end{align*}
\begin{align*}
S_{(21)(34)}(a_1,a_2,a_3,a_4)&=\Bigl(\1,\Bigl(\Bigl(a_2(z_{21})a_1\Bigr)(z_{14})a_3(z_{34})a_4\Bigr) \tag{s-channel}\\
S_{(41)(32)}(a_1,a_2,a_3,a_4)&=\Bigl(\1,\Bigl(\Bigl(a_4(z_{41})a_1\Bigr)(z_{12})a_3(z_{32})a_2\Bigr) \tag{t-channel}\\
S_{(31)(24)}(a_1,a_2,a_3,a_4)&=\Bigl(\1,\Bigl(\Bigl(a_3(z_{31})a_1\Bigr)(z_{14})a_2(z_{24})a_4\Bigr) \tag{u-channel}.
\end{align*}

$(F,Y,\1,L(i),\Ld(i),(-,-))$ is said to 
satisfy a bootstrap equation for the s-channel and t-channel (resp. for the s-channel and u-channel)
if for any $a_1,a_2,a_3,a_4 \in F$,
there exists $\phi \in \Cor_4$ such that
$S_{(21)(34)}(a_1,a_2,a_3,a_4)=e_{(21)(34)}(\phi)$ and $S_{(41)(32)}(a_1,a_2,a_3,a_4)=e_{(41)(32)}(\phi)$
(resp. $S_{(21)(34)}(a_1,a_2,a_3,a_4)=e_{(21)(34)}(\phi)$ and $S_{(31)(24)}(a_1,a_2,a_3,a_4)=e_{(31)(24)}(\phi)$).

The notion, QP-generated, is defined for $(F,Y,\1,L(i),\Ld(i),(-,-))$ similarly to the full vertex operator algebra (see Section \ref{sec_QP}).

Let $(F,Y,\1,L(i),\Ld(i),(-,-))$ be QP-generated and satisfy the bootstrap equation for the s-channel and t-channel.
Set $<->=(\1,-)$. We will show that $(F,Y,\1,D,\D,<>)$ satisfy the assumption of Proposition \ref{locality_4}.
It suffices to show that (FFA) holds.
Similarly to the proof of Proposition \ref{convergence_standard},
we may assume that all $a_i \in F_{h_i,\h_i}$ are quasi-primary states.
Then, by $\mathrm{PSL}_2\C$-symmetry, in order to verify (FFA),
it suffices to show that there exists $f \in \F$ such that
\begin{align*}
Q(h,z)^{-1}(\1,Y(a_1,z_1)Y(a_2,z_2)Y(a_3,z_3)a_4) &=e_{1(2(34))}(f\circ \xi), \\
Q(h,z)^{-1}(\1,Y(a_1,z_1)Y(a_4,z_4)Y(a_2,z_2)a_3) &= e_{1(4(32))}(f\circ \xi).
\end{align*}
\begin{comment}
Since 
\begin{align*}
(\1,Y(Y(a_2,z_{21})a_1,z_{14})Y(a_3,z_{34})a_4)
&=(\1,Y(\exp(Dz_{12}+\D \z_{12})Y(a_2,-z_{12})a_1,z_{24})Y(a_3,z_{34})a_4)\\
&=\lim_{(z_{21},z_{14}) \to (-z_{12},(z_{24}+z_{12})|_{|z_{24}|>|z_{12}|})}
(\1,Y(Y(a_2,z_{21})a_1,z_{14})Y(a_3,z_{34})a_4),
\end{align*}
similarly to the proof of Lemma \ref{ass_1},
there exists $\phi \in \Cor_4$ such that
\begin{align*}
S_{(21)(34)}(a_1,a_2,a_3,a_4)&=e_{(21)(34)}(Q(h,z)\cdot f\circ \xi),\\
S_{(41)(32)}(a_1,a_2,a_3,a_4)&=e_{(41)(23)}(Q(h,z)\cdot f \circ \xi).
\end{align*}
\end{comment}
By Lemma \ref{vertex_channel},
we can calculate
$v_{1(2(34))}(Q(h,z)^{-1}S_{1(2(34))})$ and $v_{1(4(32))}(Q(h,z)^{-1}S_{1(4(32))})$
similarly to the proof of Theorem \ref{consistency},
which are equal to $j(p,f)$ and $j(1-p,f)$.
The bootstrap equation for the s-channel and u-channel implies the same result.
The detail verification is left to the reader.
Thus, we have:
\begin{prop}\label{bootstrap}
If  $(F,Y,\1,L(i),\Ld(i),(-,-))$ is QP-generated and satisfy the bootstrap equation for the
s-channel and t-channel or the s-channel and u-channel, then $F$ is a full vertex algebra.
\end{prop}

\section{Example}
% associated with AH pair!!
% conceptial meaning,
In this section, we construct a full vertex algebras from an AH pair,
which is introduced in \cite{M1}.
One can construct an AH pair from an even lattice, called twisted group algebra \cite{FLM,M1}.
The full vertex algebra constructed from an even lattice is appeared in the toroidal compactification of string theory
and is shown to be a QP-generated self-dual full vertex operator algebra.
A conceptual origin of AH pairs is provided in Section \ref{sec_braided},
where we show that
an AH pair is a commutative algebra object in some braided tensor category.

\subsection{Full vertex algebra and AH pair}
We first recall a concept of an AH pair introduced in \cite{M1}.
Let $H$ be a finite-dimensional vector space over $\mathbb{R}$ equipped with a non-degenerate symmetric bilinear form 
$(-,-):H\times H\rightarrow \R$ and $A$ a unital associative algebra over $\mathbb{C}$ with the unity $1$. 
Assume that $A$ is graded by $H$ as $A=\bigoplus_{\alpha\in H}A^\alpha$. 

We will say that such a pair $(A,H)$ is an {\it AH pair} if the following conditions hold:
\begin{enumerate}
\item[AH1)]
$1\in A^0$ and $A^\alpha A^\beta\subset A^{\alpha+\beta}$ for any $\alpha,\beta\in H$ ;
\item[AH2)]
$A^\alpha A^\beta\ne 0$ implies $(\alpha,\beta)\in \mathbb{Z}$;
\item[AH3)]
For $v\in A^\alpha,w\in A^\beta$, $vw=(-1)^{(\alpha,\beta)}wv$.
\end{enumerate}
An AH pair $(A,H)$ is called an {\it even AH pair} if
\begin{enumerate}
\item[EAH]
$(\al,\al)\in 2\Z$ for any $\al \in H$ with $A^\al \neq 0$.
\end{enumerate}

Let $P(H)$ be the set of linear maps $p \in \End \;H$ such that:
\begin{enumerate}
\item[P1)]
$p$ is a projection, that is, $p^2=p$;
\item[P2)]
The subspaces $\ker p$ and $\ker (1-p)$ are orthogonal to each other.
\end{enumerate}
Let $P_{>}(H)$ be the subset of $P(H)$ such that
\begin{enumerate}
\item[P3)]
$\ker (1-p)$ is positive definite and $\ker p$ is negative-definite.
\end{enumerate}

Let $(A,H)$ be an even AH pair and $p\in P(H)$.
In this section, we construct a full vertex algebra $F_{A,H,p}$ for the triple $(A,H,p)$.

Set $\p=1-p$ and $H_l=\ker \p$ and $H_r=\ker p$.
Define the new symmetric bilinear forms $(-,-)_p: H\times H \rightarrow \R$
by 
$$(h,h')_p=(ph,ph')-(\p h,\p h').$$
By (P1) and (P2), $(-,-)_p$ is non-degenerate.
Let $\hat{H}^p=\bigoplus_{n \in \Z} H\otimes \C t^n \oplus \C c$ be the affine Heisenberg Lie algebra associated with $(\hat{H}^p,(-,-)_p)$
and $\hat{H}_{\geq 0}^p=\bigoplus_{n \geq 0} H\otimes \C t^n \oplus \C c$ a subalgebra of $\hat{H}^p$.
Define the action of $\hat{H}_{\geq 0}^p$ on $A$ by
\begin{align*}
c a&=a \\
h\otimes t^n a &=
\begin{cases}
0, &n \geq 1,\cr
(h,\al)_p a, &n = 0
\end{cases}
\end{align*}
for $\al \in H$ and $a \in A^\al$.
Let $F_{A,H,p}$ be the $\hat{H}^p$-module induced from $A$
and $M_{H,p}$ be the submodule of $F_{A,H,p}$ generated by the unit $1 \in A^0$ as a $\hat{H}^p$-module.
Denote by $h(n)$ the action of $h \otimes t^n$ on $F_{A,H,p}$ for $n \in \Z$.
For $h \in H$, set 
\begin{align*}
h(z,\z) &= \sum_{n \in \Z}( (ph)(n)z^{-n-1}+ (\p h)(n)\z^{-n-1}) \in
\End F_{A,H,p}[[z^\pm,\z^\pm]]\\
h^+(\uz) &=\sum_{n \geq 0}( (ph)(n)z^{-n-1}+ (\p h)(n)\z^{-n-1}) \\
h^-(\uz) &=\sum_{n \geq 0}( (ph)(-n-1)z^{n}+ (\p h)(-n-1)\z^{n}).\\
E^+(h,\uz)&=
\exp\biggl(-\sum_{n\geq 1}(\frac{p h(n)}{n}z^{-n}+\frac{\p h(n)}{n}\z^{-n})
\biggr)\\
E^-(h,\uz)&=\exp\biggl(-\sum_{n\leq -1}(\frac{p h(n)}{n}z^{-n} +\frac{\p h(n)}{n}\z^{-n})
\biggr).
\end{align*}
For $h_r \in H_r$ and $h_l \in H_l$,
$h_r(\uz)$ and $h_l(\uz)$ are denoted by $h_r(z)$ and $h_l(z)$.

Then, similarly to the case of a lattice vertex algebra, we have:
\begin{lem}\label{lattice_commutator}
For any $h_1, h_2 \in H$,
$$E^+(h_1,\uz_1)E^-(h_2,\uz_2)
= \biggl( \sum_{n,\n \geq 0} \binom{(ph_1,ph_2)_p}{n}\binom{(\p h_1,\p h_2)_p}{\n} z_2^{n}z_1^{-n}\z_2^{\n}\z_1^{-\n}\biggr) E^-(h_2,\uz_2) E^+(h_1,\uz_1).$$
\end{lem}
We remark that the formal power series $\sum_{n,\n \geq 0} \binom{(ph_1,ph_2)_p}{n}\binom{(\p h_1,\p h_2)_p}{\n}\; z_2^{n}z_1^{-n}\z_2^{\n}\z_1^{-\n}$ is equal to $(1-z_2/z_1)^{(ph_1,ph_2)_p} (1-\z_2/\z_1)^{(\p h_1,\p h_2)_p}|_{|z_1|>|z_2|}.$

Let $\al \in H$ and $a \in A^\al$.
Denote by $l_a \in \End\, A$ the left multiplication by $a$ and define the linear map
$l_a z^{p\alpha} \z^{\p \alpha} :A \rightarrow A[z,\z,|z|^\R]$
by $l_a z^{p\alpha} \z^{\p \alpha}b= z^{(p\alpha,p\beta)_p}\z^{(\p\al,\p\be)_p} ab$ 
for $\be\in H$ and $b \in A^\be$.
Assume that $ab \neq 0$.
Then, $(\al,\be) \in \Z$
and since $z^{(p\alpha,p\beta)_p}\z^{(\p\al,\p\be)_p}=|z|^{(p\alpha,p\beta)_p}\z^{(\p\alpha,\p\beta)_p-
(p\al,p\be)_p}$ and ${(\p\alpha,\p\beta)_p}-(p\al,p\be)_p=-(\al,\be) \in \Z$,
we have $l_a z^{p\alpha} \z^{\p \alpha}b \in A[z,\z,|z|^\R]$.
Then, set
\begin{align*}
a(\uz)&=E^-(\al,\uz)E^+(\al,\uz)l_{a} z^{p \alpha} \z^{\p \alpha},
\end{align*}
which is a linear map 
$F_{A,H,p} \rightarrow F_{A,H,p}[[z,\z,|z|^\R]].$

By Poincar{\'e}-Birkhoff-Witt theorem,
$F_{A,H,p} \cong M_{H,p}\otimes A$, that is, $F_{A,H,p}$ is spanned by 
$$\{h_l^1(-n_1-1)\dots h_l^l(-n_l-1)h_r^1(-\n_1-1)\dots h_r^k(-\n_k-1)a \},
$$
where $h_l^i \in H_l$, $n_i \in \Z_{\geq 0}$ and $h_r^j \in H_r$, $\n_j \in \Z_{\geq 0}$
 for any $1\leq i \leq l$ and $1\leq j \leq k$ and $a \in A$.
Then, a map
$Y:F_{A,H,p} \rightarrow \End F_{A,H,p}[[z,\z,|z|^\R]]$ is defined inductively as follows:
For $a \in A$,
define $Y(a,\uz)$ by $Y(a,\uz)=a(\uz)$.
Assume that $Y(v,\uz)$ is already defined for $v \in F_{A,H,p}$.
Then, for $h_r \in H_r$ and $h_l \in H_l$ and $n,\n \in \Z_{\geq 0}$,
$Y(h_l(-n-1)v,\uz)$ and $Y(h_r(-\n-1)v,\uz)$ is defined by 
\begin{align*}
Y(h_l(-n-1)v,\uz)&=1/n!\Bigl(\frac{d}{dz}^n h_l^-(z)\Bigr)Y(v,\uz)+Y(v,\uz)1/n!\Bigl(\frac{d}{dz}^n h_l^+(z)\Bigr) \\
Y(h_r(-\n-1)v,\uz)&=1/\n!\Bigl(\frac{d}{d\z}^\n h_r^-(\z)\Bigr)Y(v,\uz)+Y(v,\uz)1/\n!\Bigl(\frac{d}{d\z}^\n h_r^+(\z)\Bigr).
\end{align*}
By the direct calculus, we can easily prove the following lemma:
\begin{lem}
For any $v \in F_{A,H,p}$ and $h_l \in H_l$ and $h_r \in H_r$,
\begin{align*}
[h_l(n),Y(v,\uz)]&=\sum_{i\geq 0} \binom{n}{i}Y(h_l(i)v,\uz)z^{n-i},\\
[h_r(n),Y(v,\uz)]&=\sum_{i\geq 0} \binom{n}{i}Y(h_r(i)v,\uz)\z^{n-i}.
\end{align*}
\end{lem}
By the above lemma,
\begin{align*}
[h_l^+(z_1),Y(v,\uz_2)]&=\sum_{i\geq 0,n \geq 0} \binom{n}{i}Y(h_l(i)v,\uz_2)z_1^{-n-1} z_2^{n-i}\\
&=\sum_{i \geq 0}Y(h_l(i)v,\uz_2)\frac{1}{i!} \frac{d}{dz_2}^i  \sum_{n \geq 0} z_1^{-n-1} z_2^{n} \\
&=\sum_{i \geq 0} Y(h_l(i)v,\uz_2) \frac{1}{(z_1-z_2)^{i+1}}|_{|z_1|>|z_2|}.
\end{align*}
Thus, we have:
\begin{lem}\label{intertwiner_lattice}
For any $v \in F_{A,H,p}$ and $h_l \in H_l$,
\begin{align*}
[h_l^+(z_1),Y(v,\uz_2)]&=\sum_{i \geq 0} Y(h_l(i)v,\uz_2) \frac{1}{(z_1-z_2)^{i+1}}|_{|z_1|>|z_2|},\\
[Y(v,\uz_1), h_l^-(z_2)]&= \sum_{i\geq 0} Y(h_l(i)v,\uz) \frac{(-1)^{i+1}}{(z_1-z_2)^{i+1}}|_{|z_1|>|z_2|}.
\end{align*}
\end{lem}

Set
\begin{align*} 
\1&=1\otimes 1, \\
\omega&  =1/2 \sum_{i=1}^{\dim H_l} h_l^i (-1)h_l^i,\\
\bar{\om}& =1/2 \sum_{j=1}^{\dim H_r} h_r^j (-1)h_r^j,
\end{align*}
where $h_l^i$ and $h_r^j$ is an orthonormal basis of $H_l\otimes_{\R} \C$ and $H_r\otimes_{\R} \C$ with respect to the bilinear form $(-,-)_p$.

We will prove the following proposition by using Proposition \ref{graded_locality}:
\begin{prop}\label{AH_full}
For an even AH pair $(A,H)$ and $p \in P(H)$,
$(F_{A,H,p},Y,\1)$ is a full vertex algebra.
\end{prop}
Set $F=F_{A,H,p}$ and $D=\omega(0)$ and $\D=\bar{\om}(0)$
and 
$$F_{h,\h}^\al=\{v \in F\;|\; \omega(1)v=hv, \bar{\om}(1)v=\h v, h(0)v=(\al,h)_p v \fora h \in H \}$$
and $$F_{h,\h}=\bigoplus_{\al \in H} F_{h,\h}^\al,$$
for $h,\h \in \R$ and $\al \in H$.
Then, $F=\bigoplus_{h,\h \in \R} \bigoplus_{\al \in H}F_{h,\h}^\al$.
In order to prove $(F,Y,\1)$ is a full vertex algebra,
it suffices to show that
$(F,Y,\1,D,\D)$ satisfies the condition in Proposition \ref{graded_locality}.
An easy computation show that $(F,Y,\1,D,\D)$ is a translation covariant full prevertex algebra
and $F=\bigoplus_{h,\h \in \R} \bigoplus_{\al \in H}F_{h,\h}^\al$ satisfies
the conditions (1),(2) and (3) in Proposition \ref{graded_locality}.
Thus, it suffices to show that the condition (FL).

For $\al \in H$ and $n,m\in \Z_{\geq 0}$,
it is easy to show that 
$F_{\frac{1}{2}(p\al,p\al)_p+n,\frac{1}{2}(\p\al,\p\al)_p+m}^\al$
is spanned by
$\{h_l^1(-i_1)\dots h_l^k(-i_k)h_r^1(j_1)\dots h_r^l(j_l)a\}$,
where $k,l \in \Z_{\geq 0}$, $h_l^a \in H_l$, $h_r^b \in H_r$, $a\in A^\al$,
 $i_a,j_b \in \Z_{\geq 1}$,
$i_1+\dots+i_k=n$ and $j_1+\dots+j_l=m$ for any $a=1,\dots,k$ and $b=1,\dots,l$.
Then,
$$F^\al=\bigoplus_{n,m \in \Z_{\geq 0}}F_{\frac{1}{2}(p\al,p\al)_p+n,\frac{1}{2}(\p\al,\p\al)_p+m}^\al
$$
and
$$F_{\frac{1}{2}(p\al,p\al)_p,\frac{1}{2}(\p\al,\p\al)_p}^\al=A^\al.
$$
Let $a^* \in A^\vee=\bigoplus_{\al \in H} (A^\al)^*$
and $<a^*,->$ be the linear map $F\rightarrow \C$ defined by
the composition of the projection
$F=A\oplus \bigoplus_{\substack{n,m\in \Z_{\geq 0}\\(n,m)\neq (0,0)}} 
F_{\frac{1}{2}(p\al,p\al)_p+n,\frac{1}{2}(\p\al,\p\al)_p+m}^\al \rightarrow A
$ and $a^*: A\rightarrow \C$.
Then, it is easy to verify $<a^*,->$ is a highest weight vector,
that is, $<a^*,h(-n)->=0$ for any $n \geq 1$ and $h \in H$.
Thus, for any $\al \in H$,
we have:
\begin{align*}
E^+(\al,\uz)\1&=\1,\\
<a^*, E^-(\al,\uz) - >&= <a^*, - >.
\end{align*}

Thus, by using the above fact and Lemma \ref{lattice_commutator},
we have:
\begin{lem}\label{lattice_four}
For $\al_0,\al_1,\al_2,\al_3 \in H$ and $a_0^* \in (A^{\al_0})^*$ and $a_i \in A^{\al_i}$ (for $i=1,2,3$),
\begin{align*}
<a_0^*,&Y(a_1,\uz_1)Y(a_2,\uz_2)a_3>\\
&=z_1^{(p\al_1,p\al_3)_p}z_2^{(p\al_2,p\al_3)_p}(z_1-z_2)^{(p\al_1,p\al_2)_p}|_{|z_1|>|z_2|}\\
&\z_1^{(\p\al_1,\p\al_3)_p}\z_2^{(\p\al_2,\p\al_3)_p}(\z_1-\z_2)^{(\p\al_1,\p\al_2)_p}|_{|z_1|>|z_2|}
<a_0^*,a_1a_2a_3>.
\end{align*}
\end{lem}
Assume that none of $a_1a_2,a_1a_3,a_2a_3$ is equal to zero.
Then, $(\al_1,\al_2),(\al_2,\al_3),(\al_1,\al_3) \in \Z$ by (AH2).
Set $$\phi=z_1^{(p\al_1,p\al_3)_p}z_2^{(p\al_2,p\al_3)_p}(z_1-z_2)^{(p\al_1,p\al_2)_p}|_{|z_1|>|z_2|}
\z_1^{(\p\al_1,\p\al_3)_p}\z_2^{(\p\al_2,\p\al_3)_p}(\z_1-\z_2)^{(\p\al_1,\p\al_2)_p}|_{|z_1|>|z_2|},$$
which is a function in $\GCor_2$.
The above lemma implies that 
$<a_0^*,Y(a_2,\uz_2)Y(a_1,\uz_1)a_3>=
(-1)^{(\al_2,\al_3)} \phi|_{{|z_1|>|z_2|}}<a_0^*, a_2a_1a_3>$.
By the definition of an AH pair, $a_3a_2=(-1)^{(\al_2,\al_3)}a_2a_3$.
Thus, (FL) holds.
If one of $a_1a_2,a_1a_3,a_2a_3$ is equal to zero,
then both $<a_0^*,a_1a_2a_3>$ and $<a_0^*, a_2a_1a_3>$ are zero,
which implies (FL).
For $u \in F^\vee$ and $v_1,v_2,v_3\in F$,
set
$$S(u;v_1,v_2,v_3)=u(Y(v_1,\uz_1)Y(v_2,\uz_2)v_3).$$
We can compute $S(u;v_1,v_2,v_3)$ by using Lemma \ref{lattice_four} and Lemma \ref{intertwiner_lattice}.
We observe that $\bigoplus (F_{h,\h}^\al)^*$ is spanned by
$\{<a^*,h_1(i_1)\dots h_k(i_k) - >\}$,
where $a^* \in (A^\al)^*$ and $i_1,\dots,i_k \Z_{\geq 1}$ and $h_1,\dots,h_k \in H$.
Let $u \in F^\vee$ and $v_1,v_2,v_3\in F$ and $h_l \in H_l$ and $n \in \Z_{\geq 1}$.
Then, by Lemma \ref{intertwiner_lattice},
\begin{align*}
S(&u(h_l(n)-); v_1,v_2,v_3)\\
&=u(h_l(n)Y(v_1,\uz_1)Y(v_2,\uz_2)v_3)\\
&=S(u; v_1,v_2,h_l(n)v_3)+
\sum_{i \geq 0}\binom{n}{i}\Bigl(S(u; h_l(i)v_1,v_2,v_3)z_1^{n-i}
+S(u; v_1,h_l(i)v_2,v_3)z_2^{n-i}\Bigr).
%\Bigl(\frac{1}{(z_1-z_2)^{n+i+1}}|_{|z_1|>|z_2|}\\&+\frac{1}{(z_1-z_3)^{n+i+1}}|_{|z_1|>|z_3|}S(v_1,v_2,h_l(i)v_3,v_4)+
%\frac{1}{(z_1-z_4)^{n+i+1}}|_{|z_1|>|z_4|}S(v_1,v_2,v_3,h_l(i)v_4)\Bigr)
\end{align*}
Since the above formula is symmetric for $v_1$ and $v_2$,
in order to show (FL), we may assume that $u=a^* \in A^\vee$.
Similarly, for $n \in \Z_{\geq 0}$,
\begin{align*}
a^*(Y(h_l(-n-1)v_1,\uz_1)Y(v_2,\uz_2)v_3)&=\sum_{i \geq 0}\binom{-i-1}{n}
\Bigl(\frac{(-1)^{n}}{(z_1-z_2)^{n+i+1}}|_{|z_1|>|z_2|}a^*(Y(v_1,\uz_1)Y(h_l(i)v_2,\uz_2)v_3)\\
&+\frac{1}{z_1^{n+i+1}} a^*(Y(v_1,\uz_1)Y(v_2,\uz_2)h_l(i)v_3)\Bigr)
%(-1)^n\binom{-i-1}{n} 
%\Bigl(\frac{(-1)^{n+i+1}}{(z_1-z_2)^{n+i+1}}|_{|z_1|>|z_2|}S(v_1,h_l(i)v_2,v_3,v_4)\\
%&+\frac{1}{(z_2-z_3)^{n+i+1}}|_{|z_2|>|z_3|}S(v_1,v_2,h_l(i)v_3,v_4)+
%\frac{1}{(z_2-z_4)^{n+i+1}}|_{|z_2|>|z_4|}S(v_1,v_2,v_3,h_l(i)v_4)\Bigr).
\end{align*}
and
\begin{align*}
a^*(Y(v_2,\uz_2)Y(h_l(-n-1)v_1,\uz_1)v_3)&=\sum_{i \geq 0}\binom{-i-1}{n}
\Bigl(\frac{(-1)^{n}}{(z_1-z_2)^{n+i+1}}|_{|z_2|>|z_1|}a^*(Y(h_l(i)v_2,\uz_2)Y(v_1,\uz_1)v_3)\\
&+\frac{1}{z_1^{n+i+1}}a^*(Y(v_2,\uz_2)Y(v_1,\uz_1)h_l(i)v_3),
\end{align*}
which implies that $a^*(Y(h_l(-n-1)v_1,\uz_1)Y(v_2,\uz_2)v_3)$ and $a^*(Y(v_2,\uz_2)Y(h_l(-n-1)v_1,\uz_1)v_3)$
are the expansion of the same function in the different region $|z_1|>|z_2|$ and $|z_2|>|z_1|$
if (FL) holds for $(a^*,v_1,h_l(i)v_2,v_3)$ and $(a^*,v_1,v_2, h_l(i)v_3)$ for any $i \geq 0$.
In this way, we can inductively show the assumption (FL), by Lemma \ref{lattice_four}
and Lemma \ref{intertwiner_lattice}.

\subsection{Troidal compactification of string theory}
In this section, we construct a full vertex operator algebra associated with an even lattice.
An even lattice is a free abelian group $L$ of finite rank equipped with
a bilinear map $(-,-):L\times L \rightarrow \Z$ such that
$(\al,\al) \in 2\Z$ for any $\al \in L$.
Set $H_L=L\otimes_\Z \R$.
Then, a bilinear form on $H_L$ is induced by the bilinear form on $L$,
denoted by $(-,-)$ again.
An even lattice is called non-degenerate if $(H_L,(-,-))$ is non-degenerate.
Let $L$ be a non-degenerate even lattice.
In \cite{FLM}, they construct an AH pair associated with an even lattice:
\begin{prop}\cite{FLM}
There exists a unital associative algebra $\C[\hat{L}]=\bigoplus_{\al \in L} \C e_\al$
 such that:
\begin{enumerate}
\item
$e_0$ is the unit;
\item
$e_\al e_\be \neq 0$ for any $\al,\be \in L$;
\item
$e_\al e_\be =(-1)^{(\al,\be)}e_\be e_\al$ for any $\al,\be \in L$.
\end{enumerate}
\end{prop}
In fact, a unital associative algebra satisfies the conditions in the above proposition is unique
up to isomorphism as an AH pair (\cite{FLM,M1}).
The algebra $\C[\hat{L}]$ is called a twisted group algebra,
which is clearly an even AH pair.
For $p \in P(H_L)$, set $F_{L,H_L,p}=F_{\C[\hat{L}],H_L,p}$.
By Proposition \ref{AH_full}, we have:
\begin{prop}\label{lattice_full}
For a non-degenerate even lattice $L$ and $p \in P(H_L)$,
$F_{L,H_L,p}$ is a full vertex algebra.
Furthermore, if $p \in P_{>}(H_L)$, then $(F_{L,H_L,p},Y,\1,\om_{H_L},\omb_{H_L})$ is
a self-dual QP-generated full vertex operator algebra.
\end{prop}
\begin{proof}
Let $p \in P_{>}(H_L)$ and set $F=F_{L,H_L,p}$.
Since the spectrum of $F_{L,H_L,p}$ is $\{(p\al,p\al)/2+n ,-(\p\al,\p\al)/2 +m \}_{\al \in L,n,m \in \Z_{\geq 0}}$
and the bilinear form $(p-,p-)$ on $H_l$ and $-(\p-,\p-)$ on $H_r$ is positive definite,
the spectrum is discrete, thus, bounded below.
It is easy to show that $(F,Y,\1,\om,\omb)$ is a simple full vertex operator algebra.
To verify the assertion, it suffices to show that $F$ satisfies the assumption of Corollary \ref{grading_qp}.
By the positivity of the bilinear forms, (PN1) and (PN2) is clear.
Let $0 \neq v \in F_{0,n}^\al$ for $\al \in L$.
Then, $(p\al,p\al)=0$ and, thus, $p\al=0$,
which implies that $D v=0$ and (PN4).
Let $v \in F_{1,n}^\al$ for $\al \in L$.
Suppose that $L(1)v \neq 0$. Then, $L(1)v \in F_{0,n}^\al$ implies that $p\al=0$,
that is, $\al \in H_r$.
Thus, $F_{1,n}^\al$ is spanned by $\{h_l(-1)e_\al \}_{h_l \in H_l}$ and $(\al,h_l)=0$.
Then, for $h_l \in H_l$, $L(1)h_l(-1)e_\al=h_l(0)e_\al=(h_l,\al)e_\al=0$, a contradiction,
which implies (PN3). Hence, the assertion holds.
\end{proof}

Let $(n,m)$ be the signature of $H_L$.
Then, the orthogonal group $O(H_L) \cong O(n,m;\R)$ acts on $P_{>}(H_L)$ by
$g\cdot p=g\circ p \circ g^{-1}$ for $p\in P_{>}(H_L)$ and $g \in O(H_L)$.
By the elementary calculus in linear algebra, we can show that the action is transitive.
The stabilizer subgroup of $p$ is $O(n;\R)\times O(m;\R)$.
Thus, $P_{>}(H_L)$ is homeomorphic to the Grassmanian $O(n,m;\R)/O(n;\R)\times O(m;\R)$.
Hence, we construct a family of full vertex operator algebras which is parametrized by
the Grassmanian $O(n,m;\R)/O(n;\R)\times O(m;\R)$.

Let $\tw$ be the unique even unimodular lattice with the signature $(1,1)$,
more explicitly, $\tw=\Z z\oplus \Z w$ and $(z,z)=(w,w)=0$ and $(z,w)=-1$
and for $N \in \Z_{>}$, set $\twN=\tw^N$.
%The full vertex operator algebra associated with 
Then, we have a family of full vertex operator algebras of the central charge $(N,N)$
parametrized by $O(N,N;\R)/O(N;\R)\times O(N;\R)$.
\begin{rem}
Let $\mathrm{Aut}\; \twN$ be the lattice automorphism group of $\twN$,
which is isomorphic to $O(N,N;\Z)$, the $\Z$-value points of the algebraic group $O(N,N)$.
Then, we can prove that if $\si\in \mathrm{Aut}\; \twN$
then $F_{\twN,H_{\twN},p}$ is isomorphic to $F_{\twN,H_{\twN},\si p \si^{-1}}$ as a full vertex operator algebra.
Thus, the true parameter space is 
$$O(N,N;\Z)  \backslash O(N,N;\R)/O(N;\R)\times O(N;\R).
$$
This statement is proved in \cite{M2} in more general setting.
We also remark that this conformal field theory and 
parameter space appear in the toroidal compactification of string theory \cite{P}.
\end{rem}

\subsection{Commutative algebra object in H-Vect}\label{sec_braided}
% commutative algebra object in braided monoidal category, H-Vect.
Let $H$ be a finite-dimensional real vector space with a non-degenerate symmetric bilinear form
$$(-,-):H\times H \rightarrow \R.$$
Let $\underline{\text{H-Vect}}$ be the category whose object is an $H$-graded vector space over $\C$
and morphism is a grading preserving $\C$-linear map.
Let $V=\bigoplus_{\al \in H} V^\al$ and $W=\bigoplus_{\be\in H} W^\be$ be $H$-graded vector spaces.
A tensor product of $H$-graded vector spaces are defined by the tensor product of vector spaces
with the $H$-grading:
\begin{align*}
(V\otimes W)^\al=\bigoplus_{\al' \in H} V^{\al'}\otimes W^{\al-\al'}.
\end{align*}
Then, $\HVect$ is a strict monoidal category whose unit is $\C=\C^0$.
Define a braiding $B_{V,W}: V\otimes W \rightarrow W\otimes V$
by $v_\al \otimes w_\be \mapsto \exp((\al,\be) \pi i) w_\be\otimes v_\al$
for $\al,\be \in H$, $v_\al \in V^\al$ and $w_\be \in W^\be$.
Then, it is easy to show that $B$ satisfies the hexagon identity,
thus $(\HVect, B)$ is a braided tensor category.
Let $A$ be a unital commutative associative algebra object in $\HVect$
and $a \in A^\al$ and $b \in A^\be$ for $\al,\be \in H$.
If $ab \neq 0$,
then, by the commutativity, $ab=\exp((\al,\be) \pi i)ba$ and $ba=\exp((\al,\be)\pi i)ab$.
Thus, $(\al,\be) \in \Z$ and $ab=(-1)^{(\al,\be)}ba$,
which implies that $A$ is an AH pair.
Thus, we have:
\begin{prop}
A unital commutative associative algebra object in $\HVect$ is an AH pair, and vice versa.
\end{prop}

\section{Appendix}
\subsection{Binary tree and expansions}\label{sec_parenthesized}
Let $F$ be a self-dual full vertex operator algebra and denote $Y(a,\ux)$ by $a(x)$ for $a\in F$.

For $n \in \Z_{>0}$, let $P_n$ be the set of parenthesized products of $n$ elements $1,2,\dots,n$, e.g., $(((31)6)(24))(57)$.
%We remark that $P_n$ is a $n$th component of the magma operad, thus the set of multi-linear elements in free magma (see \cite{}).

We can naturally associate a tree and a parenthesized correlation function (see Section \ref{sec_parenthesized})
 to each element in $P_n$. For example, an element $(((31)6)(24))(57) \in P_7$ defines the following binary tree:

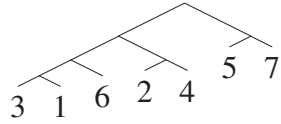
\Tree [. [ [ [ 3 1 ]. 6 ]. [ 2 4 ]. ]. [ 5 7 ]. ]

and the parenthesized correlation function:
\begin{align*}
S_{(((31)6)(24))(57)}(a_1,a_2,a_3,a_4,a_5,a_6,a_7):=
\Bigl(\1, \Biggl[\biggl( \Bigl(a_3(x_3)a_1\Bigr)(x_1)a_6\biggr)(x_6) a_2(x_2)a_4\Biggr](x_4) a_5(x_5)a_7\Bigr),
\end{align*}
where $a_1,\dots,a_7 \in F$.

We conjecture that $S_{(((31)6)(24))(57)}(a_1,a_2,a_3,a_4,a_5,a_6,a_7)$ is
the expansion of a real analytic function $\phi(z_1,\dots,z_7)$ on $X_7$ 
in the domain $$\{(x_1,\dots,x_7)\in \C^7\;|\; \{|x_4|>|x_6|>|x_1|>|x_3|\} \cap \{|x_6|>|x_2|\} \cap \{|x_4|>|x_5|\} \}$$
after the change of variables:
$$
(x_1,x_2,x_3,x_4,x_5,x_6)=(z_1-z_6,z_2-z_4,z_3-z_1,z_4-z_7,z_5-z_7,z_6-z_4).
$$
This conjecture is proved in the case where $n \leq 4$ or $n=5$ and the right most state is the vacuum state (Theorem \ref{consistency}).
Here, we give the rule for the change of variables and the convergence region by using trees.

\begin{wrapfigure}{r}{40mm}
\begin{minipage}{0.5\hsize}
%\begin{figure}
\begin{center}
\begin{bundle}{}
\chunk[j\;\;]{\begin{bundle}{}
\chunk[i]{\;\;\;\;}
\chunk[\;\;j]{\;\;\;\;}
\end{bundle}
}
\end{bundle}
\vspace{8mm}
Fig.1
\end{center}
%\caption
%\end{figure}
\end{minipage}
\begin{minipage}{0.23\hsize}
\begin{center}
\begin{bundle}{}
\chunk[j\;\;\;]{\begin{bundle}{{$z_i-z_j$}}
\chunk[i\;]{\;\;\;\;\;\;}
\chunk[\;\;j]{\;\;\;\;\;\;}
\end{bundle}
}
\end{bundle}
\end{center}
\begin{center}
Fig.2
\end{center}
\end{minipage}
\end{wrapfigure}

We remark that the tree associated with $A \in P_n$ is a full binary tree, i.e., every node has either $0$ or $2$ children.
A node with no children is called a leaf.
We will inductively assign a number to each edge of the tree
and a formal variable to a node of the tree.
First, the number $i$ is assigned to the unique edge connected to leaf labeled by $i$.
On each node, the number of the above edge is equal to the number of the right below edge,
which determine the numbers of the all edges (Fig.1).

Then, we assigned the formal variables $z_i-z_j$ on the node whose left below edge is labeled by $i$
and right below edge is labeled by $j$ and change the variable $x_i$ into $z_i-z_j$ (Fig.2).

For example, for $(((31)6)(24))(57) \in P_7$, we have:

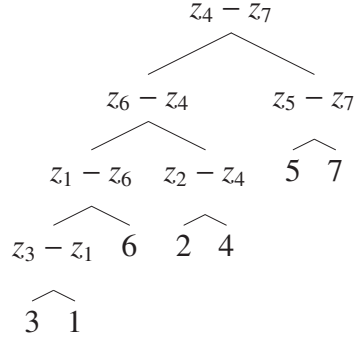
\Tree [.{{$z_4-z_7$}} [ [ [.{$z_3-z_1$} $3$ $1$ ].{$z_3-z_1$} 6 ].{$z_1-z_6$} [ 2 4 ].{$z_2-z_4$} ].{$z_6-z_4$} [ 5 7 ].{$z_5-z_7$} ]

Thus, the rule for the change of variables are determined.

We now explain how to get the convergence domain.
If a node labeled by $z_i-z_j$ is a child of $z_k-z_l$,
then $|z_k-z_l|>|z_i-z_j|$ or in the other coordinate $|x_k|>|x_i|$.
%For $A\in P_n$, denote by $X_n^A$ the open domain of $X_n$ defined by the above rule.
%In section \ref{sec_expansion} and \ref{sec_relation_expansion},
%we define and study the expansion of a function in $\Cor_4$ in the domain $X_4^A$ for each $A\in P_4$.
\begin{comment}
\begin{rem}???
In this remark, we give a physical reason of the change of variables.
By the state-field correspondence, each field $Y(a,z)$ corresponds to the state $\lim_{z \to 0}Y(a,z)\1=a \in F$.
Thus, each state in $F$ is placed at the origin $0 \in \C P^1$.
We recall that ,by the translation invariance of quantum field theory,
there is a operator $\rho(w)$ for $w\in \C$ such that $\rho(w)Y(a,z)\rho(w)^{-1}=Y(a,z+w)$.
Hence, in order to relate $1(23)$ to $(12)3$ in term of the vertex operator $Y$,
$a_1(z_1)a_2(z_2)=\rho(z_2) a_1(z_1-z_2)a_2(0) \rho(z_2)^{-1}
=\rho(z_2) Y(Y(a_1,z_1-z_2)a_2,0) \rho(z_2)^{-1} =Y(Y(a_1,z_1-z_2)a_2,z_2)$.
\end{rem}
\end{comment}

% explicit description and table of P_4 is given here.
% the space of formal power series is also given.

We also consider the special case that the right most state is the vacuum state.
Let $Q_n$ be the set of parenthesized products of $n+1$ elements $1,2,\dots,n,\star$
with $\star$ at the right most, e.g., $2(3(1(4\star))) \in Q_4$.
Then, the parenthesized correlation function associated with $A \in Q_n$ is defined by $S_A(a_1,\dots,a_n,\1)$,
that is, we insert the vacuum state at $\star$.
We denote it by $S_A(a_1,\dots,a_n)$ again.
For example, for $(((31)6)(24))(5\star) \in Q_6$ and $a_1,\dots,a_6\in F$,
we have
\begin{align*}
S_{(((31)6)(24))(5\star)}(a_1,a_2,a_3,a_4,a_5,a_6):=
\Biggl(\1, \Biggl[\biggl( \Bigl(a_3(x_3)a_1\Bigr)(x_1)a_6\biggr)(x_6) a_2(x_2)a_4\Biggr](x_4) a_5(x_5)\1\Biggr),
\end{align*}

The rule for the change of variable is given by the substitution of $z_\star=0$ into the rule for $P_{n+1}$.
%We also denote it by .
For $(((31)6)(24))(5\star)\in P_7$, the rule for the change of variables is:
$$(x_1,x_2,x_3,x_4,x_5,x_6)=(z_1-z_6,z_2-z_4,z_3-z_1,z_4-z_\star,z_5-z_\star,z_6-z_4).$$
Then, the rule for the change of variables for $(((31)6)(24))(5\star)\in Q_6$ is given by 
$$(x_1,x_2,x_3,x_4,x_5,x_6)=(z_1-z_6,z_2-z_4,z_3-z_1,z_4,z_5,z_6-z_4).$$

We list the trees and the rules for the change of variables for all the standard elements
in $P_4$:

%\begin{figure}[htbp]
\begin{minipage}[t]{0.33\hsize}
\begin{center}
{(12)(34):}
\end{center}

\Tree [.{{$z_2-z_4$}} [ 1 2 ].{{$z_1-z_2$}} [ 3 4 ].{{$z_3-z_4$}} ]
\end{minipage}
\begin{minipage}[t]{0.33\hsize}
\begin{center}
{((12)3)4:}
\end{center}

\Tree [.{$z_3-z_4$} [ [ 1 2 ].{$z_1-z_2$} 3 ].{$z_2-z_3$} 4 ]
\end{minipage}
\begin{minipage}[t]{0.33\hsize}
\begin{center}
{(1(23))4:}
\end{center}

\Tree [.{$z_3-z_4$} [ 1 [ 2 3 ].{$z_2-z_3$} ].{$z_1-z_3$} 4 ]
\end{minipage}

\vspace{5mm}

\begin{minipage}[t]{0.33\hsize}
\begin{center}
{1((23)4):}
\end{center}

\Tree [.{$z_1-z_4$} 1 [ [ 2 3 ].{$z_2-z_3$} 4 ].{$z_3-z_4$} ]
\end{minipage}
\begin{minipage}[t]{0.33\hsize}
\begin{center}
{1(2(34)):}
\end{center}

\Tree [.{$z_1-z_4$} 1 [ 2 [ 3 4 ].{$z_3-z_4$} ].{$z_2-z_4$} ]
\end{minipage}

\end{document}